\documentclass[11pt]{article}
\usepackage{bbm}
\usepackage{eqnarray,amsmath,amsfonts,amsthm,mathrsfs}
\usepackage{color}
\usepackage{bm}
\usepackage{amssymb}
\usepackage{graphicx}
\usepackage{epsfig}
\usepackage{epsf}
\usepackage{float}
\usepackage{subfigure}
\usepackage{amsfonts,amsmath,amsthm,amssymb,graphicx,float,fancyhdr,multirow,hyperref}
\usepackage{booktabs,longtable,authblk}
\usepackage{mathrsfs,hhline}

\usepackage{fancyhdr}
\usepackage[top=2.5cm, bottom=2.5cm, left=3cm, right=3cm]{geometry}
\usepackage{graphicx}
\newtheorem{theorem}{Theorem}
\newtheorem{assumption}{Assumption}

\newtheorem{definition}{Definition}

\newtheorem{lemma}{Lemma}

\newtheorem{remark}{Remark}
\newtheorem{example}{Example}

\allowdisplaybreaks[4]

\numberwithin{equation}{section}

\title{Solving PDEs on Spheres with Physics-Informed Convolutional Neural Networks$^\dag$\footnotetext{\dag~The work described in this paper is supported partially by Shanghai Science and Technology Program (Project No. 21JC1400600) and  NSFC/RGC Joint Research Scheme (Project No. 12061160462 and N\_CityU102/20). The work of Zhen Lei is also supported by the New Cornerstone Science Foundation through the XPLORER PRIZE and Sino-German Center Mobility Programme (Project No. M-0548). The work of Lei Shi is also supported by the National Natural Science Foundation of China (Grant No.12171039). The work of Ding-Xuan Zhou is partially supported by the Australian Research Council under project DP240101919. The corresponding author is Lei Shi.}}

\author[1]{Guanhang Lei}
\author[1]{Zhen Lei}
\author[1]{Lei Shi}

\author[1]{Chenyu Zeng}

\author[2]{Ding-Xuan Zhou}

\affil[1]{School of Mathematical Sciences, \linebreak
Shanghai Key Laboratory for Contemporary Applied Mathematics, \linebreak
Fudan University, Shanghai, 200433, P. R. China \linebreak
Email:ghlei21@m.fudan.edu.cn \linebreak
\{zlei, leishi, cyzeng19\}@fudan.edu.cn}

\affil[2]{School of Mathematics and Statistics, \linebreak
University of Sydney, \linebreak
Sydney NSW 2006, Australia \linebreak
Email:dingxuan.zhou@sydney.edu.au}

\date{}

\begin{document}
	\maketitle
\begin{abstract}
	Physics-informed neural networks (PINNs) have been demonstrated to be efficient in solving partial differential equations (PDEs) from a variety of experimental perspectives. Some recent studies have also proposed PINN algorithms for PDEs on surfaces, including spheres. However, theoretical understanding of the numerical performance of PINNs, especially PINNs on surfaces or manifolds, is still lacking. In this paper, we establish rigorous analysis of the physics-informed convolutional neural network (PICNN) for solving PDEs on the sphere. By using and improving the latest approximation results of deep convolutional neural networks and spherical harmonic analysis, we prove an upper bound for the approximation error with respect to the Sobolev norm. Subsequently, we integrate this with innovative localization complexity analysis to establish fast convergence rates for PICNN. Our theoretical results are also confirmed and supplemented by our experiments. In light of these findings, we explore potential strategies for circumventing the curse of dimensionality that arises when solving high-dimensional PDEs.
\end{abstract}
	

\section{Introduction}\label{Section: Introduction}

Solving partial differential equations (PDEs) is crucial in many science and engineering problems. Numerous methods, including finite differences, finite elements, and some meshless schemes, have been well-developed, particularly for low-dimensional PDEs. Nonetheless, these classical strategies often become impractical and time-consuming for high-dimensional PDEs attributed to their computational inefficiencies. The amalgamation of deep learning methodologies and data-driven architectures has recently exhibited superiority across various fields. There have been extensive studies on solving high-dimensional PDEs with deep neural networks, such as the Deep Ritz methods (DRMs) \cite{E2018Deep}, Physics-Informed Neural Networks (PINNs) \cite{Sirignano2018DGM, Raissi2019Physicsinformed}, Neural Operators \cite{Li2020Neural, Li2021Fourier, Kovachki2023Neural}, and DeepONets \cite{Lu2021Learning}. DRMs and PINNs utilize the powerful approximation capabilities of neural networks to directly learn solutions of PDEs. In contrast, Neural Operators and DeepONets specialize in learning the operators that map initial or boundary conditions to solution functions. One can refer to \cite{Karniadakis2021Physicsinformed} for an exhaustive review of deep learning techniques for resolving PDEs.

In this paper, we focus on the PINN approach. While previous research (e.g., \cite{Lu2022Machine, Jiao2022rate}) has performed convergence analysis on the generalization bound of PINNs, it primarily pertains to fully connected neural networks learning solutions to PDEs on a physical domain of a Euclidean space. However, there is still a gap in employing convolutional neural networks (CNNs) for this purpose. In addition, solving PDE systems on manifolds holds considerable relevance for practical applications, situating manifold learning as a vibrant domain within machine learning. The low intrinsic data dimensions have been validated as effective in shielding learning algorithms from the curse of dimensionality \cite{Ye2008Learning, Ye2009SVM, Fu2011Manifold, Yang2016Bayesian, Hamm2021Adaptive, Lei2024Pairwise}. Although there has been some exploration of the PINN framework for PDEs on manifolds \cite{Fang2020PhysicsInformed, Tang2021Physicsinformed, SahliCostabal2024Delta, Bastek2023PhysicsInformed, Zelig2023Numerical}, comprehensive convergence analysis is still absent. This study is the inaugural endeavor to address this gap. We propose the application of PINNs with CNN architectures, coined as PICNNs, to solve general PDEs of order $s$ on a unit sphere, while concurrently establishing convergence analysis for such PDE solvers. We probe into the approximation potential of CNNs by employing the spherical harmonic analysis, assess the Rademacher complexity of our model, and derive elegant generalization bounds through various technical estimates. In this context, our research offers significant insights into the ability of deep learning to tackle high-dimensional PDEs and leverage the low-dimensional attributes of physical domains. We summarize the contributions of this paper as follows.

\begin{itemize}
    \item We are the first to investigate PICNN PDE solvers on a unit sphere. Previous work (e.g., \cite{Fang2020PhysicsInformed, Tang2021Physicsinformed}) employs fully connected neural networks as solvers on the sphere. Although we focus on PINNs working with CNNs, some discussions in this paper are also applicable to the general PINNs. Consequently, we may use the term PINN when referring to the non-convolutional architecture.
    
    \item We present comprehensive analysis of the relationship between PDEs and PINNs, facilitating rigorous assumptions on the well-posedness and regularity of PDEs. In contrast to previous studies (e.g., \cite{Lu2022Machine,Jiao2022rate}) that exclusively focus on second-order elliptic PDEs, our approach encompasses a broad class of PDEs.
      
    \item We demonstrate fast convergence of approximation using $\mathrm{ReLU}$-$\mathrm{ReLU}^k$ CNNs operating on a unit sphere. Conversely, earlier studies \cite{Lu2022Machine, Jiao2022rate} solely explore the approximation ability of $\mathrm{ReLU}^3$ fully connected networks, which frequently suffer from gradient explosion during training. Despite the success of the ReLU network in diverse machine learning tasks attributed to its gradient calculation simplicity, it proves inadequate for approximating high-order derivatives in PDE problems due to the saturation phenomena (e.g., see \cite{Chen2020comparison}). To tackle this issue, we propose a hybrid CNN architecture utilizing both $\mathrm{ReLU}$ and $\mathrm{ReLU}^k$. Leveraging the advantages of each activation function, we achieve an effective approximation of the derivatives of PDE solutions while maintaining ease of network training. During the process of proof, we develop solid analysis employing ideas and techniques from spherical approximation theory and B-spline interpolation. These approximation analyses enable us to address a broad Sobolev smoothness condition, assuming the PDE solution $u^*$ belongs to $W^r_p(\mathbb{S}^{d-1})$.
    
    \item We develop a novel approach involving localization Rademacher complexity analysis to establish an oracle inequality that bounds the statistical error. To the best of our knowledge, we are the first to compute the VC-dimensions of hypothesis spaces generated by the CNN architecture and high-order derivatives. Furthermore, unlike \cite{Lu2022Machine}, our approach does not require a sup-norm restriction on network parameters, making it more aligned with practical algorithms. Additionally, our upper bound exhibits a significantly sharper estimation than that presented in \cite{Jiao2022rate} due to the absence of a localization technique in their work.
    
    \item We derive fast convergence upper bounds for the excess risk of PINNs in diverse scenarios, encompassing varying dimensions $d \geq 2$, PDE orders $s \geq 1$, smoothness orders $r \geq s+1$, and suitably chosen activation degrees $k \geq s$ of $\mathrm{ReLU}^k$. We establish a convergence rate in the form of $n^{-a}(\log n)^{2a}$, where 
    \[
        a = 
        \left\{
            \begin{aligned}
                &\frac{r-s}{(r-s) + (d-1)}, &&\mbox{ if }r < \infty, d > 3,\\
                &1 - \frac{d(k - s + 2) + r + k}{d(k - s + 2) + 2(r - s)(k - s + 1) + r + k}, &&\mbox{ if }r < \infty, 2 \leq d \leq 3,  \\
                &1 - \frac{1}{2(k - s) + 3}, &&\mbox{ if  }r = \infty, d \geq 2.
            \end{aligned}
        \right.
    \]
    
    \item We validate our theory through comprehensive numerical experiments, and by integrating our theoretical analysis with these experiments, we ascertain the conditions under which a PDE PINN solver can surmount the curse of dimensionality.
\end{itemize}

The rest of the paper is organized as follows. In \autoref{section: Preliminaries and Main Result}, we first discuss the strong convexity of the PINN risk, subsequently formulating reasonable assumptions on the well-posedness and regularity of PDEs. We then introduce the structure of the CNN employed in our analysis. After elaborating on the assumptions, we present our main result in \autoref{main theorem}. In \autoref{section: Approximation Error Analysis}, we prove an upper bound for the approximation error, as evidenced in \autoref{upper bound of the approximation error},  leveraging techniques from the spherical harmonics analysis and B-spline interpolation. In \autoref{section: Statistical Error Analysis}, novel localization analysis related to the stochastic component of error evaluation is conducted, leading to the derivation of a pivotal oracle inequality as seen in \autoref{oracle inequality}. \autoref{section: Convergence Analysis: Proof of main theorem} amalgamates the derived approximation bound and the oracle inequality, culminating in the derivation of accelerated convergence rates for PICNN when applied to solving sphere PDEs, which gives a proof of \autoref{main theorem}. We present experimental results in \autoref{Section: Experiments} to validate our theoretical assertions and shed light on the conditions circumventing the algorithmic curse of dimensionality.

\section{Preliminaries and Main Result}\label{section: Preliminaries and Main Result}
\subsection{PINN for General PDEs}\label{subsection: PINN for General PDEs}
Consider solving the following PDE with a Dirichlet boundary condition:
\begin{equation}\label{A general PDE}
    \left\{
        \begin{aligned}
            (\mathcal{L}u)(x) &= f(x),  &&x \in \Omega \subset \mathbb{R}^d, \\
            u(x) &= g(x),  &&x \in \partial\Omega,        
        \end{aligned}
    \right.
\end{equation}
where $\mathcal{L}$ is a general differential operator and $\Omega$ is a bounded domain. To construct an approximate solution, PINN converts \eqref{A general PDE} to a minimization problem with the objective function $\mathcal{R}$ defined as 
\[
    \mathcal{R}(u) = \frac{1}{\mathrm{vol}(\Omega)}\int_{\Omega} \vert (\mathcal{L}u)(x) - f(x)  \vert^2 dx + 
    \frac{1}{\sigma(\partial \Omega)}\int_{\partial \Omega} \vert u(x) - g(x)   \vert^2 d\sigma(x).
\]
Here, $\mathrm{vol}$ is the Lebesgue measure and $\sigma$ is the surface measure on $\partial \Omega$. This objective function is the mean squared error of the residual from \eqref{A general PDE} and the following empirical version is used for numerical optimization:
\[
    \mathcal{R}_{n, m}(u) = \frac{1}{n} \sum_{i=1}^n \vert (\mathcal{L}u)(X_i) - f(X_i) \vert^2  
    + \frac{1}{m} \sum_{i=1}^m \vert u(Y_i) - g(Y_i) \vert^2,
\]
where $\{X_i\}_{i=1}^{n} \subset \Omega, \{Y_i\}_{i=1}^{m} \subset \partial\Omega$ are independent and identically distributed (i.i.d.) random samples from the uniform distribution. $\mathcal{R}(u)$ and $\mathcal{R}_{n, m}(u)$ are referred to as the population risk and empirical risk respectively. Given a function class $\mathcal{F}$, PINN solves the following empirical risk minimization (ERM) problem:
\begin{equation}\label{ERM of general PINN} 
    u_{n, m} = \arg \min_{u \in \mathcal{F}}\mathcal{R}_{n, m}(u).
\end{equation}

If equation \eqref{A general PDE} has a unique classical solution, denoted by $u^*$, obviously $\mathcal{R}(u^*) = \mathcal{R}_{n, m}(u^*) = 0$. Under the framework of learning theory, this ideal solution $u^*$ is termed the Bayes function since it minimizes the population risk $\mathcal{R}(u)$. The performance of the ERM estimator $u_{n, m}$ can be quantified by the excess risk $\mathcal{R}(u_{n, m}) - \mathcal{R}(u^*)$. In our PINN model, given our assumption that a true solution $u^*$ exists, we have $\inf_{u}\mathcal{R}(u) = \mathcal{R}(u^*) = 0$. However, for the sake of conventional notation for excess risk, we will continue to express the excess risk as $\mathcal{R}(u_{n, m}) - \mathcal{R}(u^*)$ instead of $\mathcal{R}(u_{n, m})$. This study establishes rapidly decaying rates of $\mathcal{R}(u_{n, m}) - \mathcal{R}(u^*)$ as the size of the training data set increases. This decaying rate, often termed the convergence rate or learning rate, is an important measure of the algorithm's generalization performance. Contrastingly, the accuracy of an approximate PDE solution is traditionally measured by estimating the error $\|u_{n, m} - u^*\|$, where $\| \cdot \|$ represents the norm of a regularity space, such as the Sobolev space $W^r_p(\Omega)$. This raises a natural question: Can the excess risk $\mathcal{R}(u) - \mathcal{R}(u^*)$ bound $\|u - u^*\|$ in a manner that superior generalization performance corresponds to enhanced accuracy? This property is indispensable for constraining the statistical error through localization analysis. Equally significant is the converse question: Can $\|u - u^*\|$ control $\mathcal{R}(u) - \mathcal{R}(u^*)$? If so, an upper bound on $\|u - u^*\|$, derived from approximation analysis, could provide a bound for the approximation error $\mathcal{R}(u_\mathcal{F}) - \mathcal{R}(u^*)$ where $u_\mathcal{F}:= \arg \min_{u \in \mathcal{F}} \mathcal{R}(u)$. Our primary objective is to establish an equivalence between the excess risk $\mathcal{R}(u) - \mathcal{R}(u^*)$ and the error $\|u - u^*\|_{W^r_p(\Omega)}$. This relationship is known as the strong convexity of the PINN risk with respect to the $W^r_p(\Omega)$ norm, which is only determined by the underlying PDE.

For the sake of theoretical simplicity, we consider a linear PDE where the differential operator $\mathcal{L}$ is linear. Hence we write 
\begin{align*}
    \mathcal{R}(u) - \mathcal{R}(u^*) &= \frac{1}{\mathrm{vol}(\Omega)}\int_{\Omega} \vert (\mathcal{L}u)(x) - f(x)  \vert^2 dx + \frac{1}{\sigma(\partial \Omega)}\int_{\partial \Omega} \vert u(x) - g(x)\vert^2 d\sigma(x) \\
    &= \frac{1}{\mathrm{vol}(\Omega)}\int_{\Omega} \vert \mathcal{L}(u - u^*)(x)\vert^2 dx + \frac{1}{\sigma(\partial \Omega)}\int_{\partial \Omega} \vert u(x) - u^*(x) \vert^2 d\sigma(x) \\
    &= \frac{1}{\mathrm{vol}(\Omega)} \| \mathcal{L}(u - u^*) \|^2_{L^2(\Omega)} + \frac{1}{\sigma(\partial \Omega)}\| u - u^* \|^2_{L^2(\partial \Omega)}.
\end{align*}
Let us first consider controlling $\|u - u^*\|_{W^r_p(\Omega)}$ by $\mathcal{R}(u) - \mathcal{R}(u^*)$. This problem is intrinsically connected to the global $W^r_p$ estimates for PDEs. For an elliptic operator $\mathcal{L}$ of even order $2s$ where $s\in \mathbb{N}$, some prior estimates related to Gårding's inequality have already been well-established.  See \autoref{lem1} and \autoref{lem2} below. Throughout the subsequent discourse, if $A$ and $B$ are two quantities (typically non-negative), the notation $A \lesssim B$ indicates that $A \leq C B$ for some positive constant $C$. It should be emphasized that $C$ is independent of both $A$ and $B$, or is universal for indexed $A$ and $B$ of a certain class. The constant $C$ might depend on particular parameters, such as dimension, regularity, or exponent, as determined by the context. In this paper, we will not delve into the specifics of which parameters $C$ relies on, nor will we discuss the optimal estimation of $C$. Moreover, we use the notation $A \asymp B$ when $A \lesssim B$ and $B \lesssim A$. Denote the Sobolev space $W^s_2 (\cdot)$ by $H^s (\cdot)$ and its subspace consisting of functions vanishing on the boundary by $H^{s}_0(\Omega)$. 

\begin{lemma}\label{lem1}(Theorem 12.8 in \cite{Agmon1959Estimates})
    Let $s \in \mathbb{N}$ and $\mathcal{L}$ denote a uniformly elliptic operator of order $2s$ possessing bounded coefficients, with its leading coefficients being continuous. Furthermore, suppose $\mathcal{L}$ is weakly positive semi-definite. For a sufficiently large positive $\lambda$, the inequality 
    \[ 
        \|u\|_{H^{2s}(\Omega)} \lesssim  \|(\mathcal{L} + \lambda)u\|_{L^2(\Omega)} 
    \]
    holds for every \( u \in H^{2s}(\Omega) \cap H^{s}_0(\Omega) \).
\end{lemma}

\begin{lemma}\label{lem2}(Theorem 15.2 in \cite{Agmon1959Estimates} and remark therein)
    Given $s \in \mathbb{N}$ and $1 < p < \infty$, consider an operator $\mathcal{L}$ of even order $2s$ that is uniformly elliptic. Suppose the coefficients of $\mathcal{L}$ belong to $C(\bar{\Omega})$ and the boundary \(\partial \Omega\) is of class \(C^{2s}\). If $u$ is a solution to the equation \eqref{A general PDE} such that $\|u\|_{W^{2s}_p(\Omega)}$, $\|f\|_{L^p(\Omega)}$, and $\|g\|_{W^{2s - 1/p}_p(\partial \Omega)}$ are all finite, then it follows that 
    \[
        \|u\|_{W^{2s}_p(\Omega)} \lesssim \|f\|_{L^p(\Omega)} + \|g\|_{W^{2s - 1/p}_p(\partial \Omega)} + \|u\|_{L^p(\Omega)}.
    \] 
    Additionally, given the uniqueness of the solution to \eqref{A general PDE} with derivatives of order up to $2s$ in $L^p$, the term \(\|u\|_{L^p(\Omega)}\) can be excluded, resulting in
    \begin{equation}\label{PINN risk controls Sobolev norm}
        \|u\|_{W^{2s}_p(\Omega)} \lesssim \|f\|_{L^p(\Omega)} + \|g\|_{W^{2s - 1/p}_p(\partial \Omega)}.
    \end{equation}
\end{lemma}  

\begin{remark}
    Some similar conclusions can be generalized to elliptic PDEs on unbounded domains or Riemannian manifolds. Relevant discussions can be found in the referenced materials such as \cite{Agmon1959Estimates} (Chapter V) regarding unbounded domains, \cite{Taylor2011Partial} (Chapter 5), \cite{Warner1983Foundations}, and \cite{Lawson1990Spin} (Theorem 5.2. in Chapter III) that delve into Riemannian manifolds. Additionally, the $W^r_p$ estimates for other types of PDEs are also extensively studied. These estimates are integral to understanding the existence, uniqueness, and regularity of the solution.
\end{remark}

Returning to the discussion on PINN, let us consider a $2s$-order elliptic PDE that satisfies the conditions stated in inequality \eqref{PINN risk controls Sobolev norm}. By deriving the following inequality:
\[
    \|u - u^*\|_{H^{2s}(\Omega)} \lesssim \|\mathcal{L}(u - u^*)\|_{L^2(\Omega)} + \|u - u^*\|_{H^{2s-1/2}(\partial \Omega)},
\]
we observe that the norm of the boundary term is inconsistent with $\mathcal{R}(u) - \mathcal{R}(u^*)$. This can be a touchy issue and one may have to impose a hard-constraint for the boundary condition to carry on the analysis, that is, assuming that any potential solution $u$ from $\mathcal{F}$ exactly satisfies the boundary condition:
\begin{equation}\label{The hard-constraint}
    u(x) = g(x), \quad x \in \partial \Omega, u \in \mathcal{F}.
\end{equation}
Then 
\[
    \|u - u^*\|_{H^{2s}(\Omega)} \lesssim \|\mathcal{L}(u - u^*)\|_{L^2(\Omega)}, \quad \text{ for }u \in \mathcal{F} \cap H^{2s}(\Omega).
\]
Generally, \eqref{The hard-constraint} is not a reasonable assumption for a neural network space. \cite{Lu2022Machine} directly assume \eqref{The hard-constraint} for a general neural network function space (see Theorem B.12 therein), which lacks theoretical rigor. Notably, PDEs formulated on the whole sphere can address this problem without necessitating a hard-constraint, as the sphere is a boundaryless manifold. 

Now we consider controlling $\mathcal{R}(u) - \mathcal{R}(u^*)$ by $\|u - u^*\|_{H^s(\Omega)}$, which is much more direct, since
\begin{equation}\label{Sobolev norm controls PINN risk}
    \begin{aligned}
        \mathcal{R}(u) - \mathcal{R}(u^*) &\lesssim \| \mathcal{L}(u - u^*) \|^2_{L^2(\Omega)} + \| u - u^* \|^2_{L^2(\partial \Omega)} \\
        &\lesssim \|u - u^*\|^2_{H^s(\Omega)} + \|u - u^*\|^2_{L^2(\partial \Omega)},
    \end{aligned}
\end{equation}
where we assume that $\mathcal{L}$ is a linear differential operator of order $s$ with essentially bounded coefficient functions.

We conclude that the strong convexity of the PINN risk is a non-trivial property, especially concerning the $W^r_p$ error estimates of $u - u^*$. Another crucial consideration is the existence of a unique solution. If an exact solution does not exist or its existence remains unclear, one can still employ PINN to obtain an approximate solution through the minimization of residuals. Particularly, when the solution lacks uniqueness, PINN is still workable even though we might encounter challenges in ascertaining which solution PINN aims to approximate. Many references are available on the well-posedness of PDEs, and we will not go into further discussion here.

\subsection{PINN for PDEs on Spheres}
Prior to the discussion regarding PDEs on the sphere, one may refer to a brief introduction to the Laplace-Beltrami operator $\Delta_0$ and Sobolev spaces $W^r_p(\mathbb{S}^{d-1}), H^r(\mathbb{S}^{d-1})$ on spheres in \autoref{appsec: The Laplace-Beltrami operator and Sobolev spaces on spheres}. We now formulate the PINN algorithm on the sphere as follows. We focus on linear PDEs on $\mathbb{S}^{d-1}$ with $d \geq 2$, thus eliminating the necessity of imposing a boundary condition:
\begin{equation}\label{PDEs on sphere}
    (\mathcal{L}u)(x) = f(x), \quad x \in \mathbb{S}^{d-1}.
\end{equation}
Let $\sigma$ be the Lebesgue measure on $\mathbb{S}^{d-1}$ and $\omega_d$ be the surface area of $\mathbb{S}^{d-1}$. The population PINN risk is then defined as 
\[
    \mathcal{R}(u) = \frac{1}{\omega_{d}}\int_{\mathbb{S}^{d-1}} \vert (\mathcal{L}u)(x) - f(x)  \vert^2 d\sigma(x).
\]
After drawing $n$ i.i.d. random variables $\{X_i\}_{i=1}^n$ from $\mathbb{S}^{d-1}$ according to the uniform distribution, we define the empirical risk as
\[
    \mathcal{R}_n(u) = \frac{1}{n} \sum_{i=1}^n \vert (\mathcal{L}u)(X_i) - f(X_i) \vert^2.
\]
Consequently, the ERM estimator $u_n$, which belongs to a predefined function space $\mathcal{F}$, is determined by the following optimization problem:
\begin{equation}\label{ERM of PINN} 
    u_n = \arg \min_{u \in \mathcal{F}}\mathcal{R}_n(u).
\end{equation}
To enhance the algorithm's stability for practical applications, it is advantageous to consider the use of uniform lattices when generating training sample points, as they can offer a minimal fill distance. An example of such a lattice is the Fibonacci lattice, which provides near-uniform coverage on the $2$-D sphere (see \cite{Aistleitner2012Point}). However, generating almost uniform lattices on a general manifold can pose significant challenges. As an alternative, a uniform random sampling method may be employed. For the sake of simplicity in our analysis, we restrict our discussion to uniform random sampling on the sphere in this paper.

As previously discussed in \autoref{subsection: PINN for General PDEs}, we can identify approximate solutions within a function space without the hard constraint \eqref{The hard-constraint}, by the absence of a boundary on the sphere. It is also essential to assume that \eqref{PDEs on sphere} possesses a unique solution $u^*$. Furthermore, the strong convexity of the PINN risk is a requirement for both approximation and statistical analysis. For the approximation of $u^*$ using neural networks, it is imperative to make reasonable assumptions regarding the excess regularity of $u^*$. Take, for instance, $\mathcal{L}$, an elliptic operator of order $s$. By \eqref{Sobolev norm controls PINN risk}, our aim is to approximate $u^*$ in $H^s(\mathbb{S}^{d-1})$ norm, necessitating a higher order of smoothness of $u^*$ according to our approximation analysis. We should then presume $u^* \in H^r(\mathbb{S}^{d-1})$ for some $r \geq s+1$. Based on the elliptic regularity theorem (Theorem 6.30 in \cite{Warner1983Foundations}), we establish:
\[
    u^* \in H^{-\infty}(\mathbb{S}^{d-1}), f \in H^{t}(\mathbb{S}^{d-1}) \Longrightarrow u^* \in H^{t + s}(\mathbb{S}^{d-1}),
\]
where $H^{-\infty}(\mathbb{S}^{d-1})$ represents the dual space of $H^{\infty}(\mathbb{S}^{d-1})$.
Consequently, it is sufficient to assume $f \in H^{r-s}(\mathbb{S}^{d-1})$ in this context.

To employ concentration inequalities such as Bernstein's and Talagrand's inequalities in the generalization analysis, the introduction of a boundedness assumption becomes necessary. This requirement must be satisfied by both $u^*$ and the approximate function $u \in \mathcal{F}$ up to derivatives of order $s$. Our approximation analysis allows us to establish an upper bound for $\|u\|_{W^s_{\infty}(\mathbb{S}^{d-1})}$ only if $u^* \in W^r_{\infty}(\mathbb{S}^{d-1})$ where $r \geq s+1$. As a result, it is essential to assume, at the very least, that $u^* \in W^r_{\infty}(\mathbb{S}^{d-1})$ for some $r \geq s+1$. In conclusion, we state the following assumption.

\begin{assumption}\label{Assumption for PDE}
    Consider the sphere PDE \eqref{PDEs on sphere}. Assume that $\mathcal{L}$ is a linear differential operator of order $s \in \mathbb{N}$, taking the form:
    \begin{equation}\label{form of linear differential operator}
        \mathcal{L} = \sum_{\vert \alpha \vert \leq s}a_\alpha(x)D^\alpha,
    \end{equation}
    where $\alpha = (\alpha_1, \alpha_2, \cdots, \alpha_n)$ denotes a multi-index of non-negative integers, $\vert \alpha \vert = \alpha_1 + \alpha_2 + \cdots + \alpha_n$, $a_\alpha  \in L^\infty(\mathbb{S}^{d-1})$, and
    \[
        D^\alpha = \frac{\partial^{\vert \alpha \vert}}{\partial x^{\alpha_1}_1 \partial x^{\alpha_2}_2 \cdots \partial x^{\alpha_d}_d}.
    \] 
    Assume that \eqref{PDEs on sphere} has a unique solution $u^* \in W^r_{\infty}(\mathbb{S}^{d-1})$ for some $r \geq s + 1$.
    Furthermore, assume that the following $H^s(\mathbb{S}^{d-1})$ estimate holds for all $u \in H^s(\mathbb{S}^{d-1})$:
    \begin{equation}\label{H^s estimates}
        \|u - u^*\|_{H^s(\mathbb{S}^{d-1})} \lesssim \|\mathcal{L}(u - u^*)\|_{L^2(\mathbb{S}^{d-1})}.
    \end{equation}
\end{assumption}
\begin{remark}
    In \autoref{Assumption for PDE}, we suppose that the operator $\mathcal{L}$ is linear. The linearity is assumed primarily to expound upon the strong convexity associated with the PINN risk. However, this property is barely used in the convergence analysis. As such, it is feasible to consider a nonlinear operator, denoted as $\mathcal{T}$, defined as:
    \begin{equation*}\label{form of nonlinear differential operator}
        \mathcal{T} = \sum_{\vert \alpha \vert \leq s}a_\alpha D^\alpha \mbox{ such that } \mathcal{T}u = f, 
    \end{equation*}
    where coefficients $a_\alpha$ and the function $f$ might be dependent on the unknown function $u$ and its higher-order derivatives. Given that strong convexity is maintained with $\mathcal{T}$ as:
    \[ 
        \| \mathcal{T}u - \mathcal{T}u^* \|_{L^2(\mathbb{S}^{d-1})} \asymp \| u - u^* \|_{H^s(\mathbb{S}^{d-1})},
    \]
    and that $\mathcal{T}$ exhibits Lipschitz behavior for any $u_1, u_2 \in H^s(\mathbb{S}^{d-1})$ on a point-wise basis (refer to \eqref{Lipschitz of loss function}):
    \begin{align*}
        \big\vert (\mathcal{T}u_1)(x) - (\mathcal{T}u_2)(x)\big\vert \lesssim \sum_{\vert \alpha \vert \leq s} \vert D^\alpha u_1(x) - D^\alpha u_2(x) \vert,
    \end{align*}
    with coefficients $a_\alpha$ and $f$ possessing bounded sup-norms, our analysis presented in this paper remains largely applicable. Nonetheless, it warrants mention that verifying such conditions is extremely nontrivial for nonlinear operators. 
\end{remark}

Consequently, the assumption immediately indicates that $f \in L^{\infty}(\mathbb{S}^{d-1})$ and \eqref{Sobolev norm controls PINN risk} holds true. We admit that the assumption \eqref{H^s estimates} is non-trivial and has been deliberated under certain conditions in \autoref{subsection: PINN for General PDEs}, wherein \eqref{H^s estimates} could be met. Subsequently, we provide two specific examples in \autoref{appsec: Two PDE examples}.

\subsection{The CNN Architectures}\label{section: The CNN Architectures}
In this study, we specifically concentrate on the computation of approximate solutions through the ERM algorithm \eqref{ERM of PINN} within a designated space $\mathcal{F}$, which is generated by 1-D Convolutional Neural Networks (CNNs) induced by 1-D convolutions.  We present the following definition of the CNN architecture in the context of our analysis.

The CNN is specified by a series of convolution kernels, $\{w^{(l)}\}_{l=1}^{L}$, where each $w^{(l)}: \mathbb{Z} \to \mathbb{R}$ represents a vector indexed by $\mathbb{Z}$ and supported on $\{0, \ldots, S^{(l)}-1\}$, given a kernel size $S^{(l)} \geq 3$. We can iteratively define a 1-D deep CNN with $L$ hidden layers using the following expressions:
\begin{equation}\label{CNN layer}
    \begin{aligned}
        F^{(0)}:\mathbb{R}^d \to \mathbb{R}^d, \quad F^{(0)}(x) &= x; \\
        F^{(l)}:\mathbb{R}^d \to \mathbb{R}^{d_l},  \quad F^{(l)}(x) &= \sigma^{(l)}\left( \left( \sum_{j=1}^{d_{l-1}} w^{(l)}_{i-j} \big(F^{(l-1)}(x)\big)_j \right)_{i=1}^{d_l} - b^{(l)}\right), \\ 
        &l=1,\ldots,L.
    \end{aligned}
\end{equation}

In the above formulations, we denote the network widths as $d_0 = d, \{d_l = d_{l-1} + S^{(l)}-1\}_{l=1}^{L}$ and define the bias as $b^{(l)} \in \mathbb{R}^{d_l}$. The element-wise activation function, $\sigma^{(l)}: \mathbb{R} \to \mathbb{R}$, utilizes the Rectified Linear Unit (ReLU) function, $\sigma(x) = \max\{x,0\}$, which operates on each convolution layer. The convolution of a sequence $w^{(l)}$ on $F^{(l-1)}$ can be described through a convolutional matrix multiplication:
\[
    \left( \sum_{j=1}^{d_{l-1}} w^{(l)}_{i-j} \big(F^{(l-1)}(x)\big)_j \right)_{i=1}^{d_l}
    = T^{(l)} F^{(l-1)}(x),\] where $T^{(l)}$ represents a $d_l \times d_{l-1}$ matrix, defined as:
    \[
        T^{(l)} = 
        \begin{bmatrix}
            w^{(l)}_0  & 0 & 0 & 0 & \cdots & 0 \\
            w^{(l)}_1 & w^{(l)}_0  & 0 & 0 & \cdots & 0 \\
            \vdots & \ddots & \ddots &\ddots & \ddots & \vdots \\
            w^{(l)}_{S^{(l)}-1} & w^{(l)}_{S^{(l)}-2} & \cdots & w^{(l)}_{0} & 0\cdots & 0 \\
            0 & w^{(l)}_{S^{(l)}-1} & \cdots & w^{(l)}_{1} & w^{(l)}_{0} \cdots & 0 \\
            \vdots & \ddots & \ddots &\ddots & \ddots & \vdots \\
            \dots &  \cdots & 0 & w^{(l)}_{S^{(l)}-1} &\cdots & w^{(l)}_{0}\\
            \dots & \dots & 0 & 0 & w^{(l)}_{S^{(l)}-1}\cdots & w^{(l)}_{1} \\
            \vdots & \ddots & \ddots &\ddots & \ddots & \vdots \\
            0 & \cdots & \cdots & 0 & w^{(l)}_{S^{(l)}-1} & w^{(l)}_{S^{(l)}-2} \\ 
            0 & \cdots & \cdots & 0 & 0  & w^{(l)}_{S^{(l)}-1} \\
        \end{bmatrix}.
    \]
Upon the completion of $L$ convolution layers, a pooling operation is typically employed to decrease the output dimension. In this context, we consider a downsampling operator $\mathcal{D}: \mathbb{R}^{d_L} \to \mathbb{R}^{\left\lfloor d_L/d \right\rfloor}$ defined as $\mathcal{D}(x) = (x_{id})_{i=1}^{\left\lfloor d_L/d \right\rfloor }$. The convolution layers and pooling operator can be viewed as a feature extraction model. Finally, $L_0$ fully connected layers are implemented and an affine transformation computes the entire network's output $F^{(L+L_0+1)}(x) \in \mathbb{R}$ according to the following formulation:
\begin{equation}\label{FCNN layer}
    \begin{aligned}
        F^{(L+1)}&:\mathbb{R}^d \to \mathbb{R}^{d_{L+1}},\\
        &F^{(L+1)}(x) = \sigma^{(L+1)}\left( W^{(L+1)}\mathcal{D}\big(F^{(L)}(x)\big) - b^{(L+1)}\right);\\
        F^{(l)}&:\mathbb{R}^d \to \mathbb{R}^{d_l}, \\
        &F^{(l)}(x) = \sigma^{(l)}\left( W^{(l)}F^{(l-1)}(x) - b^{(l)}\right),\ l=L+2,\ldots,L+L_0; \\
        F^{(L+L_0+1)}&:\mathbb{R}^d \to \mathbb{R},\\
        &F^{(L+L_0+1)}(x) = W^{(L+L_0+1)} \cdot F^{(L+L_0)}(x) - b^{(L+L_0+1)}.
    \end{aligned}
\end{equation} 
Here, the terms $W^{(L+1)} \in \mathbb{R}^{d_{L+1} \times \left\lfloor d_{L}/d \right\rfloor}$, $W^{(l)} \in \mathbb{R}^{d_{l} \times d_{l-1}}$ for $l=L+2,\ldots,L+L_0$, and $W^{(L+L_0+1)} \in \mathbb{R}^{d_{L+L_0}}$ represent weight matrices. The elements $b^{(l)} \in \mathbb{R}^{d_{l}}$ for $l=L+1,\ldots,L+L_0$, and $b^{(L+L_0+1)} \in \mathbb{R}$ are biases.  As stated in \autoref{Section: Introduction}, it is not appropriate to use the ReLU function as the single activation function in our PICNN model. The composite function of some ReLU and affine functions becomes a piecewise linear function, which results in the network output's second derivative being strictly $0$. This outcome prevents approximating the true solution with a smoothness order of at least $s + 1 \geq 2$, a phenomenon known as saturation in approximation theory. To address this issue, we utilize alternative activation functions with non-linear second derivatives in place of the ReLU function in at least one fully connected layer. Suitable alternatives encompass $\mathrm{ReLU}^k$ \cite{Yang2024nonparametric, Yang2024optimal} for $k \geq s$, the logistic function, and smooth approximations of $\mathrm{ReLU}$ such as Softplus and $\mathrm{GeLU}$ \cite{Hendrycks2016Gaussian}:
\begin{align*}
    \mathrm{ReLU}^k(x) &= \max\{0,x\}^k; \\
    \mathrm{logistic}(x) &= \frac{1}{1 + e^{-x}}; \\
    \mathrm{Softplus}(x) &= \log(1+e^{x}); \\
    \mathrm{GeLU}(x) &= x\Phi(x) \approx 0.5x\big(1 + \tanh [\sqrt{2/\pi}(x+0.044715x^3)]\big).
\end{align*} 
Here, $\Phi$ is the cumulative distribution function of the standard normal distribution. In comparison to the depth $L$ of convolution layers, the depth $L_0$ of the Fully Connected Neural Network (FCNN) is typically much smaller. In this study, we focus on the $\mathrm{ReLU}^k$ case and set $L_0 = 2$. We emphasize here that we keep using $\mathrm{ReLU}$ in the convolution layers and only use $\mathrm{ReLU}^k$ in one fully connected layer. As we will show in the approximation error analysis (see \autoref{construction of convolution layers} below), the convolution layers are used for calculating the inner products of input $x$ and cubature samples $y_i \in \mathbb{S}^{d-1}$, which is a linear process and $\mathrm{ReLU}$ suffices. This aligns with the role of the convolutional layer in extracting features in practical applications. One fully connected layer of $\mathrm{ReLU}^k$ is sufficient to implement spline approximation (see \autoref{construction of FCNN} below) and address the phenomena of saturation. Additionally, minimizing the use of $\mathrm{ReLU}^k$ as the activation functions in network design helps prevent gradient explosion during training.

The function space $\mathcal{F}$ is defined by the output $F^{(L+L_0+1)}$. More specifically, $\mathcal{F}$ is characterized by network parameters $L, L_0, \{S^{(l)}\}_{l=1}^L, \{d_{l}\}_{l=L+1}^{L+L_0}$, activation functions $\{\sigma^{(l)}\}_{l=1}^{L+L_0}$, and supremum norm constraints $B_i$ for $i = 1 , \ldots , 4$:
\begin{equation}\label{sup-norm constraint for parameters}
    \begin{aligned}
        &\max_{l=1,\ldots,L}\|w^{(l)}\|_\infty \leq B_1, \\
        &\max_{l=1,\ldots,L}\|b^{(l)}\|_\infty \leq B_2, \\
        &\max_{l=L+1,\ldots,L+L_0+1}\|W^{(l)}\|_{\max} \leq B_3, \\
        &\max_{l=L+1,\ldots,L+L_0+1}\|b^{(l)}\|_\infty \leq B_4.
    \end{aligned} 
\end{equation} 
Additionally, the total number of free parameters $\mathcal{S}$ also defines the function space. When counting the total number of free parameters, we consider sparsity and parameters sharing in the network, which can lead to a sharper bound on $\mathcal{S}$, a lower network complexity, and ultimately a faster convergence rate. Another perspective involves restricting the supremum norm of the network output and its derivatives up to the order of $s$, which is vital for ensuring a bounded condition in the contraction inequality. Therefore, we choose $M>0$ to ensure
\begin{equation}\label{sup-norm constraint}
    \max_{\vert \alpha \vert \leq s}\|D^\alpha F^{(L+L_0+1)}\|_\infty \leq M \mbox{ when restricting $F^{(L+L_0+1)}$ on $\mathbb{S}^{d-1}$.}
\end{equation}

Precisely, we describe the function space $\mathcal{F}$ generated by 1-D CNN, which is given by
\begin{equation}\label{Assumption for hypothesis space}
    \begin{aligned}
        &\mathcal{F}(L, L_0, S, d_{L+1}, \ldots, d_{L+L_0}, \mathrm{ReLU}, \mathrm{ReLU}^{k}, M, \mathcal{S}) \\
        ={}& 
        \left\{ 
            F^{(L+L_0+1)}: \mathbb{S}^{d-1} \to \mathbb{R} 
            \left\vert
                \begin{array}{l}
                    \text{$F^{(L+L_0+1)}$ satisfies \eqref{sup-norm constraint}, which} \\
                    \text{is defined by \eqref{CNN layer} and \eqref{FCNN layer} with} \\
                    \text{$S^{(l)} = S, \sigma^{(l)} = \mathrm{ReLU}$ for $l \neq L+1$,} \\
                    \text{$\sigma^{(l)} = \mathrm{ReLU}^{k}$ for $l = L+1, k \geq s$; } \\
                    \text{the total number of free parameters} \\
                    \text{is less than or equal to $\mathcal{S}$.}
                \end{array}
            \right.
        \right\}.
    \end{aligned}
\end{equation}  
When contextually appropriate, we will continue to use the abbreviated notation, $\mathcal{F}$, to denote the aforementioned function space. By leveraging an innovative localization technique developed by a series of works \cite{Koltchinskii2000Rademacher, Koltchinskii2002Empirical, Bartlett2005Local, Koltchinskii2006Local, Koltchinskii2011Oracle, Zhou2024learning}, our function space associated with CNN does not mandate a sup-norm restriction \eqref{sup-norm constraint for parameters} on the network's trainable parameters, setting us apart from the previous analysis in \cite{Lu2022Machine}. However, for the sake of comparison, we will provide sup-norm bounds during our discussion of CNN's approximation capabilities. In addition to \autoref{Assumption for PDE}, we propose a further assumption, as outlined below.
\begin{assumption}\label{Assumption for CNN}
    The function space, denoted by 
    \[
        \mathcal{F}=\mathcal{F}(L, L_0, S, d_{L+1}, \ldots, d_{L+L_0}, \mathrm{ReLU}, \mathrm{ReLU}^{k}, M, \mathcal{S}),
    \]
    is defined as \eqref{Assumption for hypothesis space}. Herein, $M$ is chosen to be sufficiently large such that \eqref{sup-norm constraint} is satisfied and 
    \begin{equation}\label{sup-norm constraint for u^*}
        M \geq 3C_9\|u^*\|_{W^r_\infty(\mathbb{S}^{d-1})},
    \end{equation} where $u^* \in W^r_{\infty}(\mathbb{S}^{d-1})$ with $r \geq s + 1$ is the unique solution of equation \eqref{PDEs on sphere} and the constant $C_9$ only depends on $d, k$ and $s$(see \autoref{Sobolev approximation lemma} below).
\end{assumption}

\subsection{Main Result}

In this subsection, our primary focus lies in presenting our main result, where we establish fast convergence rates for PICNN in solving PDEs on spheres. Recall that, the function space denoted by 
\[
    \mathcal{F}(L, L_0, S, d_{L+1}, \ldots, d_{L+L_0}, \mathrm{ReLU}, \mathrm{ReLU}^{k}, M, \mathcal{S})
\] 
produced by 1D CNNs is provided in \eqref{Assumption for hypothesis space}. This space is parameterized by the number of convolutional layers ($L$), fully connected layers ($L_0$), the size of convolution kernels ($S$), the neuron count in fully connected layers ($\{d_{l}\}_{l=L+1}^{L+L_0}$), and the upper limits of output function and the overall count of free parameters ($M$ and $\mathcal{S}$ respectively). We highlight that the proposed CNN architecture \eqref{Assumption for hypothesis space} is supplemented with downsampling layers and utilizes both $\mathrm{ReLU}$ and $\mathrm{ReLU}^k$ activation functions. The former activation function is employed on the convolutional layers, while the latter operates on the fully connected layers. Our theoretical analysis and numerical experiments necessitate just $L_0=2$. This architecture has displayed impressive results in areas such as natural language processing, speech recognition, and biomedical data classification, as documented in \cite{Kiranyaz20211D} and the references therein. When contrasted with 2D CNNs that are only designed for 2D data like images and videos, 1D CNNs significantly curtail the computational load and are proven to be effective for handling data generated by low-cost applications, especially on portable devices. In this study, we apply this CNN architecture to solve the spherical PDE represented by \eqref{PDEs on sphere}. We will demonstrate that under a regularity condition, i.e., when the spherical PDE \eqref{PDEs on sphere} satisfies \autoref{Assumption for PDE}, both the excess risk and estimation error of PICNN estimators decay at polynomial rates. For two positive sequences, $\{A_n\}_{n\geq 1}$ and $\{B_n\}_{n\geq 1}$, recall that $A_n \lesssim B_n$ indicates there exists a positive constant $C$ independent of $n$, such that $A_n \leq C B_n, \forall \;n\geq 1$. Moreover, we write $A_n \asymp B_n$ if and only if both $A_n \lesssim B_n$ and $B_n \lesssim A_n$ hold true. Recall that $u^*$ denotes the solution of the equation \eqref{PDEs on sphere}. 

\begin{theorem}\label{main theorem}
    Suppose that \autoref{Assumption for PDE} is satisfied with some $d \geq 2, s \geq 1, r \geq s+1$ and \autoref{Assumption for CNN} holds. Let $d \geq 2$ and $\mathbf{x} = \{X_i\}_{i=1}^n$ be an i.i.d. sample following the uniform distribution on $\mathbb{S}^{d-1}$. Choose $3 \leq S \leq d+1$ and $k$ satisfying
    \[
        \left\{
            \begin{aligned}
                &k = s + \left\lceil \frac{r + s + 2}{d - 3} \right\rceil \geq s + 1, && \mbox{ if } r < \infty, d > 3,\\
                &k \geq s, && \mbox{ if } r < \infty, 2 \leq d \leq 3 \mbox{ or } r = \infty, d \geq 2.\\
            \end{aligned}
        \right.\] 
    Let $u_n$ be the estimator of PICNN solving sphere PDE \eqref{PDEs on sphere}, which is defined by \eqref{ERM of PINN} in the function space of CNN 
    \[
        \mathcal{F} = \mathcal{F}(L, L_0 = 2, S, d_{L+1}, d_{L+2}, \mathrm{ReLU}, \mathrm{ReLU}^{k}, M, \mathcal{S})
    \] 
    with 
    \begin{align*}
        L & \asymp n^{\frac{a(d-1)}{2(r - s)}}(\log n)^{-\frac{a(d-1)}{r - s}},\\
        d_{L+1} &\asymp  n^{\frac{a(d + r + s - 1)}{2(r - s)(k - s + 1)} + \frac{a(d+1)}{2(r - s)}} (\log n)^{- \frac{a(d + r + s - 1)}{(r - s)(k - s + 1)} - \frac{a(d+1)}{r - s}}, \\
        d_{L+2} &\asymp n^{\frac{a(d-1)}{2(r-s)}}(\log n)^{- \frac{a(d-1)}{r-s}}, \\
    \end{align*} 
    and
    \begin{align*}
        \mathcal{S} \asymp 
            \left\{
                \begin{aligned}
                    & n^{\frac{a(d-1)}{2(r-s)}}(\log n)^{-\frac{a(d-1)}{r-s}}, && \mbox{ if } r < \infty, d > 3,\\
                    & n^{\frac{a(d + r + s - 1)}{2(r-s)(k - s + 1)} + \frac{a}{r-s}} (\log n)^{- \frac{a(d + r + s - 1)}{(r-s)(k - s + 1)} - \frac{2a}{r-s}}, &&\mbox{ if } r < \infty, 2 \leq d \leq 3 \mbox{ or } r = \infty, d \geq 2.
                \end{aligned}
            \right.
    \end{align*} 
    Here, the constant $a$ is given by
    \[
        a = 
        \left\{
            \begin{aligned}
                &\frac{r-s}{(r-s) + (d-1)}, && \mbox { for }r < \infty, d > 3,\\
                &1 - \frac{d(k - s + 2) + r + k}{d(k - s + 2) + 2(r - s)(k - s + 1) + r + k}, && \mbox{ for }r < \infty, 2 \leq d \leq 3,  \\
                &1 - \frac{1}{2(k - s) + 3}, &&\mbox{ for }r = \infty, d \geq 2.
            \end{aligned}
        \right.
    \]
    Then for all $n \geq 1$, with probability at least $1 - \exp(-n^{1-a} (\log n)^{2a})$, there hold
    \[
        \mathcal{R}(u_n) - \mathcal{R}(u^*) \lesssim  n^{-a}(\log n )^{2a} 
    \]
    and
    \[
        \|u_n - u^*\|_{H^s(\mathbb{S}^{d-1})} \lesssim  n^{-a/2}(\log n )^{a}.
    \]
\end{theorem}

To the best of our knowledge, \autoref{main theorem} provides the first rigorous analysis of convergence for the PINN algorithm with a CNN architecture, thereby demonstrating its practical performance. Recent advancements in the theory of approximation and complexity for deep ReLU CNNs lay the foundation for our results. A significant contribution of our proof is the implementation of a scale-sensitive localization theory with scale-insensitive measures, such as VC-dimension. This approach is coupled with recent work on CNN approximation theory \cite{Zhou2020Universality, Zhou2020Theory, Fang2020Theory, Feng2023Generalization}, enabling us to derive these elegant bounds and fast rates. The idea of our localization is rooted in the work of \cite{Koltchinskii2000Rademacher, Koltchinskii2002Empirical, Bartlett2005Local, Zhou2024learning}, which allows us to examine broader classes of neural networks. This contrasts with analyses that limit the networks to have bounded parameters for each unit, a constraint resulting from applying scale-sensitive measures like metric entropy. In our proof, we have broadened previous approximation results and established an estimation of VC-dimensions involving the function spaces generated by CNNs and their derivatives. These results hold significance in their own right and warrant special attention (refer to \autoref{section: Approximation Error Analysis} and \autoref{section: Statistical Error Analysis}). To our knowledge, our work is the first to establish corresponding localization analysis for estimating the part of the stochastic error in the theoretical analysis framework of solving equations with neural networks. Our results can be generalized to theoretical analysis of other PDE solvers with CNNs. The sphere is commonly viewed as the quintessential low-dimension manifold; hence, our work has the potential to be extrapolated to PDE solvers operating on general manifolds. Exploring further possibilities by extending this approach to other neural network architectures and different types of PDEs can certainly be advantageous.  We posit, with firm conviction, that for an $s$-order PDE PINN solver operating on an arbitrary $m$-dimensional manifold $\mathcal{M}$ embedded in $\mathbb{R}^d$, and assuming the true solution $u^* \in C^r(\mathcal{M})$ (or within another regularity function space, such as Sobolev space), it is plausible to demonstrate a convergence rate of $n^{-\frac{r-s}{r-s + m}}$ (modulated by a logarithmic factor). This intriguing prospect presents promising avenues for future research.

\subsection{Error Decomposition}\label{subsection: Error Decomposition}

This subsection primarily establishes the proof structure for \autoref{main theorem}, which is derived from an error decomposition of the excess risk that we introduce below. Error decomposition is a standard paradigm for deducing the generalization bounds of the ERM algorithm. Research on algorithmic generalization bounds is a fundamental concern in learning theory. Specifically, it involves providing a theoretical non-asymptotic bound for the excess risk $\mathcal{R}(u_n) - \inf_{u}\mathcal{R}(u)$ with respect the number of training samples $n$. We now apply a standard error decomposition procedure to the excess risk of PINN. If we define $u_\mathcal{F} = \arg \min_{u \in \mathcal{F}} \mathcal{R}(u)$, we obtain the following decomposition:
\begin{align*}
    \mathcal{R}(u_n) - \mathcal{R}(u^*) \leq& \big(\mathcal{R}(u_\mathcal{F}) - \mathcal{R}(u^*)\big) + \big(\mathcal{R}_n(u_\mathcal{F}) - \mathcal{R}_n(u^*) - \mathcal{R}(u_\mathcal{F}) + \mathcal{R}(u^*)\big) \\
    &+ \big(\mathcal{R}(u_n) - \mathcal{R}(u^*) - \mathcal{R}_n(u_n) + \mathcal{R}_n(u^*)\big).
\end{align*} 
The first term is often referred to as the approximation error, while the third term is known as the estimation error or statistical error. The second term can be bounded by Bernstein's inequality. See \autoref{appsec: Bernstein's bound}. Therefore, it remains to derive upper bounds for the approximation error and the statistical error respectively. In \autoref{section: Approximation Error Analysis}, we will establish an estimation of the approximation error, and in \autoref{section: Statistical Error Analysis}, we will develop an estimation of the statistical error. By combining these two, we can derive a generalization bound for the excess risk $\mathcal{R}(u_n) - \mathcal{R}(u^*)$, and hence, provide the proof for \autoref{main theorem}. A comprehensive proof of \autoref{main theorem} can be found in \autoref{section: Convergence Analysis: Proof of main theorem}.
    
\section{Approximation Error Analysis}\label{section: Approximation Error Analysis}

This section focuses on the estimation of the approximation error. We provide a proof sketch here and see \autoref{appsec: Supplement for approximation error analysis} for details. Recent progress in the approximation theory of deep ReLU CNN \cite{Zhou2020Universality, Zhou2020Theory, Fang2020Theory, Feng2023Generalization} and spherical harmonic analysis \cite{Dai2013Approximation} are important building blocks for our estimates. As referenced in \autoref{Assumption for PDE}, it is assumed that $u^* \in W^{r}_{\infty}(\mathbb{S}^{d-1})$ for some $r \geq s+1$. To establish the generalization error analysis for spherical data classification, \cite{Feng2023Generalization} has obtained an $L^p(\mathbb{S}^{d-1})$ norm approximation rate for functions in $W^r_p (\mathbb{S}^{d-1})$ using CNNs. In our PINN setup, the $L^p(\mathbb{S}^{d-1})$ approximation is insufficient as we require an approximation rate of the $H^s(\mathbb{S}^{d-1})$ norm, as outlined by \eqref{Sobolev norm controls PINN risk}. Moreover, their work only considers a ReLU network, which is suitable for an $L^p(\mathbb{S}^{d-1})$ norm approximation but not applicable here due to the phenomenon of saturation. Hence, we employ and extend their results to the Sobolev norm approximation and our $\mathrm{ReLU}$-$\mathrm{ReLU}^k$ network architectures. The approximator \cite{Feng2023Generalization} is known as the near-best approximation by polynomials. Let $u \in L^p(\mathbb{S}^{d-1})$ where $1 \leq p \leq \infty$. For $n_0 \in \mathbb{N}$, we define the error of the best approximation to $u$ by polynomials of degree at most $n_0$ as 
\[
    E_{n_0}(u)_p = \inf_{v \in \varPi_{n_0}(\mathbb{S}^{d-1})} \|u-v\|_p,
\] 
where $\varPi_{n_0}(\mathbb{S}^{d-1})$ denotes the space of polynomials of degree at most $n_0$. Let $\lambda = \frac{d-2}{2}$, and let $C^\lambda_i(t)$ represent the Gegenbauer polynomial of degree $i$ with parameter $\lambda$. Given a function $\eta \in C^\infty([0,\infty))$ where $\eta(t) = 1$ for $0 \leq t \leq 1$, $0 \leq \eta(t) \leq 1$ for $1 \leq t \leq 2$, and $\eta(t) = 0$ for $ t \geq 2$, we define the kernel $l_{n_0}$ by 
\[
    l_{n_0}(t) = \sum_{i=0}^{2n_0} \eta\bigg(\frac{i}{n_0}\bigg) \frac{\lambda + i}{\lambda}C^\lambda_i(t), \quad t \in [-1,1],
\] 
and the corresponding linear operator 
\[
    L_{n_0}(u)(x) = \frac{1}{\omega_d}\int_{\mathbb{S}^{d-1}} u(y)l_{n_0}(\langle x,y \rangle)d\sigma(y), \quad x \in \mathbb{S}^{d-1}.
\] 
When $d = 2$, $l_{n_0}$ and $L_{n_0}$ are well-defined in the sense of limit, by using the relation 
\[
    \lim_{\lambda \to 0}\frac{1}{\lambda} C^{\lambda}_i(\cos \theta) = \frac{2}{i}\cos i\theta.
\]
This linear operator provides a near-best approximation in the $L^p$ norm. Notice that the construction of the function $\eta$ is not unique, but we can fix a specific $\eta$ beforehand in the following statements, i.e., we can construct a uniform $\eta$ applied to derive all the estimates. 

However, $L_{n_0}(u)$ is not the approximator we ultimately use to estimate the error. We introduce a new kernel, $\widetilde{l}_{n_0}$, and a linear operator, $\widetilde{L}_{n_0}$, defined as follows: 
\begin{equation}\label{integral of linear operator}
    \begin{split}
        \widetilde{l}_{n_0}(t) &= \sum_{i=0}^{2n_0} \bigg[\eta\bigg(\frac{i}{n_0}\bigg) \bigg]^2 
        \frac{\lambda + i}{\lambda}C^\lambda_i(t), \quad t \in [-1,1], \\
        \widetilde{L}_{n_0}(u)(x) &= \frac{1}{\omega_d}\int_{\mathbb{S}^{d-1}} u(y)\widetilde{l}_{n_0}(\langle x,y \rangle)d\sigma(y), 
        \quad x \in \mathbb{S}^{d-1}.
    \end{split}
\end{equation}
Again, when $d=2$, $\widetilde{l}_{n_0}$ and $\widetilde{L}_{n_0}$ are well-defined in the sense of limit $\lambda \to 0$. It should be noted that the function $\eta^2$ still meets all the conditions required above for the function $\eta$. Additionally, $\widetilde{L}_{n_0}(u)$ exhibits near-best approximation properties. We select $\widetilde{L}_{n_0}(u)$ because the integral \eqref{integral of linear operator} can be discretized using a cubature formula, thereby expressing it as an additive ridge function. According to Lemma 3 in \cite{Feng2023Generalization}, we have the following lemma.
\begin{lemma}\label{cubature formula}
    Let $u \in L^p(\mathbb{S}^{d-1})$ for $1 \leq p \leq \infty$. Then there exists a constant $C_2$ only depending on $d$ such that, for all $m \geq C_2n_0^{d-1}$, there exists a cubature rule $\{(\mu_i,y_i)\}_{i=1}^m$ of degree $4n_0$ with $\mu_i \in \mathbb{R}$, $y_i \in \mathbb{S}^{d-1}$
    and we have 
    \[
        \widetilde{L}_{n_0}(u)(x) = \sum_{i=1}^m \mu_i L_{n_0}(u)(y_i) l_{n_0}(\langle x,y_i \rangle).
    \]
\end{lemma}
As a direct consequence, Lemma 3 in \cite{Feng2023Generalization} asserts that for all $r \in \mathbb{N}, 1 \leq p \leq \infty$ and $u \in W^r_p (\mathbb{S}^{d-1})$, one can obtain 
\[
    \bigg\|u - \sum_{i=1}^m \mu_i L_{n_0}(u)(y_i) l_{n_0}(\langle x,y_i \rangle)\bigg\|_p \leq C_3n_0^{-r}\|u\|_{W^r_p (\mathbb{S}^{d-1})}
\]
with a constant $C_3$ only depends on $\eta$ and $d$. This approximation bound can be extended to those with respect to the $W^s_p (\mathbb{S}^{d-1})$ norm with $s \leq r - 1$, as illustrated by Chapter 4 of \cite{Dai2013Approximation}.
\begin{lemma}\label{ridge function approximation}
    If $u \in W^r_p(\mathbb{S}^{d-1})$ for $1 \leq p \leq \infty$ and $r \in \mathbb{N}$, then there exists a constant $C_4$ only depending on $\eta$ and $d$, such that
    \[
        \begin{aligned}
            \bigg\|u - \sum_{i=1}^m \mu_i L_{n_0}(u)(y_i) l_{n_0}(\langle x,y_i \rangle)\bigg\|_{W^s_p(\mathbb{S}^{d-1})} &\leq C_4n_0^{s-r}\|u\|_{W^r_p(\mathbb{S}^{d-1})}, \\
            \bigg\|\sum_{i=1}^m \mu_i L_{n_0}(u)(y_i) l_{n_0}(\langle x,y_i \rangle)\bigg\|_{W^{r}_p(\mathbb{S}^{d-1})} &\leq C_4\|u\|_{W^r_p(\mathbb{S}^{d-1})}.
        \end{aligned}
    \]
\end{lemma}
\begin{remark}\label{Remark for fractional Sobolev space}
    According to Theorem 4.8.1 in \cite{Dai2013Approximation}, the approximation bound is also applicable to the Lipschitz space $W^{r, \alpha}_p(\mathbb{S}^{d-1})$ for $r \in \mathbb{N}$ and $\alpha \in {[0, 1)}$, as indicated by the following: 
    \[
        E_{n_0}(u)_p \leq C_4n_0^{-r-\alpha}\|u\|_{W^{r, \alpha}_p(\mathbb{S}^{d-1})}.
    \] 
    Thus, we can further extend the conclusion of \autoref{ridge function approximation} to encompass fractional-type Sobolev solutions $u \in W^r_p (\mathbb{S}^{d-1})$ where $r \geq 1$ is a real number.
\end{remark}

Given the cubature rule $\{(\mu_i,y_i)\}_{i=1}^m$ as claimed in \autoref{cubature formula}, convolution layers equipped with a downsampling operator can extract the inner product $\langle x,y_i \rangle$ for all $i=1,\ldots,m$ as representative features. The construction of the convolution layers and the total number of free parameters are given in the proof of Lemma 3 in \cite{Fang2020Theory}. We prove the sup-norm bounds on trainable parameters in \autoref{appsec: Proof of the sup-norm bounds in construction of convolution layers}. Denote the constant $1$ vector in $\mathbb{R}^s$ as ${\bf 1}_s$.

\begin{lemma}\label{construction of convolution layers}
    For $m \in \mathbb{N}$ and a set of cubature samples $y = \{y_1,\ldots,y_m\} \subset \mathbb{S}^{d-1}$, there exists a sequence of convolution kernels $\{w^{(l)}\}_{l=1}^L$ supported on $\{0,\ldots,S^{(l)}-1\}$ of equal size $S^{(l)} = S$ and depth $L \leq \left\lceil \frac{md-1}{S-2} \right\rceil$. Additionally, bias vectors $b^{(l)}$ exist for $l=1,\ldots,L$ such that
    \[
        \mathcal{D}(F^{(L)}(x)) = 
        \begin{bmatrix}
            \langle x , y_1 \rangle \\
            \vdots \\
            \langle x , y_m \rangle \\
            0 \\
            \vdots \\
            0
        \end{bmatrix} 
        + B^{(L)} {\bf 1}_{\left\lfloor \frac{d+L(S-1)}{d}\right\rfloor},
    \]
    where $B^{(L)} = \prod_{l=1}^{L} \| w^{(l)} \|_1$. Further, we have the following sup-norm bounds:
    \begin{align*}
        \|w^{(l)}\|_\infty &\leq C_5, \quad l = 1, \ldots , L , \\
        \|b^{(l)}\|_\infty &\leq  C_6^l, \quad l = 1, \ldots , L , \\
    \end{align*} 
    Here, the constants $C_5$ and $C_6$ only depend on $S$. The total number of free parameters contributed by $\{w^{(l)}\}_{l=1}^L$ and $\{b^{(l)}\}_{l=1}^L$ is $3LS - L$.
\end{lemma}

The additive ridge function $\sum_{i=1}^m \mu_i L_{n_0}(u)(y_i) l_{n_0}(\langle x,y_i \rangle)$ is employed to approximate the function $u \in W^r_p (\mathbb{S}^{d-1})$. This approximation requires an intermediary approximation of the univariate function $l_{n_0} : [-1,1] \to \mathbb{R}$. For this purpose, the interpolant of B-spline functions is utilized for the ReLU$^k$ FCNN.

\begin{lemma}\label{construction of FCNN}
    Given the output $\mathcal{D}(F^{(L)}(x))$ in \autoref{construction of convolution layers} from the convolution layers, for $N \in \mathbb{N}$, there exists a ReLU$^k$ FCNN with two hidden layers of width $d_{L + 1} = m(k^3 + 4k^2 + 4Nk + k + 8N)/2$ and $d_{L + 2} = m$ such that the network output
    \begin{equation}\label{FCNN output}
        F^{(L+3)}(x) = \sum_{i=1}^m \mu_i L_{n_0}(u)(y_i) Q_N^{k+1}(l_{n_0})(\langle x,y_i \rangle),
    \end{equation}
    where $Q_N^{k+1}(l_{n_0})$ is an interpolation spline satisfying 
    \begin{equation}\label{interpolation spline error}
        \|D^l(l_{n_0} - Q^{k+1}_N(l_{n_0}))\|_\infty \leq \frac{C_8n_0^{d+2k+1}}{N^{k-l+1}}, \quad l =0, \ldots, k,
    \end{equation}
    where $D^l$ represents the $l$-th derivative, constant $C_8$ only depends on $k$ and $d$.
    Moreover, this FCNN satisfies the following boundedness constraints:
    \[
        \begin{aligned}
            &\|W^{(L+1)}\|_{\max} \lesssim N, && \|b^{(L+1)}\|_\infty \lesssim NC_6^L, \\
            &\|W^{(L+2)}\|_{\max} \leq 5 \cdot 3^{d-2} \cdot (2k+2)^{k+1} n_0^{d-1}, && \|b^{(L+2)}\|_\infty \leq 5 \cdot 3^{d-2} \cdot (2k+2)^{k+1} n_0^{d-1}, \\
            &\|W^{(L+3)}\|_{\max} \leq C_1\|u\|_\infty, && \vert b^{(L+3)} \vert \leq 5 \cdot 3^{d-2} \cdot (2k+2)^{k+1} n_0^{d-1}C_1\|u\|_\infty,
        \end{aligned}
    \]
    where $C_1$ only depends on $\eta$ and $d$, and $C_6$ is from \autoref{construction of convolution layers}.
    The total number of free parameters of this FCNN is $\displaystyle \frac{3}{2} \cdot (k^3 + 4k^2 + 4Nk + k + 8N) + m + 2.$
\end{lemma}

At the end of this section, we can establish an upper bound for the approximation error $\mathcal{R}(u_\mathcal{F}) - \mathcal{R}(u^*)$ utilizing \eqref{Sobolev norm controls PINN risk}. Recall that $u^*$ denotes the solution of the equation \eqref{PDEs on sphere}.

\begin{theorem}\label{upper bound of the approximation error}
    Assume that \autoref{Assumption for PDE} is satisfied with some $d \geq 2, s \geq 1$ and $r \geq s+1$. For any $0 < \varepsilon < \|u^*\|^2_{W^r_\infty(\mathbb{S}^{d-1})}$, let $\delta = \varepsilon / \|u^*\|^2_{W^r_\infty(\mathbb{S}^{d-1})}$ denote the relative error. If we take the function space of CNN
    \[
        \mathcal{F} = \mathcal{F}(L, L_0 = 2, S, d_{L+1}, d_{L+2}, \mathrm{ReLU}, \mathrm{ReLU}^{k}, M, \mathcal{S})
    \] 
    with 
    \begin{align*}
        &3 \leq S \leq d+1, \quad k \geq s, \quad L \asymp \delta^{-\frac{d-1}{2(r - s)}}, \\
        &d_{L+1} \asymp \delta^{- \frac{d + r + s - 1}{2(r - s)(k - s + 1)} - \frac{d+1}{2(r - s)}}, \quad d_{L+2} \asymp \delta^{-\frac{d-1}{2(r-s)}}, \\
        &M \geq 3C_9\|u^*\|_{W^r_\infty(\mathbb{S}^{d-1})},\\
        &\mathcal{S} \asymp \max\left\{\delta^{- \frac{d + r + s - 1}{2(r-s)(k - s + 1)} - \frac{1}{r-s}}, \delta^{ - \frac{d-1}{2(r-s)}}\right\},
    \end{align*}
    where constant $C_9$ is from \autoref{Sobolev approximation lemma} and only depends on $d, k$ and $s$, then 
    \begin{equation}\label{eq: upper bound of the approximation error}
        \mathcal{R}(u_\mathcal{F}) - \mathcal{R}(u^*) \lesssim \varepsilon,
    \end{equation}
    where $u_\mathcal{F} = \arg \min_{u \in \mathcal{F}} \mathcal{R}(u)$.
\end{theorem}
\begin{remark}
    As previously emphasized in \autoref{Remark for fractional Sobolev space}, the bound above is also applicable for $r \in \mathbb{R}$ in the context of fractional Sobolev spaces.
\end{remark}

\section{Statistical Error Analysis}\label{section: Statistical Error Analysis}
This section is dedicated to the estimation of statistical error by establishing an important oracle inequality. This can be done by a novel localization analysis, an estimation of the VC-dimension bounds of our CNN space and a peeling technique. We give key results here and refer the readers to \autoref{appsec: Supplement for statistical error analysis} for details. The localization analysis is developed for the stochastic part of error analysis and presented in \autoref{appsec: Localized complexity analysis}. The VC-dimension bounds of our CNN space and its high-order derivative spaces will turn out to be an upper bound control on the Rademacher complexity of the network. The study by \cite{Harvey2017Nearlytight} substantiates near-precise VC-dimension bounds for piecewise linear neural networks, laying emphasis on feedforward neural networks with ReLU activation functions. Herein, we offer a novel evaluation on the VC-dimension bound for our specific CNN space, which incorporates derivatives and generalizes a bound for CNNs in \cite{Zhou2024learning}. The subsequent theorem, forming a cardinal contribution of this paper, warrants individual consideration owing to its central role in the analysis of stochastic error. Herein, we assume that \autoref{Assumption for PDE} and \autoref{Assumption for CNN} hold with some $s \geq 1, r\geq s+1$ and $M > 0$. For ease of notation, we define the set $\{1,2,\ldots,m\}$, for $m \in \mathbb{N}$, as $[m]$. We also use the notation $\displaystyle D^\alpha \mathcal{F} := \{D^\alpha u:\mathbb{S}^{d-1} \to \mathbb{R} \ \vert \ u \in \mathcal{F}\}.$
\begin{definition}
    Given a function class $\mathcal{G}: \mathcal{X} \to \mathbb{R}$ and a set $X = \{x_i\}_{i=1}^m$ of $m$ points in the input space $\mathcal{X}$. Let $\mathrm{sgn}(\mathcal{G}) = \{\mathrm{sgn}(g): g \in \mathcal{G}\}$ be the set of binary functions $\mathcal{X} \to \{0, 1\}$ induced by $\mathcal{G}$. If $\mathrm{sgn}(\mathcal{G})$ can compute all dichotomies of $X$, that is, $\# \{\mathrm{sgn}(g)\vert_X \in \{0, 1\}^m : g \in \mathcal{G}\} = 2^m,$ 
    we say that $\mathcal{G}$ or $\mathrm{sgn}(\mathcal{G})$ shatters $X$. The Vapnik-Chervonenkis dimension (or VC-dimension) of $\mathcal{G}$ or $\mathrm{sgn}(\mathcal{G})$ is the size of the largest shattered subset of $\mathcal{X}$, denoted by $\mathrm{VCDim}(\mathrm{sgn}(\mathcal{G}))$.
    Moreover, we say that $X$ is pseudo-shattered by $\mathcal{G}$, if there are real numbers $r_1, r_2, \ldots, r_m$, such that for each $v \in \{0, 1\}^m$, there exists a function $g_v \in \mathcal{G}$ with $\mathrm{sgn}(g_v(x_i) - r_i) = v_i$ for $1 \leq i \leq m$. The pseudo-dimension of $\mathcal{G}$, denoted by $\mathrm{PDim}(\mathcal{G})$, is the maximum cardinality of a subset $X$ of $\mathcal{X}$ that is pseudo-shattered by $\mathcal{G}$.
\end{definition}
\begin{theorem}\label{VCDim bound}
    Consider the CNN function space $\mathcal{F}$ defined by \eqref{Assumption for hypothesis space} with hybrid activation functions, i.e., $\sigma^{(l)} = \mathrm{ReLU}$ for $l \neq L+1, \sigma^{(L+1)} = \mathrm{ReLU}^{k}$ with $k \geq s$, which is parameterized by $L, L_0, S^{(l)} \equiv S$ and the total number of free parameters $\mathcal{S}$. Then for all $\vert \alpha \vert \leq s$, it follows that 
    \[
        \mathrm{VCDim}(\mathrm{sgn}(D^\alpha \mathcal{F})) \lesssim \mathcal{S}(L + L_0)\log \bigg(k(L + L_0) \cdot \max\{d_{l}: l \in [L+L_0]\} \bigg).
    \]
\end{theorem} 

Relying on the VC dimension estimates given in \autoref{VCDim bound}, we can utilize a peeling technique to demonstrate an important oracle inequality, which presents an upper bound for statistical error. Extensively studied within the realm of nonparametric statistics, oracle-type inequalities provide valuable insights (see \cite{Johnstone1998Oracle} and references therein). In the field of nonparametric statistics, an oracle inequality outlines bounds of the risk associated with an estimator. This inequality offers an asymptotic guarantee for the performance of an estimator by contrasting it with an oracle procedure that possesses knowledge of some unobservable functions or parameters within the model. \autoref{oracle inequality} establishes an oracle inequality for our PINN model. 
\begin{theorem}\label{oracle inequality}
    For all $t \geq 0$ and $\displaystyle n \geq \max_{\vert \alpha \vert \leq s} \mathrm{PDim}(D^\alpha \mathcal{F}) \vee \frac{8e^2M^2}{C_{11} + 3}$, with probability at least $1 - \exp(-t)$ we have 
    \[
        \mathcal{R}(u_n) - \mathcal{R}(u^*) \lesssim \mathcal{R}(u_\mathcal{F}) - \mathcal{R}(u^*) + \frac{\mathrm{VC}_{\mathcal{F}}}{n} \log n + \frac{t}{n} + \frac{\exp(-t)}{\log(n / \log n)}.
    \]
    Here, $C_{11} > 0$ is a constant satisfying the estimate \eqref{H^s estimates}:
    \[
        \|u - u^*\|^2_{H^s(\mathbb{S}^{d-1})} \leq C_{11} \|\mathcal{L}(u - u^*)\|^2_{L^2(\mathbb{S}^{d-1})},
    \] 
    and the notation
    \begin{equation}\label{VC-dimension of F}
        \mathrm{VC}_{\mathcal{F}} := \mathcal{S}(L + L_0)\log \bigg(k(L + L_0) \cdot \max\{d_{l}: l \in [L+L_0]\} \bigg).
    \end{equation}
    is introduced for convenience, given the recurring usage of this VC-dimension bound.
\end{theorem}

\section{Convergence Analysis: Proof of \autoref{main theorem}}\label{section: Convergence Analysis: Proof of main theorem}

In this section, we concentrate on proving \autoref{main theorem} in order to finalize the convergence analysis of our PICNN model. The combination of the approximation bound (\autoref{upper bound of the approximation error}) and the oracle inequality (\autoref{oracle inequality}) allows us to infer fast convergence rates for PICNN applied to sphere PDEs.

\begin{proof}[Proof of \autoref{main theorem}]
    We assume $\varepsilon > 0, t > 0$, with their exact values to be specified later, and we denote the relative error as $\delta = \varepsilon/\|u\|^2_{W^r_\infty(\mathbb{S}^{d-1})}$. If we take the CNN hypothesis space as 
    \[
        \mathcal{F} = \mathcal{F}(L, L_0 = 2, S, d_{L+1}, d_{L+2}, \mathrm{ReLU}, \mathrm{ReLU}^{k}, M, \mathcal{S})
    \] 
    with
    \begin{align*}
        &3 \leq S \leq d+1, \quad k \geq s, \quad L \asymp \delta^{-\frac{d-1}{2(r - s)}}, \\
        &d_{L+1} \asymp \delta^{- \frac{d + r + s - 1}{2(r - s)(k - s + 1)} - \frac{d+1}{2(r - s)}}, \quad d_{L+2} \asymp \delta^{-\frac{d-1}{2(r-s)}}, \\
        &M \geq 3C_9\|u^*\|_{W^r_\infty(\mathbb{S}^{d-1})}, \\
        &\mathcal{S} \asymp\max\left\{\delta^{- \frac{d + r + s - 1}{2(r-s)(k - s + 1)} - \frac{1}{r-s}}, \delta^{ - \frac{d-1}{2(r-s)}}\right\},
    \end{align*}
    where constant $C_9$ is from \autoref{Sobolev approximation lemma}. Then by \eqref{eq: upper bound of the approximation error}, we have 
    \[
        \mathcal{R}(u_\mathcal{F}) - \mathcal{R}(u^*) \lesssim \varepsilon.
    \]
    By \eqref{VC-dimension of F} one can calculate that 
    \begin{align*}
        \mathrm{VC}_{\mathcal{F}}   &\lesssim \mathcal{S}(L + L_0)\log \bigg(k(L + L_0) \cdot \max\{d_{l}: l \in [L+L_0]\} \bigg) \\
                                    &\asymp \mathcal{S}L \log (kLd_{L+1}) \\
                                    &\asymp \max\left\{\delta^{-\frac{d + r + s - 1}{2(r-s)(k - s + 1)} - \frac{d + 1}{2(r-s)}}, \delta^{-\frac{d-1}{r-s}}\right\} \cdot \log(\delta^{-1})
    \end{align*}

    If $r < \infty$ and $d > 3$, we can balance the terms in $\mathrm{VC}_{\mathcal{F}}$ by choosing  
    \[
        k = s + \left\lceil \frac{r + s + 2}{d - 3} \right\rceil \geq s + 1,
    \]
    then 
    \[
        \mathrm{VC}_{\mathcal{F}} \lesssim \delta^{-\frac{d-1}{r-s}} \log\delta^{-1}.
    \]
    By \autoref{oracle inequality}, with probability at least $1 - \exp(-t)$,
    \begin{align*}
        &\mathcal{R}(u_n) - \mathcal{R}(u^*) \\
        \lesssim{}& \mathcal{R}(u_\mathcal{F}) - \mathcal{R}(u^*) + \frac{\mathrm{VC}_{\mathcal{F}}}{n} \log n + \frac{t}{n} + \frac{\exp(-t)}{\log(n / \log n)} \\
        \lesssim{}& \varepsilon + \frac{\log n}{n} \cdot \varepsilon^{- \frac{d-1}{r-s}}(- \log \varepsilon) + \frac{t}{n} + \frac{\exp(-t)}{\log(n / \log n)}.
    \end{align*}
    The trade-off between the first and second term implies that we can choose 
    \[
        \varepsilon = n^{-a} (\log n)^{2a}, \quad a = \frac{r-s}{(r-s) + (d - 1)},
    \]
    then with probability at least $1 - \exp(-t)$,
    \[
        \mathcal{R}(u_n) - \mathcal{R}(u^*) \lesssim n^{-a}(\log n )^{2a} + \frac{t}{n} + \frac{\exp(-t)}{\log(n / \log n)}.
    \]
    To specify $t$, the trade-off between the first and second term implies that we can choose $t = n^{1-a} (\log n)^{2a}$, then 
    with probability at least $1 - \exp(-n^{1-a} (\log n)^{2a})$,
    \[
        \mathcal{R}(u_n) - \mathcal{R}(u^*) \lesssim n^{-a}(\log n )^{2a}.
    \]

    If $r < \infty$ and $2 \leq d \leq 3$, the first term in $\mathrm{VC}_{\mathcal{F}}$ is always dominant. Then
    \[
        \mathrm{VC}_{\mathcal{F}} \lesssim\delta^{-\frac{d + r + s - 1}{2(r-s)(k - s + 1)} - \frac{d+1}{2(r-s)}} \log\delta^{-1}.
    \]
    Let 
    \[
        b = 1 - \frac{d(k - s + 2) + r + k}{d(k - s + 2) + 2(r - s)(k - s + 1) + r + k}.
    \]
    By a similar argument, we can prove with probability at least $1 - \exp(-n^{1-b} (\log n)^{2b})$,
    \[
        \mathcal{R}(u_n) - \mathcal{R}(u^*) \lesssim n^{-b}(\log n )^{2b}.
    \]

    If $r = \infty$, then for all fixed $k \geq s$,
    \[
        \mathrm{VC}_{\mathcal{F}} \lesssim \delta^{-\frac{1}{2(k - s + 1)}} \log\delta^{-1}.
    \]
    Let 
    \[
        c = 1 - \frac{1}{2(k - s) + 3}.
    \]
    By a similar argument, we can prove with probability at least $1 - \exp(-n^{1-c} (\log n)^{2c})$,
    \[
        \mathcal{R}(u_n) - \mathcal{R}(u^*) \lesssim n^{-c}(\log n )^{2c}.
    \]
    By the $H^s(\mathbb{S}^{d-1})$ estimates \eqref{H^s estimates}, there holds
    \[
        \|u_n - u^*\|^2_{H^s(\mathbb{S}^{d-1})} \lesssim \mathcal{R}(u_n) - \mathcal{R}(u^*).
    \]
    Then we have finished the proof.
\end{proof}

\begin{remark}

The polynomial convergence rates, established by \autoref{main theorem}, is of the form $n^{-a}(\log n)^{2a}$ where the index $a$ is given by
\begin{equation}\label{the learning rates of different cases}
    a = 
    \left\{
    \begin{aligned}
        &\frac{r-s}{(r-s) + (d-1)}, &&\text{ if } r < \infty, d > 3, k = s + \left\lceil \frac{r + s + 2}{d - 3} \right\rceil,\\
        &1 - \frac{d(k - s + 2) + r + k}{d(k - s + 2) + 2(r - s)(k - s + 1) + r + k}, &&\text{ if } r < \infty, 2 \leq d \leq 3, k \geq s, \\
        &1 - \frac{1}{2(k - s) + 3}, &&\text{ if } r = \infty, d \geq 2, k \geq s.
    \end{aligned}
    \right.
\end{equation} 
Generally, a bigger $k$ can lead to a faster convergence rate as it can reduce the term  
\[
    \delta^{-\frac{d + r + s - 1}{2(r-s)(k - s + 1)} - \frac{d + 1}{2(r-s)}},
\] 
which is a key factor in the VC-dimension estimation. Particularly, when $r < \infty, 2 \leq d \leq 3$ or $r = \infty, d \geq 2$, this term dominates the VC-dimension, as
\[
    \delta^{-\frac{d + r + s - 1}{2(r-s)(k - s + 1)} - \frac{d + 1}{2(r-s)}} \geq \delta^{-\frac{d-1}{r-s}}.
\] 
Consequently, in such conditions, a larger $k$ can potentially yield a faster convergence rate theoretically. However, the practical application of $\mathrm{ReLU}^k$ activation functions with a larger $k$  may introduce complications in optimization due to the gradient explosion phenomenon, as observed in our experimental analysis.
\end{remark}

\section{Experiments}\label{Section: Experiments}

In this section, we conduct numerical experiments to verify our theoretical findings and shed light on the conditions that save the algorithm from the curse of dimensionality. As per \autoref{main theorem}, the convergence rate of the PICNN depends on the order $s$ of the PDE, the smoothness $r$ of the solution $u^*$, the dimension $d$, and the activation degrees $k$ of $\mathrm{ReLU}^k$. If the solution is smooth, the convergence rate becomes independent of $d$, indicating that the PICNN can circumvent the curse of dimensionality.

In this section, we consider \autoref{Example: The static Schrödinger equation} with $V(x) \equiv 1$ and different $f$:
\begin{equation}\label{pde of experiment}
    -\Delta_{0} u(x) + u(x) = f(x)  \quad x \in \mathbb{S}^{d-1}.
\end{equation}
After generating $n$ training points $\{x^i\}_{i=1}^n \subset \mathbb{S}^{d-1}$ with random uniform sampling, we can compute the loss function as follows:
\[
    \mathcal{R}_n(u) = \frac{1}{n} \sum_{i=1}^{n} \vert \Delta_0 u(x^i) - u(x^i) + f(x^i) \vert^2,
\]
where $\Delta_0$ is substituted by the right hand of the equality \eqref{the Laplace-Beltrami operator in xyz coordinates}. Note that all the derivatives can be computed by autograd in PyTorch. 

Our analysis uses the $\mathrm{ReLU}$-$\mathrm{ReLU}^k$ network, but in our preliminary trials, we find it difficult to train due to the singularity of the derivative of $\mathrm{ReLU}$ at $x = 0$. Hence, we opt to use $\mathrm{GeLU}$-$\mathrm{GeLU}^k$ as a substitute. $\mathrm{GeLU}$ is a smooth substitute of $\mathrm{ReLU}$ and it was first introduced in \cite{Hendrycks2016Gaussian} by combining the properties from dropout, zoneout and $\mathrm{ReLU}$. 

\autoref{main theorem} suggests that as the training data increase, one should choose a more complex CNN architecture to accommodate the dataset and achieve a promised convergence rate. Additionally, different values for $d, s, r$ and $k$ result in various recommended architectures. However, it is challenging to optimize deep and wide CNNs, as they usually require numerous training epochs to reach loss convergence. In addition, for networks of varying scales, other hyper-parameters such as the step size of gradient descent, batch size and number of epochs are also difficult to select, as our theory only covers the generalization analysis but lacks a discussion on optimization. To ensure that all the tests can be easily implemented and obtain unified and comparable results, we fix the CNN and hyper-parameters setting as 
\begin{itemize}
    \item CNN architecture parameters: $S = 3, k = 3, L = 3, L_0 = 2, d_1 = 12, d_2 = 4$ and $\mathrm{GeLU}$-$\mathrm{GeLU}^3$ as activation functions.
    We also use $56$ channels in each convolution layer to increase the network expressivity power.
    \item Optimizer parameters: Adam \cite{Kingma2015Adam} with sizes $10^{-3}$ or $10^{-4}$, $1/128$ of the training size as batch size and $100$ training epochs. 
    Hence every training process takes 12800 parameter iterations.
    \item Data size: $5120$ random points for test size and the train size varies from $128$ and up to $2^{18} = 262144$. 
    According to the definition of the minimizer \eqref{ERM of PINN}, no validation set is used and we output the best model with the minimum training loss.
\end{itemize}

We train the network using the PICNN method on various sizes of sampled data points and we plot the relative test PINN loss against the training size on the log-log scale. All the experiments are conducted using PyTorch in Python. As mentioned in \autoref{main theorem}, the PINN loss decreases polynomially fast. Therefore, the scattered points are expected to align on a line, with the estimated slope approximately representing the convergence rate.

\subsection{Smooth Solution on 2-D Sphere}\label{section: Smooth Solution on 2-D Sphere}
We first consider a smooth solution on the 2-D sphere as a toy example. Let $u^*(x_1, x_2, x_3) = x_1x_2x_3$. By \eqref{the Laplace-Beltrami operator in xyz coordinates} and 
a straight calculation, $f = 13x_1x_2x_3$. The empirical PINN risk or we can say the train PINN loss, can be calculated as 
\[
    \mathcal{R}_n(u) = \frac{1}{n} \sum_{i=1}^n \vert \Delta_0 u(x^i_1, x^i_2, x^i_3) - u(x^i_1, x^i_2, x^i_3) + 13x^i_1x^i_2x^i_3 \vert^2.
\]
Each $x^i \in \mathbb{S}^2$ represents a training point. 

\begin{figure}[htbp]
    \centering
    \includegraphics[width=0.6\textwidth]{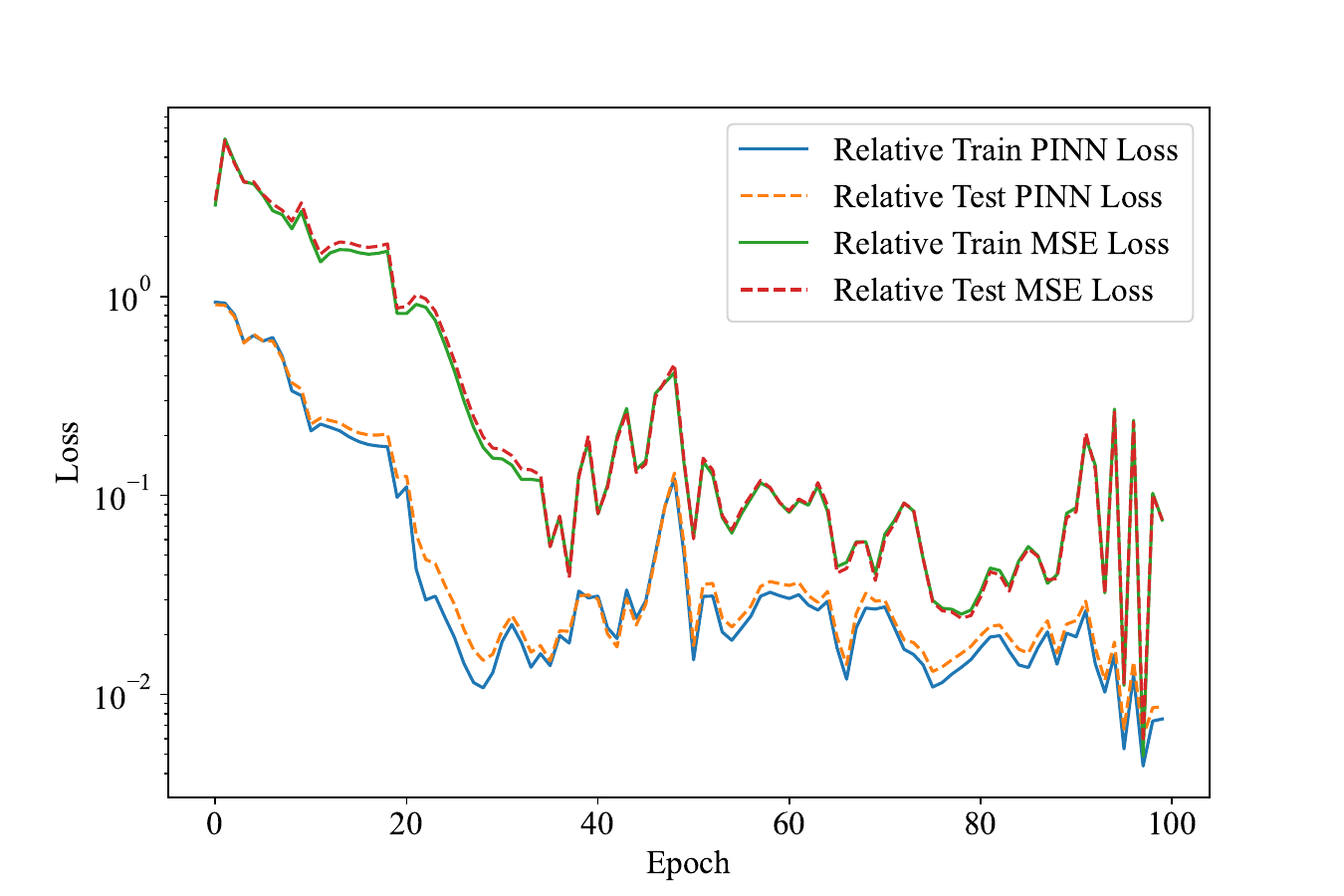}
    \caption{The relative loss history of PICNN model for $u^* = x_1x_2x_3$ on 2-D sphere. Each epoch contains $128$ training iterations.}
    \label{fig:Ex1_LossHist}
\end{figure}
\begin{figure}[htbp]
    \centering
    \includegraphics[width=0.8\textwidth]{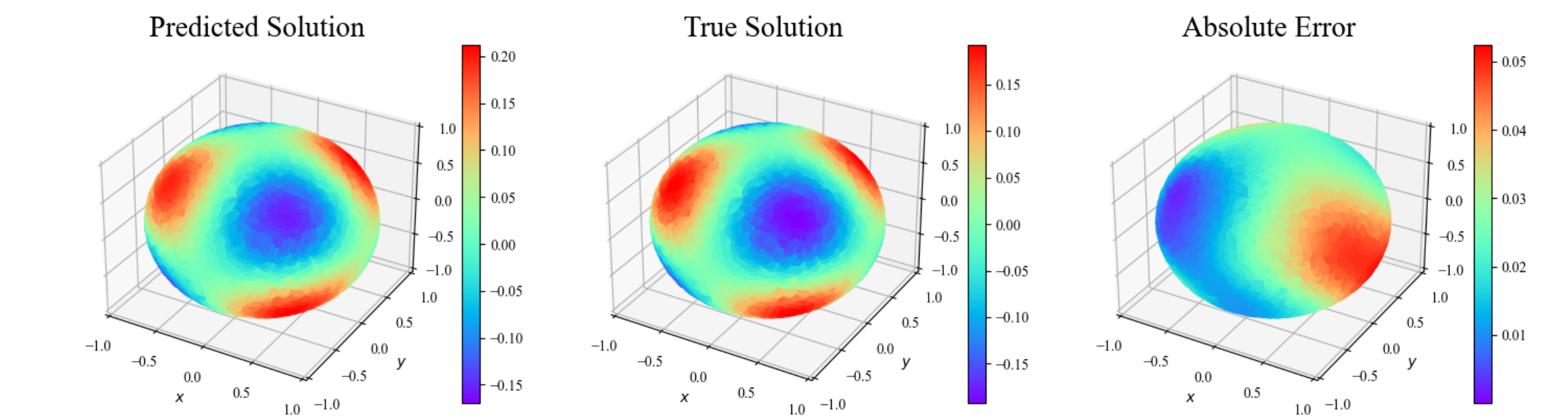}
    \caption{Scatter plots on the test points for $u^* = x_1x_2x_3$ on 2-D sphere.
    From left to right, we show the predicted solution given by the PICNN, the true solution $u^*$ and the absolute error.}
    \label{fig:Ex1_Err}
\end{figure}
\begin{figure}[htbp]
    \centering
    \includegraphics[width=0.6\textwidth]{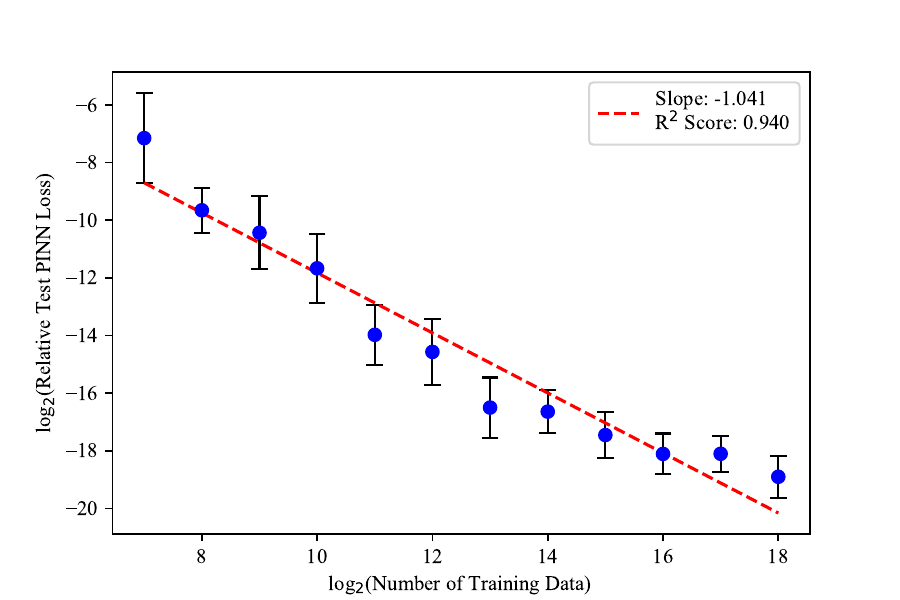}
    \caption{The log-log plot with an estimated convergence slope for $u^* = x_1x_2x_3$ on 2-D sphere. Each point shows the average over $10$ replicates.}
    \label{fig:Ex1_Loss}
\end{figure}

We first show that our PICNN model can learn the solution well. \autoref{fig:Ex1_LossHist} and \autoref{fig:Ex1_Err} illustrate the loss history and the absolute error of a trial with $5120$ test points but only $256$ train points.
These four types of relative losses appeared in \autoref{fig:Ex1_LossHist} is defined as 
\[
    \textrm{Relative Train PINN Loss} = \frac{\frac{1}{n} \sum_{i=1}^n \vert \Delta_0 u(x^i) - u(x^i) + f(x^i) \vert^2}{\frac{1}{n} \sum_{i=1}^n \vert f(x^i) \vert^2},
\]
\[
    \textrm{Relative Test PINN Loss} = \frac{\frac{1}{m} \sum_{i=1}^m \vert \Delta_0 u(y^i) - u(y^i) + f(y^i) \vert^2}{\frac{1}{m} \sum_{i=1}^m \vert f(y^i) \vert^2},
\]
\[
    \textrm{Relative Train MSE Loss} = \frac{\frac{1}{n} \sum_{i=1}^n \vert u(x^i) - u^*(x^i)\vert^2}{\frac{1}{n} \sum_{i=1}^n \vert u^*(x^i) \vert^2},
\]
\[
    \textrm{Relative Test MSE Loss} = \frac{\frac{1}{m} \sum_{i=1}^m \vert u(y^i) - u^*(y^i)  \vert^2}{\frac{1}{m} \sum_{i=1}^m \vert u^*(y^i) \vert^2},
\] 
where $\{x^i\}_{i=1}^n$ is the train set and $\{y^i\}_{i=1}^m$ is the test set. We observe that the test loss and train loss share a very similar value with a small gap, indicating a strong generalization ability of our model. Besides, the PINN loss and MSE loss follow a similar trend, which is a practical manifestation of the strong convexity of PINN risk. This emphasizes that the PINN model can learn the true solution accurately using the information from the physical equation.

To validate our theoretical prediction of the fast polynomial convergence rate, we use a log-log plot to visualize the decreasing trend of the PINN loss in \autoref{fig:Ex1_Loss}, with an estimated linear regression slope and the $\mathrm{R}^2$ score. Each point represents the average log PINN loss from $10$ independent training trials, with the standard deviation also shown. For a smooth solution, \autoref{main theorem} suggests a decreases slope $0.8$ for $k = 3$ and as $k$ increases, the best slope we can expect is $1$, which is discovered in our practical test.

\subsection{Influence of Smoothness}\label{section: Influence of Smoothness}
We now consider other solutions with less smoothness order on the 2-D sphere. We construct
\begin{align*}
    u^*(x_1,x_2,x_3) =& \sum_{i=1}^3 \mathrm{ReLU}^r\left(\frac{1}{2} + \frac{x_i}{r}\right) + \sum_{i=1}^3 \mathrm{ReLU}^r\left(\frac{1}{2} - \frac{x_i}{r}\right), \\
    f(x_1,x_2,x_3) =& \sum_{i=1}^3 \bigg\{\mathrm{ReLU}^r\left(\frac{1}{2} + \frac{x_i}{r}\right) + \mathrm{ReLU}^r\left(\frac{1}{2} - \frac{x_i}{r}\right)  \\
    &+ 2x_i\left[ \mathrm{ReLU}^{r-1}\left(\frac{1}{2} + \frac{x_i}{r}\right) - \mathrm{ReLU}^{r-1}\left(\frac{1}{2} - \frac{x_i}{r}\right) \right] \\
    &- \frac{r-1}{r}(1-x_i^2)\left[ \mathrm{ReLU}^{r-2}\left(\frac{1}{2} + \frac{x_i}{r}\right) + \mathrm{ReLU}^{r-2}\left(\frac{1}{2} - \frac{x_i}{r}\right) \right] \bigg\}.
\end{align*}
Then, one may check that 
\[
    \frac{\partial^r u^*}{\partial x_i^r} = \frac{r!}{r^r}\left(\mathrm{sgn}\left(\frac{1}{2} + \frac{x_i}{r}\right) + (-1)^r \mathrm{sgn}\left(\frac{1}{2} - \frac{x_i}{r}\right)\right)
\]
and $u^* \in W^r_\infty(\mathbb{S}^{2})$ but $u^* \notin W^{r + \varepsilon}_\infty(\mathbb{S}^{2})$ for any $\varepsilon > 0$. We conduct the experiments for various smoothness ranging from $r = 3$ to $r = 8$ and an increasing convergence slope against $r$ is expected by \autoref{main theorem}, which is also showed in \autoref{fig:Ex2_Loss} and \autoref{fig:Ex2_Slope}.
\begin{figure}[htbp]
    \centering
    \includegraphics[width=0.8\textwidth]{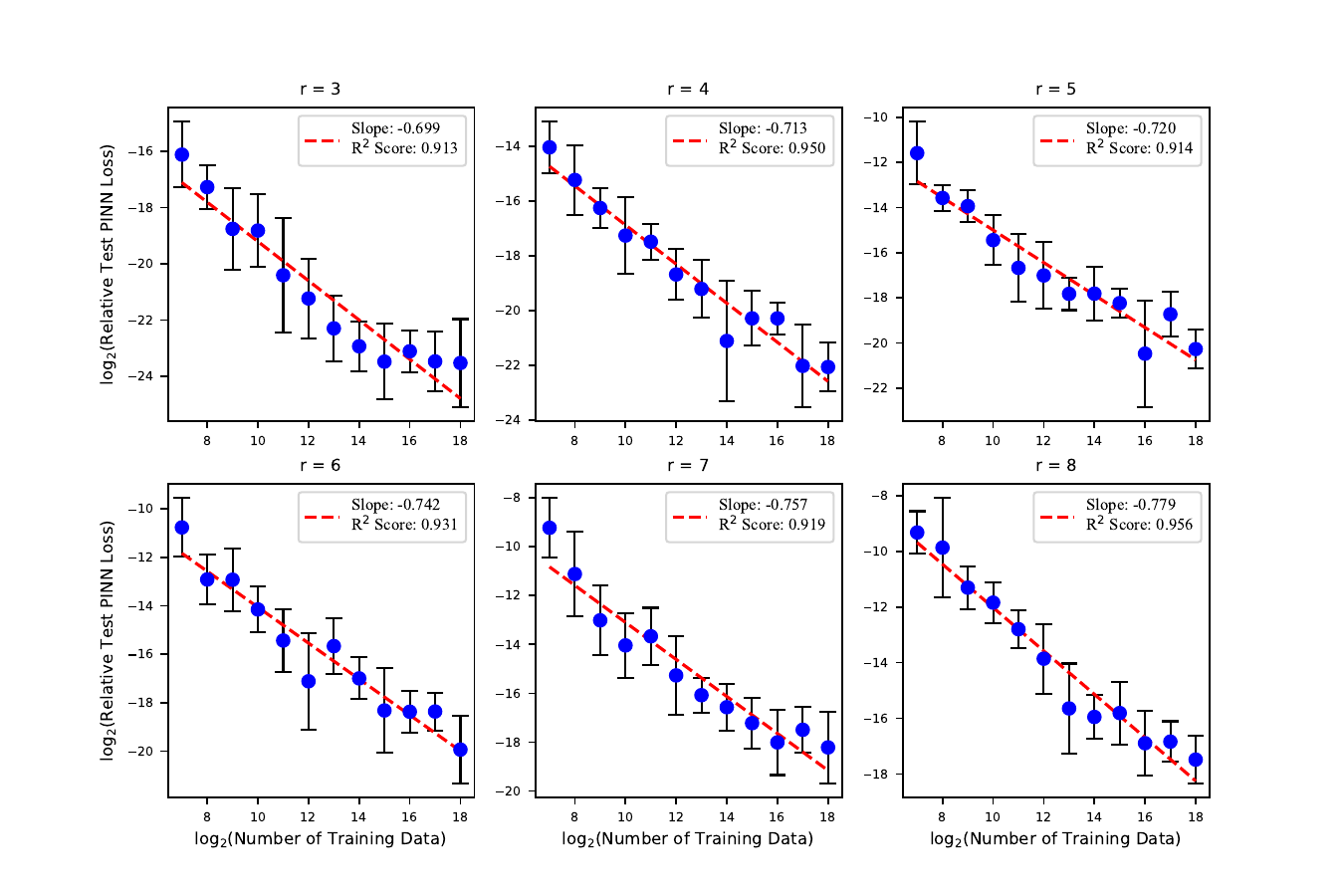}
    \caption{The log-log plot for solutions with different smoothness order on the 2-D sphere. Each point shows the average over $5$ replicates.}
    \label{fig:Ex2_Loss}
    \centering
    \includegraphics[width=0.6\textwidth]{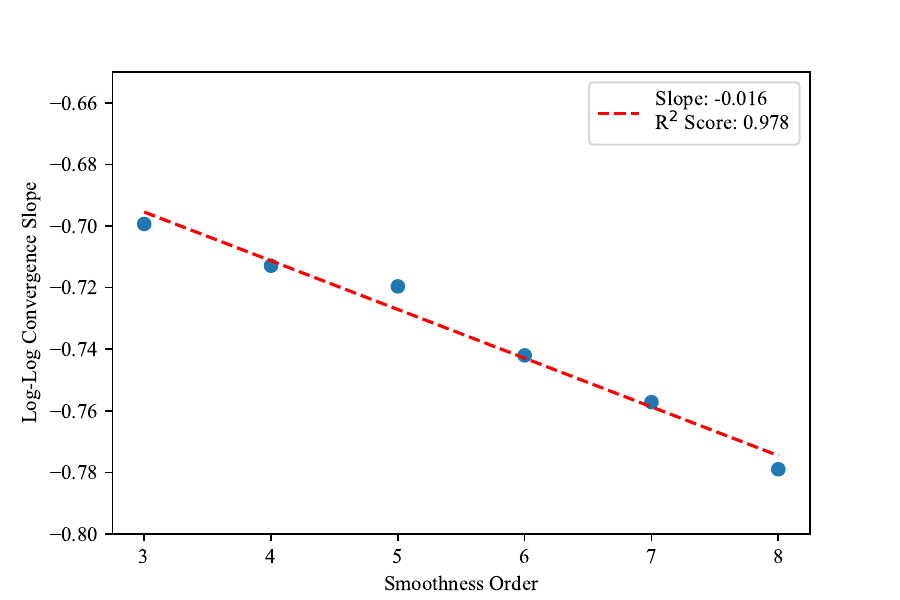}
    \caption{The log-log slope against smoothness order on the 2-D sphere.}
    \label{fig:Ex2_Slope}
\end{figure}

\subsection{Overcoming the Curse of Dimensionality}\label{section: Overcoming the Curse of Dimensionality: Some Promising Perspectives}
The convergence rate is heavily influenced by the dimension, which refers to the phenomenon called the curse of dimensionality. To address it in our framework, reasonable assumptions on the problem must be exploited. In this regard, we propose some promising perspectives.

The first perspective, as we have seen in \autoref{main theorem} and \autoref{section: Influence of Smoothness}, is to assume that the true solution $u^*$ possesses a high degree of smoothness, i.e., $u^* \in W^{r}_\infty(\mathbb{S}^{d-1})$ with an $r$ comparable to $d$. Notably, when $u^*$ is smooth ($r = \infty$), \autoref{main theorem} demonstrates that the convergence rate is independent of $d$. To show this in our experiment, we generalize the experiment from \autoref{section: Smooth Solution on 2-D Sphere} and fix the ground truth solutions as simple polynomials:
\[
    u^*(x) = \sum_{i=1}^{d-2} x_ix_{i+1}x_{i+2} = x_1x_2x_3 + x_2x_3x_4 + \cdots + x_{d-2}x_{d-1}x_d, \quad f(x) = (3d+4)u^*(x).
\]
Similarly, \cite{Lu2022Machine} also considers such polynomials in their experiment. We perform the experiments for various dimensions ranging from $d = 3$ to $d = 10$ and find no significant relationship between the convergence rate and dimension in this problem, as depicted in \autoref{fig:Ex4_Loss} and \autoref{fig:Ex4_Slope}.

\begin{figure}[htbp]
    \centering
    \includegraphics[width=1.0\textwidth]{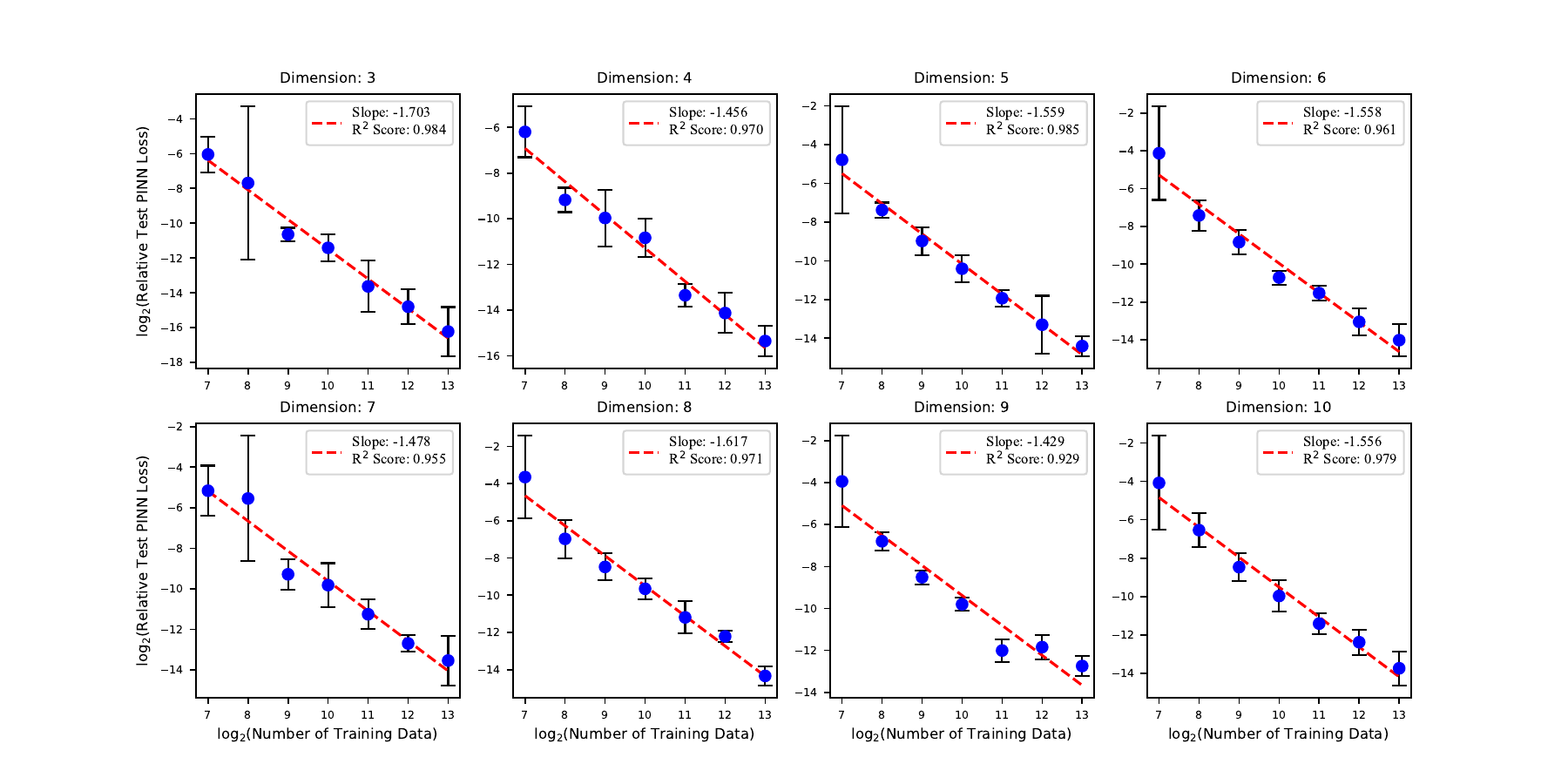}
    \caption{The log-log plot for smooth solutions on $d-1$ dimension sphere. Each point shows the average over $5$ replicates.}
    \label{fig:Ex4_Loss}
    \centering
    \includegraphics[width=0.6\textwidth]{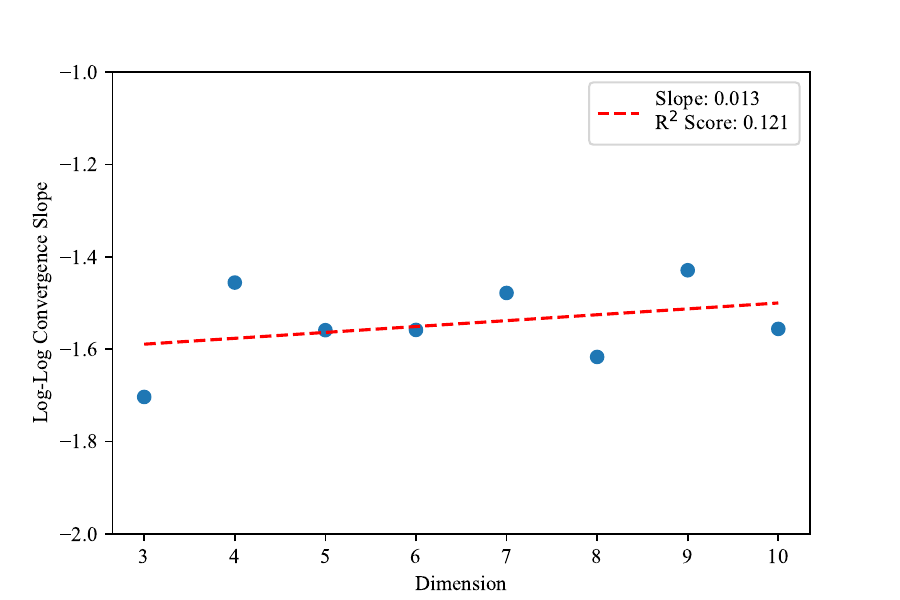}
    \caption{The log-log slope against dimension for smooth solutions on $d-1$ dimension sphere.}
    \label{fig:Ex4_Slope}
\end{figure}

\begin{figure}[htbp]
    \centering
    \includegraphics[width=1.0\textwidth]{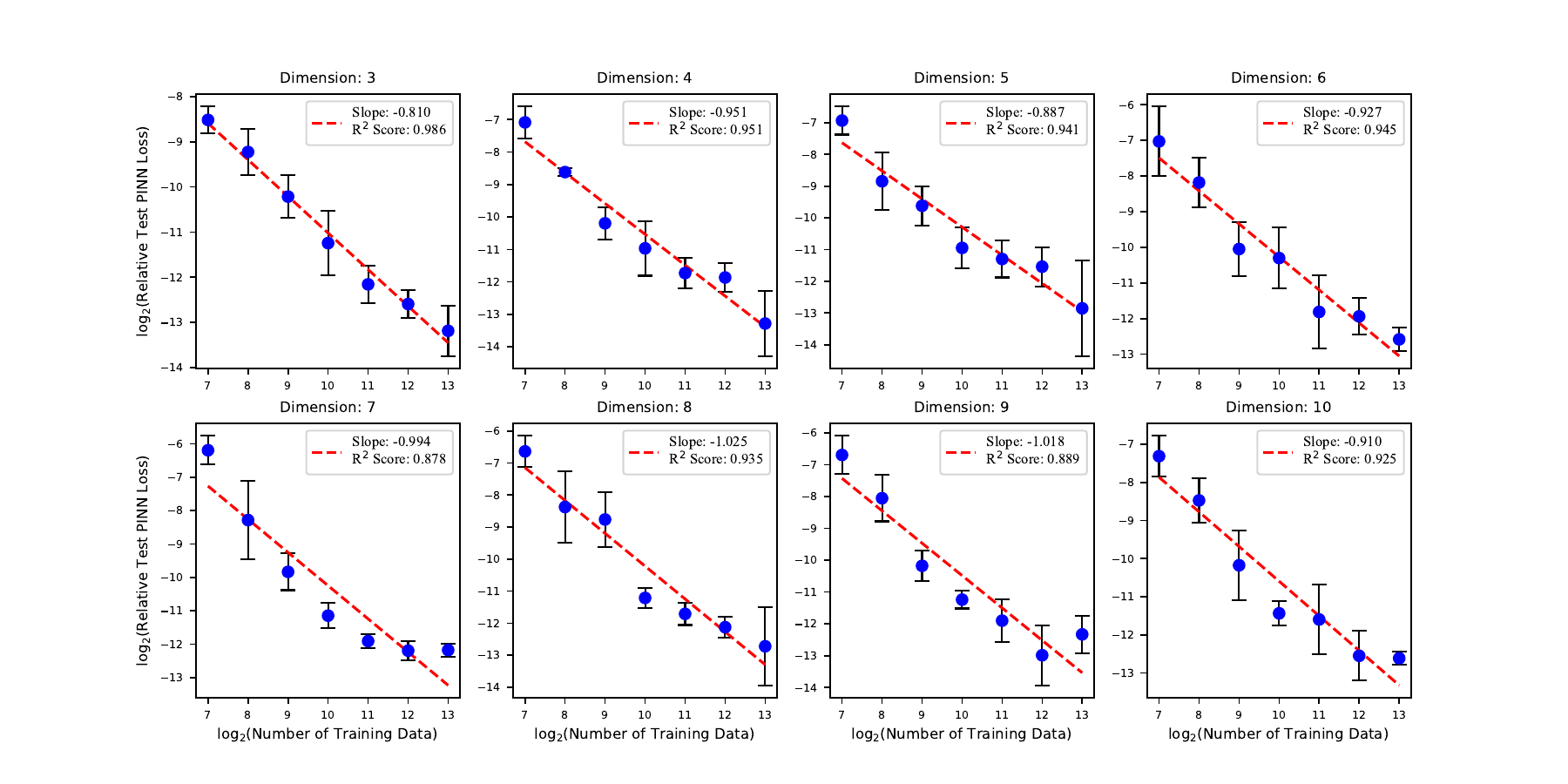}
    \caption{The log-log plot for solutions with special smoothness structure on $d-1$ dimension sphere. Each point shows the average over $5$ replicates.}
    \label{fig:Ex3_Loss}
    \centering
    \includegraphics[width=0.6\textwidth]{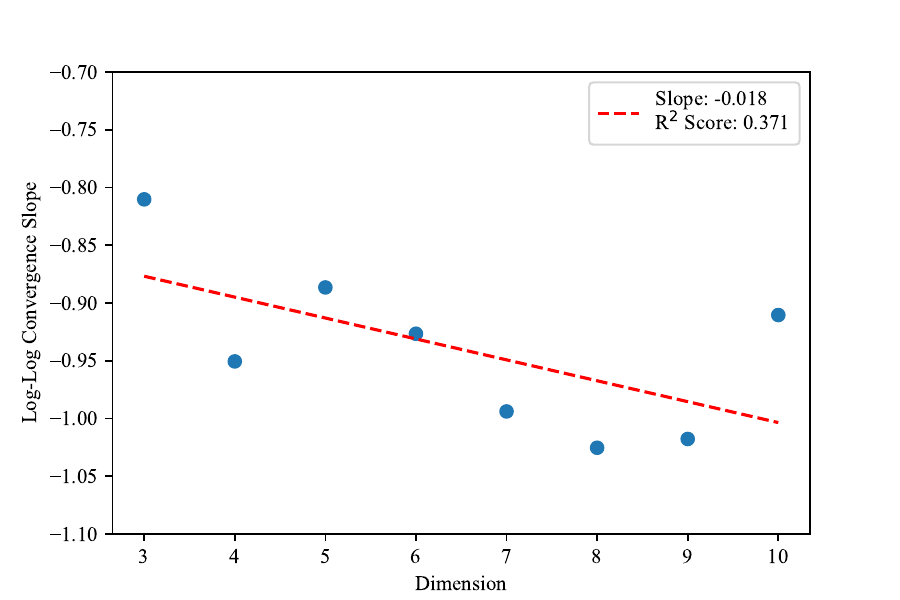}
    \caption{The log-log slope against dimension for solutions with special smoothness structure on $d-1$ dimension sphere.}
    \label{fig:Ex3_Slope}
\end{figure}

Smooth enough ground truth can lead to a fast convergence rate. However, the increasing $d$ still places high demands on $r$. To achieve faster convergence, we turn to make use of some special structures of the ground truth function involving the variables. Some previous works have derived fast approximation rates with neural networks for functions of mixed smoothness \cite{Suzuki2018Adaptivity, Dung2021Deep} or anisotropic smoothness \cite{Suzuki2021Deep}, Korobov space \cite{Montanelli2019New, Mao2022Approximation}, additive ridge functions \cite{Fang2020Theory}, radial functions \cite{Mao2021Theory} and generalized bandlimited functions \cite{Montanelli2021Deep}. Based on these results, we design ground truth functions $u^* \in W^3_\infty(\mathbb{S}^{d-1})$, similar to those in \autoref{section: Influence of Smoothness}:
\begin{align*}
    u^*(x_1,\ldots,x_d) =& \sum_{i=1}^d \mathrm{ReLU}^3\left(\frac{1}{2} + x_i\right) + \sum_{i=1}^d \mathrm{ReLU}^3\left(\frac{1}{2} - x_i\right), \\
    f(x_1,\ldots,x_d) =& \sum_{i=1}^d \bigg\{\mathrm{ReLU}^3\left(\frac{1}{2} + x_i\right) + \mathrm{ReLU}^3\left(\frac{1}{2} - x_i\right)  \\
    &+ 3(d-1)x_i\left[ \mathrm{ReLU}^{2}\left(\frac{1}{2} + x_i\right) - \mathrm{ReLU}^{2}\left(\frac{1}{2} - x_i\right) \right] \\
    & - 6(1-x_i^2)\left[ \mathrm{ReLU}\left(\frac{1}{2} + x_i\right) + \mathrm{ReLU}\left(\frac{1}{2} - x_i\right) \right] \bigg\}.
\end{align*}
Despite having a low general Sobolev smoothness order of $r = 3$, $u^*$ can be expressed as a sum of univariate functions exhibiting isotropic smoothness on each variable. Due to this unique smoothness structure, one can anticipate a fast convergence independent of the increasing dimension for these functions. To verify this claim, we conduct the experiments for various dimensions ranging from $d = 3$ to $d = 10$ and we illustrate this independence in \autoref{fig:Ex3_Loss} and \autoref{fig:Ex3_Slope}.

Our paper investigates functions and PDEs on sphere $\mathbb{S}^{d-1}$, which serves as a prime example of a low dimensional manifold embedded in the Euclidean space $\mathbb{R}^{d}$. Recall the approximation bound (\autoref{upper bound of the approximation error}), the convergence rate $\frac{r-s}{(r-s) + (d-1)}$ derived in our paper and compare to other related results discussing functions and PDEs on $\mathbb{R}^d$, one can observe that the intrinsic dimension $d-1$ replaces the ambient dimension $d$ in our rate. This is also an important insight or question in learning theory: can a model accurately approximate functions on low-dimensional manifolds with a convergence rate that depends on the intrinsic dimension rather than the ambient dimension? Recent representative works \cite{Schmidt-Hieber2019Deep, Chen2019Efficient, Liu2021Besov, Chen2022Nonparametric, Jiao2023Deep} have studied the approximation, nonparametric regression and binary classification on low dimensional manifolds using deep $\mathrm{ReLU}$ network or convolutional residual networks. However, few studies have focused on learning solutions to PDEs on general manifolds by neural networks. While some relative works including \cite{Fang2020PhysicsInformed, Tang2021Physicsinformed, SahliCostabal2024Delta, Bastek2023PhysicsInformed, Zelig2023Numerical} consider $2$-D manifolds in $\mathbb{R}^3$, convergence analysis has not yet been conducted to the best of our knowledge. Inspired by our result, we conjecture that when considering an $s$-order PDE PINN solver on a general $m$-dimensional manifold $\mathcal{M} \subset \mathbb{R}^d$ to approximate the true solution $u^* \in C^r(\mathcal{M})$ (or other regularity function spaces like Sobolev spaces), a fast theoretical convergence rate of $n^{-\frac{r-s}{r-s + m}}$ (up to a logarithmic factor) can be derived. To this end, one must extend existing neural network manifold approximation results to the approximation in high-order smoothness norm like Sobolev norm and bound the complexity of the high-order derivatives of neural networks. Additionally, intrinsic properties of PDEs, such as the strong convexity of PINN risk, must be carefully verified in a general manifold formulation.

We summarize that exploiting smoothness structures of the ground truth function, or the low dimensional manifolds setting, is crucial in overcoming the curse of dimensionality. Our experiments and theoretical analyses have demonstrated the former, while the latter is deserved further study in the future.

\appendix
\section{The Laplace-Beltrami operator and Sobolev spaces on spheres}\label{appsec: The Laplace-Beltrami operator and Sobolev spaces on spheres}
We give a brief introduction to the Laplace-Beltrami operator and Sobolev spaces on spheres. One may refer to \cite{Dai2013Approximation} for more details. Consider a unit sphere $\mathbb{S}^{d-1} \subset \mathbb{R}^d$ for $d \geq 2$. Let $\Delta_0$ symbolize the spherical aspect of the Laplace operator, also known as the Laplace-Beltrami operator, satisfying the equation   
\begin{equation}\label{the Laplace-Beltrami operator}
    \Delta = \frac{\partial^2}{\partial \rho^2}  +  \frac{d-1}{\rho} \frac{\partial}{\partial \rho} + \frac{1}{\rho^2}\Delta_0
\end{equation}
where $x = \rho\theta$ represents the spherical-polar coordinates, $\rho > 0$ indicates the radius, and $\theta \in \mathbb{S}^{d-1}$. Applying equation \eqref{the Laplace-Beltrami operator} and 
\[
    \frac{\partial}{\partial \rho} =\frac{1}{\rho}\sum_{i=1}^d x_i \frac{\partial}{\partial x_i},
\]
one can directly calculate that
\begin{equation}\label{the Laplace-Beltrami operator in xyz coordinates}
    \Delta_0 = \sum_{i=1}^d \frac{\partial^2}{\partial x_i^2} - \sum_{i=1}^d \sum_{j=1}^d x_i x_j \frac{\partial^2}{\partial x_i \partial x_j} - (d-1)\sum_{i=1}^d x_i \frac{\partial}{\partial x_i}.
\end{equation}
The operator $\Delta_0$ is self-adjoint. Moreover, the spherical harmonics are eigenfunctions of $\Delta_0$, with corresponding eigenvalues $\lambda_n = -n(n + d -2)$, where $n=0,1,\cdots$.

The Sobolev space on a sphere, designated as $W^r_p(\mathbb{S}^{d-1})$, is a particular subset of $L^p(\mathbb{S}^{d-1})$, given that $1 \leq p \leq \infty$ and $r \in \mathbb{N}$. This space is characterized by a finite Sobolev norm: 
\[
    \|f\|_{W^r_p(\mathbb{S}^{d-1})} = \|f\|_p + \sum_{1 \leq i < j \leq d} \|D^r_{i,j}f\|_p,
\]
where $\| \cdot \|_p$ denotes the $L^p(\mathbb{S}^{d-1})$ norm corresponding to the uniform measure on $\mathbb{S}^{d-1}$. Here, $D_{i,j}$ denotes the angular derivatives, defined as
\begin{equation}\label{angularderivatives}
    D_{i,j} = x_i\frac{\partial}{\partial x_j} - x_j\frac{\partial}{\partial x_i}.
\end{equation} 
Interestingly, the Laplace-Beltrami operator, $\Delta_0$, can be expressed in terms of these angular derivatives:
\[
    \Delta_0 = \sum_{1 \leq i < j \leq d}D^2_{i,j}.
\] 
For the case where $p = 2$, $W^r_2(\mathbb{S}^{d-1})$ is recognized as a Hilbert space, for which a specific notation, $H^r(\mathbb{S}^{d-1}) = W^r_2(\mathbb{S}^{d-1})$, is utilized. One could also consider a fractional-type Sobolev space, referred to as the Lipschitz space, denoted as $W^{r, \alpha}_p(\mathbb{S}^{d-1})$, with $\alpha \in [0, 1)$. As such, the Sobolev space $W^r_p(\mathbb{S}^{d-1})$ can be defined where $r \geq 0$. For an elaborate discussion of the Lipschitz space, the reader may refer to Chapter 4 of \cite{Dai2013Approximation}.

\section{Two PDE examples}\label{appsec: Two PDE examples}
We provide two PDE examples that satisfy the strong convexity of the PINN risk.
\begin{example}\label{Example: The static Schrödinger equation}
    Consider the static Schrödinger equation given by
    \begin{equation}\label{The static Schrödinger equation}
        -\Delta_0 u(x) + V(x)u(x) = f(x), \quad x \in \mathbb{S}^{d-1}.
    \end{equation}
    Notice that $\Delta_0$ is self-adjoint with eigenvalues $\lambda \leq 0$. If $V$ is a strictly positive constant function, then \eqref{The static Schrödinger equation} has a unique solution $u^*$ and exhibits a relation between the smoothness of $f$ and $u^*$ as described by the elliptic regularity theorem:
    \[
        u^* \in H^r(\mathbb{S}^{d-1}) \Longleftrightarrow f \in H^{r-2}(\mathbb{S}^{d-1}).
    \]
    Furthermore, the strong convexity of the PINN risk is demonstrated in relation to the $H^2(\mathbb{S}^{d-1})$ norm: for all $u \in H^2(\mathbb{S}^{d-1})$, it can be computed that
    \begin{align*}
        \mathcal{R}(u) - \mathcal{R}(u^*) &= \frac{1}{\omega_{d}}\int_{\mathbb{S}^{d-1}} \vert \Delta_0 u(x) -V u(x) + f(x)  \vert^2 d\sigma(x) \\
        &= \frac{1}{\omega_{d}}\int_{\mathbb{S}^{d-1}} \vert \Delta_0 u(x) - V u(x) - \Delta_0 u^*(x) + V u^*(x)  \vert^2 d\sigma(x) \\
        &= \frac{1}{\omega_{d}}\int_{\mathbb{S}^{d-1}} ( \Delta_0 u - \Delta_0 u^* )^2 + ( V u  - V u^* )^2 - 2 (\Delta_0 u - \Delta_0 u^*) \cdot (V u  - V u^* )d\sigma \\
        &= \frac{1}{\omega_{d}}\int_{\mathbb{S}^{d-1}} ( \Delta_0 u - \Delta_0 u^* )^2 + V^2( u  - u^* )^2 + 2V \|\nabla_0( u -  u^*)\|^2d\sigma.
    \end{align*}
    where $\nabla_0$ is the projected gradient operator on $\mathbb{S}^{d-1}$ (see Lemma 1.4.4. in \cite{Dai2013Approximation}).
    Hence,
    \[
        \frac{\min\{1, V^2, 2V\}}{\omega_{d}}\|u - u^*\|^2_{H^2(\mathbb{S}^{d-1})} \leq \mathcal{R}(u) - \mathcal{R}(u^*) \leq \frac{\max\{1, V^2, 2V\}}{\omega_{d}}\|u - u^*\|^2_{H^2(\mathbb{S}^{d-1})}.
    \]
    Having noticed that this equation enjoys many useful properties, we use it to implement our experiments in \autoref{Section: Experiments}.
\end{example}

\begin{example}
    Consider an elliptic equation of order $2s$ given by
    \begin{equation}\label{2s order elliptic equation}
        (-\Delta_0 + \lambda I)^s u(x) = f(x), \quad x \in \mathbb{S}^{d-1},
    \end{equation} where $\lambda > 0$.
    Similar to the discussion in \autoref{Example: The static Schrödinger equation}, we can verify the strong convexity of PINN risk.
    An alternative definition of the Sobolev space $W^r_p(\mathbb{S}^{d-1})$ employs the operator $-\Delta_0 + I$:
    \[
        \|u\|_{W^r_p(\mathbb{S}^{d-1})} = \| (-\Delta_0 + I)^{r/2} u \|_p.
    \]
    In this context, $r \in \mathbb{R}$ and the fractional power is defined in the distributional sense via the spherical harmonic expansion.
    Therefore, this equation inherently obeys the regularity theorem, such that 
    \[
        u \in W^r_p(\mathbb{S}^{d-1}) \Longleftrightarrow f \in W^{r-2s}_p(\mathbb{S}^{d-1}).
    \]
    Accordingly, \autoref{Assumption for PDE} is equivalent to assuming that $f \in W^{r-2s}_\infty(\mathbb{S}^{d-1})$ for some $r \geq 2s+1$.
\end{example}

\section{Bernstein's bound}\label{appsec: Bernstein's bound}
\begin{lemma}\label{Bernstein's bound}
    If \autoref{Assumption for PDE} and \autoref{Assumption for CNN} hold, then for each $u \in \mathcal{F} \cup \{u^*\}$, 
    \begin{equation}\label{variance bound}
        \mathbb{E}_P (L(X, u) - L(X, u^*))^2 \leq C_0 \mathbb{E}_P (L(X, u) - L(X, u^*)) = C_0 (\mathcal{R}(u) - \mathcal{R}(u^*)),
    \end{equation}
    where $P$ is the uniform distribution on $\mathbb{S}^{d-1}$, 
    $C_0 = (M \cdot \sum_{\vert \alpha \vert \leq s} \| a_\alpha \|_\infty + \|f\|_\infty)^2$ and 
    \[
        L(x, u)= \vert (\mathcal{L}u)(x) - f(x) \vert^2.
    \]
    Furthermore, for all $t > 0$, with probability at least $1 - \exp(-t)$, 
    \[
        \mathcal{R}_{n}(u_\mathcal{F}) - \mathcal{R}_{n}(u^*) - \mathcal{R}(u_\mathcal{F}) + \mathcal{R}(u^*) \leq \mathcal{R}(u_\mathcal{F}) - \mathcal{R}(u^*) + \frac{7C_0t}{6n}.
    \]
\end{lemma}

\begin{proof}
    Notice that $u^*$ is the solution of the equation \eqref{PDEs on sphere}, $L(x,u^*) \equiv 0$ so \eqref{variance bound} holds for $u = u^*$. For $u \in \mathcal{F}$ we have 
    \begin{align*}
        &\mathbb{E}_P (L(X, u) - L(X, u^*))^2 \\
        ={}& \frac{1}{\omega_{d}}\int_{\mathbb{S}^{d-1}} \vert  (\mathcal{L}u)(x) - f(x)  \vert^4 d\sigma(x) \\
        \leq{}& \frac{C_0}{\omega_{d}}\int_{\mathbb{S}^{d-1}} \vert  (\mathcal{L}u)(x) - f(x)  \vert^2 d\sigma(x) \\
        ={}& C_0\mathbb{E}_P (L(X, u) - L(X, u^*)).
    \end{align*}
    By Bernstein's inequality and the above estimate, with probability at least $1 - \exp(-t)$, 
    \begin{align*}
        &\mathcal{R}_{n}(u_\mathcal{F}) - \mathcal{R}_{n}(u^*) - \mathcal{R}(u_\mathcal{F}) + \mathcal{R}(u^*) \\
        \leq{}& \sqrt{\frac{2C_0\mathbb{E}_P (L(X, u_{\mathcal{F}}) - L(X, u^*))t}{n}} + \frac{2C_0t}{3n} \\
        \leq{}& \mathbb{E}_P (L(X, u_{\mathcal{F}}) - L(X, u^*)) + \frac{7C_0t}{6n} \\
        ={}& \mathcal{R}(u_\mathcal{F}) - \mathcal{R}(u^*) + \frac{7C_0t}{6n}.
    \end{align*} Then we complete the proof.
\end{proof}

\section{Supplement for approximation error analysis}\label{appsec: Supplement for approximation error analysis}
Here, we aim to complete the proof in \autoref{section: Approximation Error Analysis}, through constructing the neural networks and estimating the approximation error. We first introduce an important lemma from \cite[Theorem 2.6.3]{Dai2013Approximation}.
\begin{lemma}\label{Theorem 2.6.3 of Dai2013Approximation}
    Let $u \in L^p(\mathbb{S}^{d-1})$ for $1 \leq p \leq \infty$. Then there exists a constant $C_1$ only depending on $\eta$ and $d$, such that for all $n_0>0$,
    \[
        \begin{aligned}
            \|L_{n_0}(u)\|_p &\leq C_1\|u\|_p, \\
            \|u-L_{n_0}(u)\|_p &\leq (1+C_1)E_{n_0}(u)_p.
        \end{aligned}
    \]
\end{lemma} 

\subsection{Proof of \autoref{ridge function approximation}}\label{appsec: Proof of ridge function approximation}
    Recall that $D_{i,j}$ denotes the angular derivatives given by \eqref{angularderivatives}. We claim that 
    \begin{align}
        \|D^r_{i,j}(u - L_{n_0}(u))\|_p &\leq C_4E_{n_0}(D^r_{i,j}u)_p, \quad 1 \leq i < j \leq d, \label{Theorem 4.5.5 from Dai2013Approximation} \\
        \|D^r_{i,j} L_{n_0}(u)\|_p &\leq C_4\|D^r_{i,j}u\|_p, \quad \quad \ 1 \leq i < j \leq d, \label{Upper bound for L_n_0} \\
        E_{n_0}(u)_p &\leq C_4n_0^{-r}\|u\|_{W^r_p(\mathbb{S}^{d-1})}. \label{Corollary 4.5.6 from Dai2013Approximation}
    \end{align}
    Consequently, for any positive integer $s \leq r - 1$ we have 
    \begin{align}
        \|u - L_{n_0}(u)\|_{W^s_p(\mathbb{S}^{d-1})} &\leq C_4n_0^{s-r}\|u\|_{W^r_p(\mathbb{S}^{d-1})}, \label{L_n_0 W^s_p approximation}\\
        \|L_{n_0}(u)\|_{W^r_p(\mathbb{S}^{d-1})} &\leq C_4\|u\|_{W^r_p(\mathbb{S}^{d-1})}. \label{L_n_0 W^r_p sup-norm}
    \end{align}
    All these inequalities are also valid for $\widetilde{L}_{n_0}$. Hence, we have
    \[
        \begin{aligned}
            \bigg\|u - \sum_{i=1}^m \mu_i L_{n_0}(u)(y_i) l_{n_0}(\langle x,y_i \rangle)\bigg\|_{W^s_p(\mathbb{S}^{d-1})} &\leq C_4n_0^{s-r}\|u\|_{W^r_p(\mathbb{S}^{d-1})}, \\
            \bigg\|\sum_{i=1}^m \mu_i L_{n_0}(u)(y_i) l_{n_0}(\langle x,y_i \rangle)\bigg\|_{W^{r}_p(\mathbb{S}^{d-1})} &\leq C_4\|u\|_{W^r_p(\mathbb{S}^{d-1})}.
        \end{aligned}
    \]

For the validation of \eqref{Theorem 4.5.5 from Dai2013Approximation} and \eqref{Corollary 4.5.6 from Dai2013Approximation}, one can refer to \cite[Theorem 4.5.5 and Corollary 4.5.6]{Dai2013Approximation}. Inequality \eqref{Upper bound for L_n_0} is a direct consequence of \autoref{Theorem 2.6.3 of Dai2013Approximation} and the equality $L_{n_0}D^r_{i,j} = D^r_{i,j}L_{n_0}$. For any positive integer $s \leq r - 1$, we write
\begin{equation}\label{Sobolev norm of approximation residual}
    \|u - L_{n_0}(u)\|_{W^s_p(\mathbb{S}^{d-1})} = \| u - L_{n_0}(u) \|_p + \sum_{1 \leq i < j \leq d}\|D^s_{i,j} (u - L_{n_0}(u))\|_p.
\end{equation} 
Observing that $D^s_{i,j} u \in W^{r-s}_p(\mathbb{S}^{d-1})$, the combination of \eqref{Theorem 4.5.5 from Dai2013Approximation} and \eqref{Corollary 4.5.6 from Dai2013Approximation} yields
\[
    \begin{aligned}
        \| u - L_{n_0}(u) \|_p &\leq C_4 E_{n_0}(u)_p \leq C_4n_0^{-r}\|u\|_{W^r_p(\mathbb{S}^{d-1})}, \\
        \|D^s_{i,j} (u - L_{n_0}(u))\|_p &\leq C_4E_{n_0}(D^s_{i,j}u)_p \leq C_4n_0^{s-r}\|u\|_{W^r_p(\mathbb{S}^{d-1})}.
    \end{aligned}
\] 
Combining these inequalities with \eqref{Sobolev norm of approximation residual} results in \eqref{L_n_0 W^s_p approximation}, and \eqref{L_n_0 W^r_p sup-norm} is a direct inference from \eqref{Upper bound for L_n_0} (with a possible reselection of the constant $C_4$). Hence, the proof is completed.

\subsection{Proof of the sup-norm bounds in \autoref{construction of convolution layers}}\label{appsec: Proof of the sup-norm bounds in construction of convolution layers}

Define a sequence $W$ supported in $\{0, \ldots , md-1\}$ as $W_{(j-1)d+(d-i)} = (y_j)_i$, where $j = 1, \ldots , m$ and $i = 1, \ldots , d$. Following \cite[Theorem 3]{Zhou2020Universality}, there exists a sequence of convolution kernels $\{w^{(l)}\}_{l=1}^L$ supported on $\{0,\ldots,S-1\}$ with $L = \left\lceil \frac{md-1}{S-2} \right\rceil$ such that $W$ admits the convolutional factorization: 
\[
    W = w^{(L)} \ast w^{(L-1)} \ast \cdots \ast w^{(1)}
\] 
with the corresponding convolution matrices satisfying 
\[
    T = T^{(L)} T^{(L-1)}\cdots T^{(1)}.
\]
Given a convolution kernel $w$, one can define a polynomial on $\mathbb{C}$ by $w(z) = \sum_{i=0}^\infty w_i z^i$. Also by the proof of \cite[Theorem 3]{Zhou2020Universality}, we obtain the following polynomial factorization: 
\[
    W(z) = w^{(L)}(z) w^{(L-1)}(z) \cdots w^{(1)}(z),
\] 
according to the convolutional factorization above.

Considering a rotation of the cubature sample $y = \{y_1,\ldots,y_m\}$, we can assume that $y_m = (1, 0 , \cdots , 0)$, which implies that $W_{md-1} = 1$. Consequently, the polynomial $W(z)$ can be entirely factorized as 
\[
    W(z) = \prod_{i = 1}^{s_1} (z^2 - 2x_iz + x_i^2 + y_i^2)\prod_{j = 2s_1 + 1}^{md-1} (z - x_j),
\] 
comprising $2s_1$ complex roots $x_i \pm iy_i$ and $md-2s_1-1$ real roots $x_j$, both appearing with multiplicity. We can construct $w^{(l)}$ by taking some quadratic factors and linear factors from the above factorization so that $w^{(l)}(z)$ is a polynomial of degree up to $S-1$. Employing Cauchy's bound on the magnitudes of all complex roots, we establish that 
\[
    \vert x_j\vert  \vee \vert x_i \pm iy_i \vert \leq 1 + \max\bigg\{\bigg\vert \frac{W_{md-2}}{W_{md-1}}\bigg\vert, \ldots, \bigg\vert \frac{W_0}{W_{md-1}}\bigg\vert \bigg\} \leq 2.
\] 
Therefore, the coefficients of $w^{(l)}(z)$ are bounded by a constant $C_5$, which depends only on $S$. This verifies that 
\[
    \|w^{(l)}\|_\infty \leq C_5, \quad l = 1, \ldots , L.
\]

Adopting the proof of \cite[Lemma 3]{Fang2020Theory}, we set $b^{(1)} = -\|w^{(1)}\|_1 \textbf{1}_{d_1}$ and 
\[
    b^{(l)} = \left( \prod_{i=1}^{l-1} \| w^{(i)} \|_1 \right) T^{(l)}\textbf{1}_{d_{l-1}} - \left( \prod_{i=1}^{l} \| w^{(i)} \|_1 \right) \textbf{1}_{d_{l}},\mbox{ for } l=2,\ldots,L,
\] 
where $\textbf{1}_{d_l}$ is the constant $1$ vector in $\mathbb{R}^{d_l}$. Consequently, we find that $\|b^{(1)}\|_\infty = \|w^{(1)}\|_1 \leq S \cdot C_5$ and 
\[
    \|b^{(l)}\|_\infty \leq \bigg( \prod_{i=1}^{l-1} \| w^{(i)} \|_1 \bigg) \cdot \| w^{(l)} \|_1 + 
    \prod_{i=1}^{l} \| w^{(i)} \|_1  \leq 2\bigg( \prod_{i=1}^{l} \| w^{(i)} \|_1 \bigg) \leq 2S^l C^l_5.
\] 
Finally setting $C_6 = 2SC_5$ suffices to complete the proof of the sup-norm bounds. 

\subsection{Proof of \autoref{construction of FCNN}}\label{appsec: Proof of construction of FCNN}
We begin with some preliminaries of B-spline interpolation. Consider $u_1, \ldots, u_k$ as functions defined on an interval $I \subset \mathbb{R}$, with $t_1 \leq t_2 \leq \cdots \leq t_k$ as points in $I$. Suppose that 
\[
    t_1, \ldots,t_k = \tau_1, \ldots , \tau_1 ,\ldots, \tau_l,\ldots,\tau_l,
\]
where each $\tau_i$ is repeated $k_i$ times and $\sum_{i=1}^{l} k_i = k$. The $i$-th derivative of $u$ is denoted by $D^iu$, and we define a specific determinant based on these derivatives and the functions, as follows:

\[
    \begin{aligned}
    D
        \begin{pmatrix}
            t_1, \cdots, t_k \\
            u_1, \cdots, u_k
        \end{pmatrix} 
    = \det
        \begin{bmatrix}
            u_1(\tau_1) &  u_2(\tau_1) & \cdots & u_k(\tau_1) \\
            Du_1(\tau_1) &  Du_2(\tau_1) & \cdots & Du_k(\tau_1) \\
            \vdots & \vdots & & \vdots \\
            D^{k_1-1}u_1(\tau_1) & D^{k_1-1}u_2(\tau_1) & \cdots & D^{k_1-1}u_k(\tau_1) \\
            \vdots & \vdots & & \vdots \\
            u_1(\tau_l) &  u_2(\tau_l) & \cdots & u_k(\tau_l) \\
            Du_1(\tau_l) &  Du_2(\tau_l) & \cdots & Du_k(\tau_l) \\
            \vdots & \vdots & & \vdots \\
            D^{k_l-1}u_1(\tau_l) & D^{k_l-1}u_2(\tau_l) & \cdots & D^{k_l-1}u_k(\tau_l) \\
        \end{bmatrix}.
    \end{aligned}
\] 
Given an integer $k>0$, we define the $(k+1)$-th order divided difference of function $f$ on interval $I$ over points $t_1,\ldots,t_{k+2}$ as

\[
    [t_1,\cdots,t_{k+2}]f = \frac
        {
            D
            \begin{pmatrix}
                t_1, t_2, \cdots,t_{k+1}, t_{k+2} \\
                1, x, \cdots, x^{k}, f
            \end{pmatrix}
        }
        {
            D
            \begin{pmatrix}
                t_1, t_2, \cdots,t_{k+1}, t_{k+2} \\
                1, x, \cdots, x^k, x^{k+1}
            \end{pmatrix}
        }.
\] 
Consider a sequence of real numbers $\cdots \leq t_{-1} \leq t_{0} \leq t_1 \leq \cdots $, and for integers $i$ and $k \geq 0$, we define the normalized $(k+1)$-th order B-spline associated with $t_i, \ldots, t_{i+k+1}$ as
\[
    N_i^{k+1}(x) = (-1)^{k+1}(t_{i+k+1}-t_{i})[t_i,\cdots,t_{i+k+1}](x-t)^{k}_+.
\]
Here we notice that $(x)^{k}_+$ is exactly the $\mathrm{ReLU}^k$ function.

For $N > 0$, we fix an extended uniform partition of $[-1,1]$:
\[
    t_1= \cdots = t_{k+1} = -1 < t_{k+2} < \cdots < t_{2N+k} < 1= t_{2N+k+1} = \cdots =t_{2N+2k+1}
\]
with $t_{k+i+1} = -1 + \frac{i}{N}$ for $i=0,\ldots,2N$. For $f \in C[-1,1]$, the corresponding interpolation spline is defined as
\begin{equation}\label{interpolation spline Q}
    Q^{k+1}_N(f) = \sum_{i=1}^{2N + k} J_i(f)N^{k+1}_i(x),
\end{equation}
where $J_i$ are certain fixed linear functionals. By \cite[Theorem 6.22]{Schumaker2007Spline} we know that
\begin{equation}\label{sup-norm bound for B-spline interpolation function}
    \|Q^{k+1}_N(f)\|_\infty \vee \vert J_i(f) \vert \leq (2k+2)^{k+1}\|f\|_\infty.
\end{equation} 
Moreover, we have the following lemma.

\begin{lemma}\cite[Corollary 6.21]{Schumaker2007Spline}
    There exists a constant $C_7$ only depending on $k$ such that, for all 
    $f \in C^k[-1,1]$ and $l = 0, \ldots, k$, we have
    \[
        \|D^l(f - Q^{k+1}_N(f))\|_\infty \leq \frac{C_7}{N^{k-l}}\omega\bigg(D^{k}f,\frac{1}{N}\bigg).
    \]
    Here, $\omega(D^k f,1/N)$ is the modulus of continuity of the $k$-th derivative of $f$:
    \[
        \omega\bigg(D^k f,\frac{1}{N}\bigg) = \sup_{\substack{t, t' \in [-1,1],\\\vert t-t'\vert \leq 1/N}}\vert D^k f(t)-D^k f(t')\vert.
    \]
\end{lemma}
    
Indeed, for a particular $l_{n_0} \in C^\infty [-1,1]$, when $d \geq 3$, given that $\|C^\lambda_i\|_\infty = C^\lambda_i(1) = \binom{i + d - 3}{i} \leq (i+1)^{d-3}$, we can establish the following inequality:
\[
    \|l_{n_0}\|_\infty \leq 5\cdot3^{d-2}n_0^{d-1}.
\]
Obviously, this inequality also holds when $d = 2$. It is crucial to note that $l_{n_0}$ represents a polynomial of degree at most $2n_0$. Applying the Markov inequality for polynomials, one can derive
\[
    \omega\bigg(D^k l_{n_0},\frac{1}{N}\bigg) \leq \frac{1}{N}\|D^{k+1} l_{n_0}\|_\infty \leq \frac{1}{N}(2n_0)^{2k+2}\|l_{n_0}\|_\infty \leq \frac{5 \cdot 2^{2k+2} \cdot 3^{d-2} n_0^{d+2k+1}}{N}.
\]
As a consequence, we obtain
\[
    \|D^l(l_{n_0} - Q^{k+1}_N(l_{n_0}))\|_\infty \leq \frac{C_8n_0^{d+2k+1}}{N^{k-l+1}}
\]
where the constant $C_8$ only depends on $k$ and $d$, which is exactly \eqref{interpolation spline error}.

To demonstrate that ReLU$^k$ is capable of expressing the interpolation spline function, we rewrite $N^{k+1}_i(x)$ for all $x \in [-1, 1]$ explicitly, as stated by \cite[Theorem 4.14 and Equation (4.49)]{Schumaker2007Spline}:
\begin{equation}\label{interpolation spline N}
    N^{k+1}_i(x) = 
    \left\{
        \begin{aligned}
            &\sum_{j=1}^{k-i+2} \alpha_{ij}(x+1)^{k-j+1}_+ + \sum_{j=1}^{i}\beta_{ij}(x+t_{k+j+1})^{k}_+,
            && 1 \leq i \leq k, \\
            &\frac{N^k}{k!}\sum_{j=0}^{k+1}(-1)^j\binom{k+1}{j}\bigg((x-t_i)-\frac{j}{N}\bigg)^k_+,  
            && k+1 \leq i \leq 2N, \\
            &\sum_{j=1}^{2N+k-i+1} \gamma_{ij}(x-t_{i + j -1})^k_+,
            && 2N+1 \leq i \leq 2N+k,
        \end{aligned}
    \right.
\end{equation}
wherein $\alpha_{ij},\beta_{ij},\gamma_{ij}$ are constants depending on $N$ with $\max(\vert\alpha_{ij}\vert, \vert\beta_{ij}\vert, \vert\gamma_{ij}\vert) \leq N^k$. To prove $(x+1)^{k-j+1}_+$ appearing in $N^{k+1}_i(x)$ for $i=1,\ldots,k$ and $j \geq 2$ can be represented using ReLU$^k$, we propose the following lemma.

\begin{lemma}\label{Decomposition of x^s to ReLU^k}
    Given integers $l,k \in \mathbb{N}\cup \{0\}$ with $l < k$,  there exist constants $\zeta_{li}$ and $\xi_{li}, i=0,\ldots,k$, such that 
    \[
        x^l =\sum_{i=0}^{k} \xi_{li} (x+\zeta_{li})^{k}, \quad x \geq 0.
    \]
\end{lemma}
\begin{proof}
    Compare the coefficients on both sides results in a linear system of equations
        \[
            \begin{bmatrix}
                \zeta^k_{l,0} & \zeta^k_{l,1} & \cdots & \zeta^k_{l,k} \\
                \vdots & \vdots & \vdots & \vdots \\
                \binom{k}{k-l}\zeta^{k-l}_{l,0} & \binom{k}{k-l}\zeta^{k-l}_{l,1} & \cdots & \binom{k}{k-l}\zeta^{k-l}_{l,k} \\
                \vdots & \vdots & \vdots & \vdots \\
                1 & 1 & \ldots & 1
            \end{bmatrix}
            \begin{bmatrix}
                \xi_{l,0} \\
                \xi_{l,1} \\
                \vdots \\
                \vdots \\
                \xi_{l,k}
            \end{bmatrix}
            =
            \begin{bmatrix}
                0 \\
                \vdots \\
                1 \\
                \vdots \\
                0
            \end{bmatrix}.
        \]
        The solution requires finding $\zeta_{li}$, $i=0,\ldots,k$, such that the matrix is invertible. It is noteworthy that the determinant of the matrix is 
        \[
            \prod_{i=1}^k \binom{k}{i} \cdot \prod_{0 \leq i < j \leq k} (\zeta_{li} - \zeta_{lj}).
        \]
        Thus, we only need to ensure that $\zeta_{li} \neq \zeta_{lj}$ for $i \neq j$. The proof is then finished.
\end{proof}

For $x \in [-1,1]$, $(x+1)_+ = x+1 \geq 0$, hence \autoref{Decomposition of x^s to ReLU^k} can be applied to $(x+1)^{k-j+1}_+$. We can express $N^{k+1}_i(x)$ as a one hidden layer FCNN using ReLU$^k$. Now given the output of the convolution layer $\mathcal{D}(F^{(L)}(x))$ as presented in \autoref{construction of convolution layers}, we build the fully connected layer as follows.

We first consider a simpler FCNN that only accepts one input, denoted as $\langle x,y_1 \rangle + B^{(L)}$, which is also an output of the convolution layer $\mathcal{D}(F^{(L)})(x)$. For $j = 1, \ldots, k + 1$, \autoref{Decomposition of x^s to ReLU^k} allows us to write
\[
    \begin{aligned}
        &\alpha_{ij}(\langle x,y_1 \rangle + 1)^{k-j+1}_+ \\
        ={}& \alpha_{ij}\sum_{s=0}^{k} \xi_{k-j+1,s} (\langle x,y_1 \rangle + \zeta_{k-j+1,s} +1)^k_+ \\
        ={}& \sum_{s=0}^{k} \text{sgn}(\alpha_{ij}\xi_{k-j+1,s}) 
        \bigg(\sqrt[k]{\vert\alpha_{ij} \xi_{k-j+1,s}\vert}\bigg(\langle x,y_1 \rangle + B^{(L)} \bigg) - \sqrt[k]{\vert\alpha_{ij} \xi_{k-j+1,s}\vert}\bigg(B^{(L)} - \zeta_{k-j+1,s} - 1\bigg) \bigg)^{k}_+.
    \end{aligned}
\] 
Let $w = \sqrt[k]{\vert\alpha_{ij} \xi_{k-j+1,s}\vert}$ and $b = \sqrt[k]{\vert\alpha_{ij} \xi_{k-j+1,s}\vert}(B^{(L)} - \zeta_{k-j+1,s} - 1)$, we see that $\vert w \vert \lesssim N$, $\vert b \vert \lesssim NB^{(L)} \leq NC_6^L$ where $C_6$ is the constant from \autoref{construction of convolution layers}, and 
\[
    \alpha_{ij}(\langle x,y_1 \rangle + 1)^{k-j+1}_+ = \sum_{s=0}^{k} \text{sgn}(\alpha_{ij}\xi_{k-j+1,s}) (w(\mathcal{D}(F^{(L)}(x)))_1 - b)^{k}_+.
\]
Here, $(w(\mathcal{D}(F^{(L)}(x)))_1 - b)^k_+$ can be represented by the output of one single hidden node using ReLU$^k$. Hence, $\alpha_{ij}(\langle x,y_1 \rangle + 1)^{k-j+1}_+$ is a linear combination of outputs from $k+1$ hidden nodes. A similar argument also applies to 
\[
    \begin{aligned}
        &\beta_{ij}( \langle x,y_1 \rangle + t_{k+j+1})^k_+,\\
        &\frac{N^k}{k!}(-1)^j\binom{k+1}{j}\bigg(( \langle x,y_1 \rangle -t_i)-\frac{j}{N}\bigg)^k_+ \mbox{ and }\\
        &\gamma_{ij}( \langle x,y_1 \rangle -t_{i + j -1})^k_+.
    \end{aligned}
\] 
Each of these terms can be represented as the multiplication of a sign and the output of a single hidden node. After calculating these terms, note by \eqref{interpolation spline N} that $N^{k+1}_i(\langle x,y_1 \rangle)$ is also a linear combination of them. The overall number of nodes needed to represent $N^{k+1}_i(\langle x,y_1 \rangle)$ is 
\[
    \left\{
        \begin{array}{lcl}
            (k+1)(k-i+2) + i, &  &
            \mbox{ for }  1 \leq i \leq k, \\
            k + 2, & & 
            \mbox{ for }  k+1 \leq i \leq 2N, \\
            2N+k-i+1, & &
            \mbox{ for }  2N+1 \leq i \leq 2N+k,
        \end{array}
    \right.
\] 
where the maximum number of nodes, $k^2 + 2k + 2$, is achieved when $i = 1$. By \eqref{interpolation spline Q}, $Q_N^{k+1}(l_{n_0})(\langle x,y_1 \rangle)$ can  be expressed as a linear combination of $N^{k+1}_i(\langle x,y_1 \rangle)$. The total number of nodes needed to build $Q_N^{k+1}(l_{n_0})(\langle x,y_1 \rangle)$ is 
\[
    \begin{aligned}
        &\sum_{i = 1}^{k} \bigg((k+1)(k-i+2) + i\bigg) + (2N - k)(k + 2) + \sum_{i = 2N+1}^{2N+k} \bigg(2N+k-i+1\bigg) \\
        ={}& \frac{1}{2}(k^3 + 4k^2 + 4Nk + k + 8N).
    \end{aligned}
\] 
We regard $Q_N^{k+1}(l_{n_0})(\langle x,y_1 \rangle)$ as the output of this simpler FCNN which only accepts the input $\langle x,y_1 \rangle + B^{(L)}$ and uses $(k^3 + 4k^2 + 4Nk + k + 8N)/2$ hidden nodes.

Employing the same argument, we can construct $m$ simpler FCNNs, each accepting only a single input, $\langle x,y_i \rangle + B^{(L)}$, and utilizing $(k^3 + 4k^2 + 4Nk + k + 8N)/2$ hidden nodes. These $m(k^3 + 4k^2 + 4Nk + k + 8N)/2$ hidden nodes are concatenated to form the first hidden layer $F^{(L+1)}$ of the FCNN, with the $m$ output nodes forming the second hidden layer $F^{(L+2)}$. The first layer employs $\mathrm{ReLU}^k$, while the second utilizes $\mathrm{ReLU}$, i.e., 
\[
    \begin{aligned}
        F^{(L+1)}(x) &= \mathrm{ReLU}^k\left( W^{(L+1)}\mathcal{D}\big(F^{(L)}(x)\big) - b^{(L+1)}\right) \in \mathbb{R}^{m(k^3 + 4k^2 + 4Nk + k + 8N)/2}, \\
        F^{(L+2)}(x) &= \mathrm{ReLU}\left( W^{(L+2)}F^{(L+1)}(x) - b^{(L+2)}\right) \in \mathbb{R}^{m}.
    \end{aligned}
\]

These FCNNs also employ identical parameters, thus resulting in shared weights and biases. The parameters can be represented as follows:
\begin{equation}\label{forms of weights and biases}
    \begin{aligned}
        W^{(L+1)} &=     
        \begin{bmatrix}
            \overrightarrow{w_1}  & \overrightarrow{0}         & \cdots & \overrightarrow{0}         & \overrightarrow{0}         & \cdots & \overrightarrow{0} \\
            \overrightarrow{0}    & \overrightarrow{w_1}       & \cdots & \overrightarrow{0}         & \overrightarrow{0}         & \cdots & \overrightarrow{0} \\
            \vdots                & \vdots                     & \cdots & \vdots                     & \vdots                     & \vdots & \vdots\\
            \overrightarrow{0}    & \overrightarrow{0}         & \cdots & \overrightarrow{w_1}       & \overrightarrow{0}         & \cdots & \overrightarrow{0} \\
            \overrightarrow{0}    & \overrightarrow{0}         & \cdots & \overrightarrow{0}         & \overrightarrow{w_1}       & \cdots & \overrightarrow{0} \\
        \end{bmatrix} 
        \in \mathbb{R}^{m(k^3 + 4k^2 + 4Nk + k + 8N)/2 \times \left\lfloor d_L / d\right\rfloor }, \\
        b^{(L+1)} &= 
        \begin{bmatrix}
            \overrightarrow{b_1} \\
            \vdots \\
            \overrightarrow{b_1} \\
        \end{bmatrix} \in \mathbb{R}^{m(k^3 + 4k^2 + 4Nk + k + 8N)/2}, \\
        W^{(L+2)} &=     
        \begin{bmatrix}
            \overrightarrow{w_2}^T  & \overrightarrow{0}     & \cdots & \overrightarrow{0}     & \overrightarrow{0}   \\
            \overrightarrow{0}      & \overrightarrow{w_2}^T & \cdots & \overrightarrow{0}     & \overrightarrow{0}   \\
            \vdots                  & \vdots                 & \cdots & \vdots                 & \vdots               \\
            \overrightarrow{0}      & \overrightarrow{0}     & \cdots & \overrightarrow{w_2}^T & \overrightarrow{0}   \\
            \overrightarrow{0}      & \overrightarrow{0}     & \cdots & \overrightarrow{0}     & \overrightarrow{w_2}^T \\
        \end{bmatrix} 
        \in \mathbb{R}^{m \times m(k^3 + 4k^2 + 4Nk + k + 8N)/2}, \\
        b^{(L+2)} &= b_2 \mathbf{1}_m \in \mathbb{R}^{m}, \\
    \end{aligned}
\end{equation} 
where $\overrightarrow{w_1}, \overrightarrow{b_1}, \overrightarrow{w_2}\in \mathbb{R}^{(k^3 + 4k^2 + 4Nk + k + 8N)/2}$, $\overrightarrow{0}$ is a $0$ constant vector with a suitable size, and 
\[
    b_2 = - \|Q_N^{k+1}(l_{n_0})\|_\infty.
\] 
From this construction, we see that for $1 \leq i \leq m$, 
\[
    \begin{aligned}
        \big( F^{(L+2)}(x) \big)_i &= \mathrm{ReLU}\bigg(\big(W^{(L+2)}F^{(L+1)}(x) - b^{(L+2)}\big)_i\bigg) \\
        &= Q_N^{k+1}(l_{n_0})(\langle x,y_i \rangle) + \|Q_N^{k+1}(l_{n_0})\|_\infty.
    \end{aligned}
\] 
Then we apply an additional affine transformation to yield the output for our entire network, e.g., 
\[
    \begin{aligned}
        F^{(L+3)}(x) &= W^{(L+3)} \cdot F^{(L+2)}(x) - b^{(L+3)} \\ 
        &= \sum_{i=1}^m \mu_i L_{n_0}(u)(y_i) \big( F^{(L+2)}(x) \big)_i - \sum_{i=1}^m \mu_i L_{n_0}(u)(y_i) \|Q_N^{k+1}(l_{n_0})\|_\infty \\
        &= \sum_{i=1}^m \mu_i L_{n_0}(u)(y_i) Q_N^{k+1}(l_{n_0})(\langle x,y_i \rangle).
    \end{aligned}
\] 
In conclusion, we have constructed an FCNN with output \eqref{FCNN output} using two hidden layers of width $d_{L + 1} = m(k^3 + 4k^2 + 4Nk + k + 8N)/2$ and $d_{L + 2} = m$. This FCNN also satisfies the following boundedness constraints:
\[
    \begin{aligned}
        &\|W^{(L+1)}\|_{\max} \lesssim N, \quad \|b^{(L+1)}\|_\infty \lesssim NC_6^L, \\
        &\|W^{(L+2)}\|_{\max} = \max\{\vert J_i(l_{n_0})\vert : i = 1, \ldots, 2N+k\} \leq 5 \cdot 3^{d-2} \cdot (2k+2)^{k+1} n_0^{d-1}, \\
        &\|b^{(L+2)}\|_\infty = \|Q_N^{k+1}(l_{n_0})\|_\infty \leq 5 \cdot 3^{d-2} \cdot (2k+2)^{k+1} n_0^{d-1}, \\
        &\|W^{(L+3)}\|_{\max} = \max\{\vert \mu_i L_{n_0}(u)(y_i) \vert: i = 1, \ldots, m\} \leq C_1\|u\|_\infty, \\
        &\vert b^{(L+3)} \vert = \bigg\vert \sum_{i=1}^m \mu_i L_{n_0}(u)(y_i) \|Q_N^{k+1}(l_{n_0})\|_\infty \bigg\vert \leq 5 \cdot 3^{d-2} \cdot (2k+2)^{k+1} n_0^{d-1}C_1\|u\|_\infty,
    \end{aligned}
\]
where $C_1$ is from \autoref{Theorem 2.6.3 of Dai2013Approximation} and $C_6$ is from \autoref{construction of convolution layers}. The total number of free parameters contributed by the FCNN is 
\[
    \|\overrightarrow{w_1}\|_0 + \|\overrightarrow{b_1}\|_0 + \|\overrightarrow{w_2}\|_{0} + 1 + m + 1 = \frac{3}{2} \cdot (k^3 + 4k^2 + 4Nk + k + 8N) + m + 2.
\]
The proof is then finished.

\subsection{Sobolev approximation lemma}\label{appsec: Sobolev approximation lemma}
We now present \autoref{Sobolev approximation lemma}, which asserts the approximation capability of our CNN architecture with respect to the Sobolev norm. The following lemma invokes a definition analogous to definition \eqref{Assumption for hypothesis space} to describe the CNN function space equipped with hybrid activation functions. This is denoted by $\mathcal{F}(L, L_0=2, S, d_{L+1}, d_{L+2}, \mathrm{ReLU}, \mathrm{ReLU}^k, M=\infty, \mathcal{S})$. Here, the network in this function space employs activation function $\sigma^{(l)} = \mathrm{ReLU}$ for $l = 1, \ldots, L, L + 2$ and $\sigma^{(L+1)} = \mathrm{ReLU}^k$. The condition $M=\infty$ signifies there are no constraints on the supremum norm of the output function or its derivatives.

\begin{lemma}\label{Sobolev approximation lemma}
    Let $d \geq 2$, $3 \leq S \leq d+1$ and $1 \leq p \leq \infty$. Given non-negative integers $s,k,r,n_0$ satisfying $0 \leq s < \min\{r,k+1\}$ and $n_0 \geq  1$, then for all $u \in W^r_p(\mathbb{S}^{d-1})$, there exists a network
    \[
        F^{(L+3)} \in \mathcal{F}(L, L_0 = 2, S, d_{L+1}, d_{L+2}, \mathrm{ReLU}, \mathrm{ReLU}^k, M=\infty, \mathcal{S})
    \] 
    with
    \[
        \begin{aligned}
            &L = \left\lceil \frac{C_2n_0^{d-1}d -1}{S-2} \right\rceil, \quad d_{L+1} \asymp n_0^{\frac{d + r + s - 1}{k - s + 1} + d + 1}, 
            \quad d_{L+2} \asymp n_0^{d - 1},\\
            &\mathcal{S} \asymp \max\big\{n_0^{2+\frac{d + r + s - 1}{k - s + 1}}, n_0^{d-1}\big\},
        \end{aligned}
    \] 
    and a constant $C_9$ only depending on $d, p, k$ and $s$, such that 
    \begin{equation}\label{CNN W^s_p approximation}
        \|u - F^{(L+3)}\|_{W^s_p(\mathbb{S}^{d-1})} \leq C_9n_0^{s-r}\|u\|_{W^r_p(\mathbb{S}^{d-1})}.
    \end{equation}
    Here, $C_2$ is the constant from \autoref{cubature formula}. Furthermore, the network parameters satisfy the sup-norm constraints \eqref{sup-norm constraint for parameters} with 
    \[
        \begin{aligned}
            &B_1 = C_5 ,\quad B_2 = C^L_6,\quad B_3 \lesssim \max\big\{n_0^{2+\frac{d + r + s - 1}{k - s + 1}}, 3^{d} (k+1)^{k+1} n_0^{d-1}\big\}, \\
            &B_4 \lesssim\max\big\{C_6^{L} \cdot n_0^{2+\frac{d + r + s - 1}{k - s + 1}}, 3^{d} (k+1)^{k+1} n_0^{d-1}\big\},
        \end{aligned}
    \]  
    where the constants $C_5$ and $C_6$ are from \autoref{construction of convolution layers}.
\end{lemma}
\begin{proof} 
    We first prove \eqref{CNN W^s_p approximation}. We leverage \autoref{construction of convolution layers} which provides us with the equation
    \[
        \mathcal{D}(F^{(L)}(x)) = (\langle x , y_1 \rangle, \cdots, \langle x , y_m \rangle, 0, \cdots, 0)^T + B^{(L)} {\bf 1}_{\left\lfloor \frac{d+L(S-1)}{d}\right\rfloor},
    \] 
    where $m = \left\lceil C_2n_0^{d-1} \right\rceil$ and the cubature rule $\{(\mu_i,y_i)\}_{i=1}^m$ is guaranteed by \autoref{cubature formula} for degree $4n_0$. As \autoref{construction of FCNN} stated, we can construct a fully connected neural network with two hidden layers following the convolution layers. The activation functions of these hidden layers are $\sigma^{(L+1)} = \mathrm{ReLU}^k$ and $\sigma^{(L+2)} = \mathrm{ReLU}$, respectively. This network outputs
    \[
        F^{(L+3)}(x) = \sum_{i=1}^m \mu_i L_{n_0}(u)(y_i) Q_N^{k+1}(l_{n_0})\big(\langle x,y_i \rangle\big).
    \] 
    Define $B^{k+1}_N = l_{n_0} - Q_N^{k+1}(l_{n_0})$. For $1 \leq i < j \leq d, 1 \leq p  < \infty$ and $q = p/(p-1)$, by Hölder's inequality, we have 
    \begin{equation}
        \begin{aligned}
            &\bigg\| D^s_{i,j}\bigg(F^{(L+3)} - \sum_{t=1}^m \mu_t L_{n_0}(u)(y_t) l_{n_0}(\langle \cdot,y_t \rangle) \bigg) \bigg\|_p^p \\
            ={}& \frac{1}{\omega_{d}}\int_{\mathbb{S}^{d-1}} \bigg\vert \sum_{t=1}^m \mu_t L_{n_0}(u)(y_t) D^s_{i,j} B^{k+1}_N(\langle x,y_t \rangle) \bigg\vert^p d\sigma(x) \\
            \leq{}& \frac{1}{\omega_{d}}\sum_{t=1}^m \mu_t \vert L_{n_0}(u)(y_t) \vert^p \\ 
            {}& \cdot \int_{\mathbb{S}^{d-1}} 
            \bigg( \sum_{t=1}^m \mu_t \bigg\vert \bigg(x_i\frac{\partial}{\partial x_j} - x_j\frac{\partial}{\partial x_i}\bigg)^s B^{k+1}_N(\langle x,y_t \rangle) \bigg\vert^q \bigg)^{\frac{p}{q}} d\sigma(x).
        \end{aligned}
    \end{equation} 
    It is shown by \cite[Theorem 1]{Feng2023Generalization} that, there exists a constant $C'_1$ that only depends on $d, p$ and $\eta$, such that 
    \[
        \sum_{t=1}^m \mu_t \vert L_{n_0}(u)(y_t) \vert^p \leq C'^p_1\|u\|^p_{W^r_p(\mathbb{S}^{d-1})}.
    \] 
    Further, using the chain rule, we can validate the existence of a constant $C'_2$ depending only on $s$ which ensures that 
    \[
        \bigg\vert \bigg(x_i\frac{\partial}{\partial x_j} - x_j\frac{\partial}{\partial x_i}\bigg)^s B^{k+1}_N(\langle x,y_t \rangle) \bigg\vert \leq C'_2 \sum_{i=1}^s\|D^i B^{k+1}_N\|_\infty.
    \] 
    By \eqref{interpolation spline error}, there exists a constant $C'_3$ dependent only on $k, s$ and $d$ such that 
    \[
        \sum_{i=1}^s\|D^i B^{k+1}_N\|_\infty \leq \frac{C'_3 n_0^{d+2k+1}}{N^{k-s+1}}.
    \]
    Therefore, we have 
    \[
        \begin{aligned}
            &\frac{1}{\omega_{d}}\int_{\mathbb{S}^{d-1}} \bigg( \sum_{t=1}^m \mu_t \bigg\vert \bigg(x_i\frac{\partial}{\partial x_j} - x_j\frac{\partial}{\partial x_i}\bigg)^s B^{k+1}_N(\langle x,y_t \rangle) \bigg\vert^q \bigg)^{\frac{p}{q}} d\sigma(x) \\
            \leq{}& \bigg(\frac{C'_2C'_3 n_0^{d+2k+1}}{N^{k-s+1}} \bigg)^p \bigg(\sum_{t=1}^m \mu_t\bigg)^{\frac{p}{q}}=\bigg(\frac{C'_2C'_3 n_0^{d+2k+1}}{N^{k-s+1}} \bigg)^p.
        \end{aligned}
    \] 
    Combining the aforementioned inequalities, we obtain 
    \[ 
        \bigg\| D^s_{i,j}\bigg(F^{(L+3)} - \sum_{t=1}^m \mu_t L_{n_0}(u)(y_t) l_{n_0}(\langle \cdot,y_t \rangle) \bigg) \bigg\|_p \leq \frac{C'_1C'_2C'_3 n_0^{d+2k+1} \|u\|_{W^r_p(\mathbb{S}^{d-1})}}{N^{k-s+1}}.
    \] 
    An analogous process substantiates that (with a reselection of the constants $C'_1,C'_2,$ and $C'_3$)
    \[ 
        \bigg\|F^{(L+3)} - \sum_{t=1}^m \mu_t L_{n_0}(u)(y_t) l_{n_0}(\langle \cdot,y_t \rangle) \bigg\|_p \leq \frac{C'_1C'_2C'_3 n_0^{d+2k+1} \|u\|_{W^r_p(\mathbb{S}^{d-1})}}{N^{k+1}},
    \]
    and the above analysis pertains equally when $p = \infty$. Therefore, there exists a constant $C'_4$ only depends on $d, p, k, s$ and $\eta$ such that
    \begin{equation}\label{intermediate term W^s_p bound}
        \bigg\|F^{(L+3)} - \sum_{t=1}^m \mu_t L_{n_0}(u)(y_t) l_{n_0}(\langle \cdot,y_t \rangle) \bigg\|_{W^s_p(\mathbb{S}^{d-1})} \leq \frac{C'_4 n_0^{d+2k+1} \|u\|_{W^r_p(\mathbb{S}^{d-1})}}{N^{k-s+1}}.
    \end{equation} 
    Combining the above estimates with \autoref{ridge function approximation} leads to
    \[
        \begin{aligned}
            \|u - F^{(L+3)}\|_{W^s_p(\mathbb{S}^{d-1})} &\leq \bigg\|u - \sum_{i=1}^m \mu_i L_{n_0}(u)(y_i) l_{n_0}(\langle \cdot,y_i \rangle)\bigg\|_{W^s_p(\mathbb{S}^{d-1})} + \\
            &\quad\bigg\|F^{(L+3)} - \sum_{i=1}^m \mu_i L_{n_0}(u)(y_i) l_{n_0}(\langle \cdot,y_i \rangle)\bigg\|_{W^s_p(\mathbb{S}^{d-1})} \\
            &\leq C_4n_0^{s-r}\|u\|_{W^r_p(\mathbb{S}^{d-1})} + \frac{C'_4 n_0^{d+2k+1} \|u\|_{W^r_p(\mathbb{S}^{d-1})}}{N^{k-s+1}}.
        \end{aligned}
    \] 
    Next, we select $N = \left\lceil n_0^{\frac{d+r+2k-s+1}{k-s+1}} \right\rceil = \left\lceil n_0^{2+\frac{d+r+s-1}{k-s+1}} \right\rceil$, and noting that $\eta$ is fixed beforehand, which yields 
    \[
        \|u - F^{(L+3)}\|_{W^s_p(\mathbb{S}^{d-1})} \leq C_9n_0^{s-r}\|u\|_{W^r_p(\mathbb{S}^{d-1})}
    \] 
    with a constant $C_9$ only depends on $d, p, k$ and $s$. By applying the preceding analysis with $m = \left\lceil C_2n_0^{d-1} \right\rceil$ and $N = \left\lceil n_0^{2+\frac{d+r+s-1}{k-s+1}} \right\rceil$, we derive the parameters of $\mathcal{F}$ and their corresponding sup-norm constraints. Thus we complete the proof.
\end{proof}

\subsection{Proof of \autoref{upper bound of the approximation error}}\label{appsec: Proof of upper bound of the approximation error}
Now we can complete the proof of \autoref{upper bound of the approximation error}, deriving an upper bound of the approximation error. Let 
\[
    n_0 = \left\lceil \delta^{-\frac{1}{2(r-s)}} \right\rceil.
\] 
Then by \autoref{Sobolev approximation lemma}, we can take a CNN hypothesis space 
\[
    \mathcal{F} = \mathcal{F}(L, L_0 = 2, S, d_{L+1}, d_{L+2}, \mathrm{ReLU}, \mathrm{ReLU}^{k}, M = \infty, \mathcal{S})
\]
where $L, S, k, d_{L+1}, d_{L+2}, \mathcal{S}$ are specified in \autoref{Sobolev approximation lemma}. Consequently, there exists a $u \in \mathcal{F}$ satisfying 
\begin{equation}\label{uniform approximation bound}
    \|u - u^*\|^2_{W^s_{\infty}(\mathbb{S}^{d-1})} \leq C^2_9 n_0^{2(s-r)}\|u^*\|^2_{W^r_\infty(\mathbb{S}^{d-1})} \leq 2C^2_9 \varepsilon,
\end{equation}
where $C_9$ is from \autoref{Sobolev approximation lemma}. By the triangle inequality, we can deduce that
\[
    \|u\|_{W^s_{\infty}(\mathbb{S}^{d-1})} \leq C_9 \sqrt{2\varepsilon} + \|u^*\|_{W^r_{\infty}(\mathbb{S}^{d-1})} \leq 3C_9\|u^*\|_{W^r_\infty(\mathbb{S}^{d-1})}.
\]
We then further constrain the function space with the extra sup-norm constraint \eqref{sup-norm constraint} stipulating that
\[
    M \geq 3C_9\|u^*\|_{W^r_\infty(\mathbb{S}^{d-1})}.
\] 
Then we ensure that $u \in \mathcal{F}$ remains valid and that $u$ meets the approximation bound defined in \eqref{uniform approximation bound}. Combining \eqref{uniform approximation bound} with \eqref{Sobolev norm controls PINN risk}, we conclude that 
\[
    \mathcal{R}(u_\mathcal{F}) - \mathcal{R}(u^*) \leq \mathcal{R}(u) - \mathcal{R}(u^*) \lesssim \|u - u^*\|^2_{W^s_{\infty}(\mathbb{S}^{d-1})} \leq 2C_9^2 \varepsilon.
\] 
The proof is then finished.

\section{Supplement for statistical error analysis}\label{appsec: Supplement for statistical error analysis}
Recall that we have defined $L: \mathbb{S}^{d-1} \times \mathcal{F} \to \mathbb{R}$ by
\begin{equation}\label{defL}
    L(x, u) = \vert (\mathcal{L}u)(x) - f(x) \vert^2,
\end{equation} 
which measures the residual of equation \eqref{PDEs on sphere}. Then the PINN risk is the expectation of $L$ over a uniform distribution $P$ on $\mathbb{S}^{d-1}$, given by
\[
    \mathcal{R}(u) = \mathbb{E}_P[L(X, u)].
\]
Whereas the empirical PINN risk
\[
    \mathcal{R}_n(u) = \frac{1}{n}\sum_{i=1}^n L(X_i, u)
\] 
is the empirical mean of i.i.d. sample $\{X_i\}_{i=1}^n$ from $P$.
 
Define the composition of $L$ and $\mathcal{F}$ as
\[
    L \circ  \mathcal{F} = \{g:\mathbb{S}^{d-1} \to \mathbb{R} \vert g(x) = L(x, u), u \in \mathcal{F}\}.
\]
The Rademacher complexity of $L \circ \mathcal{F}$ plays a pivotal role in bounding the statistical error. For an i.i.d. sequence $\{\varepsilon_i\}_{i=1}^n$ with $\varepsilon_i = \pm 1$ having equal probability independent of $X$ and sample data $\mathbf{x} = \{X_i\}_{i=1}^n$, define the empirical Rademacher complexity of $\mathcal{F}$ as
\[
    \mathrm{Rad}(\mathcal{F}, \mathbf{x}) = \mathbb{E}_\varepsilon \sup_{u \in \mathcal{F}} \bigg\vert \frac{1}{n}\sum_{i=1}^n \varepsilon_i u(X_i) \bigg\vert.
\]
The Rademacher complexity is calculated by taking a secondary expectation over the sample data:
\[
    \mathrm{Rad}(\mathcal{F}) := \mathbb{E}_P \mathbb{E}_\varepsilon \sup_{u \in \mathcal{F}} \bigg\vert \frac{1}{n}\sum_{i=1}^n \varepsilon_i u(X_i) \bigg\vert.
\]
Accordingly, the empirical Rademacher complexity and Rademacher complexity of $L \circ \mathcal{F}$ are defined as
\[
    \begin{aligned}
        \mathrm{Rad}(L \circ \mathcal{F}, \mathbf{x}) &:= \mathbb{E}_\varepsilon \sup_{u \in \mathcal{F}} \bigg\vert \frac{1}{n}\sum_{i=1}^n \varepsilon_i L(X_i, u) \bigg\vert, \\
        \mathrm{Rad}(L \circ \mathcal{F}) &:= \mathbb{E}_P \mathbb{E}_\varepsilon \sup_{u \in \mathcal{F}} \bigg\vert \frac{1}{n}\sum_{i=1}^n \varepsilon_i L(X_i, u) \bigg\vert.
    \end{aligned}
\]
Additionally, we will employ the following notations:
\[
    \begin{aligned}
        \widetilde{\mathcal{F}} &:= \mathcal{F} - u^* = \{u-u^* : u \in \mathcal{F}\}, \\
        D^\alpha\widetilde{\mathcal{F}} &:= \{D^\alpha(u-u^*) : u \in \mathcal{F}\}, \quad \vert \alpha \vert \leq s.
    \end{aligned}
\] 
as well as their respective Rademacher complexities in our subsequent analysis.

A frequently applied constraint on $L$ requires it to be Lipschitz with respect to $u$. This constraint allows us to bound $\mathrm{Rad}(L \circ \mathcal{F}, \mathbf{x})$ by $\mathrm{Rad}(\mathcal{F}, \mathbf{x})$ employing a contraction lemma. Within our PINN framework, note that we impose sup-norm constraints as defined in \autoref{Assumption for PDE} and \autoref{Assumption for CNN}. Consequently, for all $u_1, u_2 \in \mathcal{F} \cup \{u^*\}$, the following is true: 

\begin{equation}\label{Lipschitz of loss function}
    \begin{aligned}
        &\vert L(x, u_1) - L(x, u_2) \vert \\
        ={}& \big\vert (\mathcal{L}u_1)(x) + (\mathcal{L}u_2)(x) - 2f(x) \big\vert
        \cdot \big\vert (\mathcal{L}u_1)(x) - (\mathcal{L}u_2)(x)\big\vert \\
        \leq{}& 2\bigg(M \cdot \sum_{\vert \alpha \vert \leq s} \| a_\alpha \|_\infty + \|f\|_\infty \bigg)
        \cdot \bigg(\sum_{\vert \alpha \vert \leq s} \| a_\alpha \|_\infty \cdot \vert D^\alpha u_1(x) - D^\alpha u_2(x) \vert \bigg).
    \end{aligned}
\end{equation}

Hence, by a contraction technique, we derive the subsequent lemma.

\begin{lemma}\label{contraction lemma}
Assume that $L(x,u)$ defined by \eqref{defL} is Lipschitz with respect to $u$. Then
\[
    \mathrm{Rad}(L \circ \mathcal{F}, \mathbf{x}) \leq C_{10} \sum_{\vert \alpha \vert \leq s} \mathrm{Rad}(D^\alpha\widetilde{\mathcal{F}}, \mathbf{x}),
\]
where $C_{10}: = 4\big(M \cdot \sum_{\vert \alpha \vert \leq s} \| a_\alpha \|_\infty + \|f\|_\infty \big) \cdot \max_{\vert \alpha \vert \leq s} \| a_\alpha \|_\infty.$
\end{lemma}
\begin{proof}
    By \eqref{Lipschitz of loss function}, for all $x \in \mathbb{S}^{d-1}$ and $u_1, u_2 \in \mathcal{F} \cup \{u^*\}$ we have 
    \[
        \begin{aligned}
            &\vert L(x, u_1) - L(x, u_2) \vert \\
            \leq{}& \frac{C_{10}}{2} \cdot \sum_{\vert \alpha \vert \leq s} \vert D^\alpha u_1(x) - D^\alpha u_2(x) \vert \\
            ={}& \frac{C_{10}}{2} \cdot \sum_{\vert \alpha \vert \leq s} \vert D^\alpha (u_1 - u^*)(x) - D^\alpha (u_2 - u^*)(x) \vert.
        \end{aligned}
    \]
    Using the contraction lemma \cite[Theorem 4.12]{Ledoux1991Probability}, the Lipschitz property of $L$ yields a bound of the empirical Rademacher complexity:
    \[
        \mathrm{Rad}(L \circ \mathcal{F}, \mathbf{x}) 
        \leq C_{10}\mathrm{Rad} \bigg( \sum_{\vert \alpha \vert \leq s} D^\alpha\widetilde{\mathcal{F}}, \mathbf{x} \bigg) 
        \leq C_{10}\sum_{\vert \alpha \vert \leq s}\mathrm{Rad}(D^\alpha\widetilde{\mathcal{F}}, \mathbf{x}).
    \] 
    The proof is then finished.
\end{proof}
Next, we aim to establish a novel localized complexity analysis to bound the statistical error by the Rademacher complexity.

\subsection{Localized complexity analysis}\label{appsec: Localized complexity analysis}
To estimate the term 
\[
    \mathcal{R}(u_n) - \mathcal{R}(u^*) - \mathcal{R}_n(u_n) + \mathcal{R}_n(u^*),
\]
we propose a novel localization approach that makes use of standard tools from empirical process analysis, including peeling, symmetrization, and Dudley's chaining. Firstly, we bound the estimated error by considering the supremum norm of an empirical process:
\begin{equation}\label{empirical process}
    \mathcal{R}(u_n) - \mathcal{R}(u^*) - \mathcal{R}_n(u_n) + \mathcal{R}_n(u^*) \leq \sup_{u \in \mathcal{F}} \mathcal{R}(u) - \mathcal{R}(u^*) - \mathcal{R}_n(u) + \mathcal{R}_n(u^*).
\end{equation} 
For all $u \in \mathcal{F}$ and $\gamma > 0$, we define an auxiliary function $g_u : \mathbb{S}^{d-1} \to \mathbb{R}$ as
\[
    g_u(x) := L(x,u) - L(x,u^*).
\] 
Furthermore, for all $\gamma > \inf_{u \in \mathcal{F}}\mathbb{E}_P(g_u(X))$, we introduce the following localized function classes:
\[
    \begin{aligned}
        \mathcal{F}_\gamma &:= \{u: \mathbb{S}^{d-1} \to \mathbb{R} \ \vert \ u \in \mathcal{F}, \mathbb{E}_P(g_u(X)) = \mathcal{R}(u) - \mathcal{R}(u^*) \leq \gamma\}, \\
        \widetilde{\mathcal{F}_\gamma} &:= \{u - u^*: \mathbb{S}^{d-1} \to \mathbb{R} \ \vert \ u \in \mathcal{F}_\gamma\}, \\
        \mathcal{G}_\gamma &:= \{g_u: \mathbb{S}^{d-1} \to \mathbb{R} \ \vert \ u \in \mathcal{F}_\gamma\}, \\
        D^\alpha \mathcal{F}_\gamma &:= \{D^\alpha u: \mathbb{S}^{d-1} \to \mathbb{R} \ \vert \ u \in \mathcal{F}_\gamma\}, \\
        D^\alpha \widetilde{\mathcal{F}_\gamma} &:= \{D^\alpha u - D^\alpha u^*: \mathbb{S}^{d-1} \to \mathbb{R} \ \vert \ u \in \mathcal{F}_\gamma\}.
    \end{aligned} 
\]
Given a sequence of i.i.d. samples $\mathbf{x} = \{X_i\}_{i=1}^n$ and some function class $\mathcal{G}$ of measurable functions on $\mathbb{S}^{d-1}$, we can define the empirical $L^2$-norm $\|\cdot\|_{L^2(\mathbf{x})}$ of $g \in \mathcal{G}$ as
\[
    \| g \|^2_{L^2(\mathbf{x})}:= \frac{1}{n} \sum_{i=1}^n g(X_i)^2,
\] and the sample expectation as
\[
    \mathbb{E}_{\mathbf{x}}g:= \frac{1}{n} \sum_{i=1}^n g(X_i).
\] 
Next, we introduce the symmetrization and Dudley's chaining lemmas, which can be found in \cite{Bartlett2005Local} and \cite{Mendelson2003Few}.

\begin{lemma}\label{symmetrization}
    Consider an i.i.d. sample sequence $\mathbf{x} = \{X_i\}_{i=1}^n$. For any function $g \in \mathcal{G}$, if $\|g\|_\infty \leq G$ and $\mathrm{Var}(g) \leq V$, then the following inequalities hold: for all $t > 0$, with probability at least $1 - \exp(-t)$, 
    \[
        \sup_{g \in \mathcal{G}} \mathbb{E}_{\mathbf{x}}g - \mathbb{E}_P g \leq 3\mathrm{Rad}(\mathcal{G}) + \sqrt{\frac{2Vt}{n}} + \frac{4Gt}{3n};
    \]
    and with probability at least $1 - 2\exp(-t)$,
    \[ 
        \sup_{g \in \mathcal{G}} \mathbb{E}_{\mathbf{x}}g - \mathbb{E}_P g \leq 6\mathrm{Rad}(\mathcal{G}, \mathbf{x}) + \sqrt{\frac{2Vt}{n}} + \frac{23Gt}{3n}.
    \]
    The above inequalities are also valid for $\sup_{g \in \mathcal{G}} \mathbb{E}_P g - \mathbb{E}_{\mathbf{x}}g$.
\end{lemma}

\begin{lemma}\label{Dudley's chaining}
    Let $\mathcal{N}(\delta, \mathcal{G}, \| \cdot \|_{L^2(\mathbf{x})})$ denote the covering number of $\mathcal{G}$, taking radius $\delta$ and metric $\| \cdot \|_{L^2(\mathbf{x})}$. There holds
    \[
        \mathrm{Rad}(\{g: g \in \mathcal{G}, \| g \|_{L^2(\mathbf{x})} \leq \gamma\}, \mathbf{x}) \leq \inf_{0 < \beta < \gamma} \bigg\{4\beta + \frac{12}{\sqrt{n}} \int_{\beta}^\gamma \sqrt{\log \mathcal{N}(\delta, \mathcal{G}, \| \cdot \|_{L^2(\mathbf{x})})}d\delta \bigg\}.
    \]
\end{lemma}

The quantity $\log \mathcal{N}(\delta, \mathcal{G}, \| \cdot \|_{L^2(\mathbf{x})})$ is recognized as the metric entropy of $\mathcal{G}$. Its upper bound can be defined by the pseudo-dimension, hence by the VC-dimension of $\mathcal{G}$, as outlined in \cite[Theorem 12.2 and Theorem 14.1]{Anthony1999Neural}.

\begin{lemma}\label{empirical metric entropy bound}
    Suppose that for every function $g \in \mathcal{G}$, $\|g\|_\infty \leq M$. For a set of points $\{x_i\}_{i=1}^n$, define $\mathcal{G}\vert_{x_1, \ldots ,x_n}:= \{g\vert_{x_1, \ldots ,x_n}:\{x_i\}_{i=1}^n \to \mathbb{R} \vert  g \in \mathcal{G}\}$. Then for all $n \geq \mathrm{PDim}(\mathcal{F})$ and $\delta > 0$, 
    \[ 
        \mathcal{N}(\delta, \mathcal{G}\vert_{x_1, \ldots ,x_n}, \| \cdot \|_\infty ) \leq \left(\frac{2enM}{\delta \mathrm{PDim}(\mathcal{G})}\right)^{\mathrm{PDim}(\mathcal{G})}.
    \]
\end{lemma}

\begin{lemma}\label{pseudo-dimension bound}
    For a class of neural networks $\mathcal{G}$ with a fixed architecture and fixed activation functions. Then,
    \[
        \mathrm{PDim}(\mathcal{G}) \leq \mathrm{VCDim}(\mathrm{sgn}(\mathcal{G}_0)),
    \] 
    where $\mathcal{G}_0$ is a set extended from $\mathcal{G}$ by adding one extra input neuron and one extra computational neuron. The additional computational neuron is a linear threshold neuron that takes inputs from the output unit of $\mathcal{G}$ and the new input neuron.
\end{lemma}

When \autoref{Dudley's chaining} is later employed, it is essential to control the empirical norm $\| \cdot \|_{L^2(\mathbf{x})}$ by limiting the excess risk $\mathbb{E}_P(g_u(X))$ to no more than $\gamma$. We formalize this idea in the following lemma. We define $\kappa = \# \{\alpha: \vert \alpha \vert \leq s, a_\alpha (\cdot) \neq 0\}$ to represent the number of non-zero terms in $\mathcal{L}$ satisfying \autoref{Assumption for PDE}.

\begin{lemma}\label{upper bound of the empirical norm}
    For $t > 0$, if 
    \[
        \gamma \geq \max_{\vert \alpha \vert \leq s}\bigg\{ 12M\mathrm{Rad}(D^\alpha \widetilde{\mathcal{F}_\gamma}), \frac{8C_{11}M^2t}{n} ,\frac{16M^2t}{3n} \bigg\},
    \]
    then with probability at least $1 - \kappa\exp(-t)$, there holds
    \[
        \max_{\vert \alpha \vert \leq s}\sup_{u \in \mathcal{F}_\gamma} \|D^\alpha u - D^\alpha u^*\|_{L^2(\mathbf{x})} \leq \sqrt{(C_{11} + 3)\gamma},
    \]
    where $C_{11} > 0$ is the constant satisfying the estimate \eqref{H^s estimates}:
    \begin{equation}\label{H^s estimates with C_11}
        \|u - u^*\|^2_{H^s(\mathbb{S}^{d-1})} \leq C_{11} \|\mathcal{L}(u - u^*)\|^2_{L^2(\mathbb{S}^{d-1})}.
    \end{equation}
\end{lemma}
\begin{proof}
    Consider the class 
    \[
        D^\alpha\widetilde{\mathcal{F}_\gamma}^2:= \{g^2: g \in D^\alpha \widetilde{\mathcal{F}_\gamma}\} = \{(D^\alpha u - D^\alpha u^*)^2: u \in \mathcal{F}_\gamma\}.
    \] 
    Note that for all $u \in \mathcal{F}_\gamma$ and $\vert \alpha \vert \leq s$, we have $\|(D^\alpha u - D^\alpha u^*)^2\|_\infty \leq 4M^2$ and by \eqref{H^s estimates with C_11},
    \[
        \mathrm{Var}((D^\alpha u - D^\alpha u^*)^2) \leq  \mathbb{E}_P[(D^\alpha u - D^\alpha u^*)^4] \leq 4M^2C_{11}(\mathcal{R}(u) - \mathcal{R}(u^*)) \leq 4C_{11}M^2\gamma.
    \]

    Apply \autoref{symmetrization} to $D^\alpha\widetilde{\mathcal{F}_\gamma}^2$. With probability at least $1 - \exp(-t)$, we have 
    \[ 
        \begin{aligned}
            &\sup_{u \in \mathcal{F}_\gamma} \|D^\alpha u - D^\alpha u^*\|^2_{L^2(\mathbf{x})} - \mathbb{E}_P[(D^\alpha u - D^\alpha u^*)^2] \\
            \leq{}& 3\mathrm{Rad}(D^\alpha \widetilde{\mathcal{F}_\gamma}^2) + \sqrt{\frac{8C_{11}M^2\gamma t}{n}} + \frac{16M^2t}{3n}.
        \end{aligned}
    \]

    Note that $\|D^\alpha u - D^\alpha u^*\|_\infty \leq 2M$ and the square function is locally $4M$ Lipschitz on $[-2M, 2M]$. 
    By the contraction lemma \cite[Theorem 4.12]{Ledoux1991Probability}, we have
    \[
        \mathrm{Rad}(D^\alpha \widetilde{\mathcal{F}_\gamma}^2) \leq 4M \mathrm{Rad}(D^\alpha \widetilde{\mathcal{F}_\gamma}).
    \]
    Hence, if we choose 
    \[
        \gamma \geq \max_{\vert \alpha \vert \leq s}\bigg\{ 12M\mathrm{Rad}(D^\alpha \widetilde{\mathcal{F}_\gamma}), \frac{8C_{11}M^2t}{n} ,\frac{16M^2t}{3n} \bigg\},
    \]
    then with probability at least $1 - \exp(-t)$,
    \[
        \sup_{u \in \mathcal{F}_\gamma} \|D^\alpha u - D^\alpha u^*\|^2_{L^2(\mathbf{x})} \leq \sup_{u \in \mathcal{F}_\gamma} \mathbb{E}_P[(D^\alpha u - D^\alpha u^*)^2] + \gamma + \gamma + \gamma \leq (C_{11} + 3)\gamma,
    \]
    This probability inequality holds for all $\vert \alpha \vert \leq s$. By the union bound, we have finished the proof.    
\end{proof}

\subsection{Proof of \autoref{VCDim bound}}\label{appsec: Proof of VCDim bound}
This subsection provides a proof on the VC-dimension bound for our specific CNN space. For such an endeavor, the ensuing lemma is essential.

\begin{lemma}\label{lemma: number of all sign vectors}\cite[Theorem 8.3]{Anthony1999Neural}
    Suppose $p_1, \ldots, p_m$ represent polynomials of degree not exceeding $q$ in $t \leq m$ variables. Let the sign function be defined as 
    $\mathrm{sgn}(x) := \mathbf{1}_{x > 0}$ and consider 
    \[
        K:= \vert \{ (\mathrm{sgn}(p_1(x)), \ldots, \mathrm{sgn}(p_m(x))) : x \in \mathbb{R}^t \} \vert,
    \]
    where $K$ denotes the total quantity of all sign vectors furnished by $p_1, \ldots, p_m$. Then, we have $K \leq 2(2emq/t)^t$.
\end{lemma}

Reall that in our CNN space, for $l \in [L]$, the convolution kernel $w^{(l)}$ and bias $b^{(l)}$ are the parameters of convolutional layer $l$. For $l \geq L + 1$, the weight matrices $W^{(l)}$ and bias $b^{(l)}$ are the parameters of fully connected layer $l$. Let $\mathcal{S}_l$ denote the total number of free parameters up to layer $l$ and $\theta^{(l)} \in \mathbb{R}^{\mathcal{S}_l}$ denote the concatenated free parameters vector up to layer $l$. We see that $\mathcal{S}_1 \leq \mathcal{S}_2 \leq \cdots \leq \mathcal{S}_{L + L_0 + 1} = \mathcal{S}$.

We first consider the function space $\mathcal{F}$ itself. For simplicity, let $f(x; \theta)$ denote the output of network with input $x \in \mathbb{R}^{d}$ and parameters vector $\theta \in \mathbb{R}^{\mathcal{S}}$. Thus, every $\theta \in \mathbb{R}^{\mathcal{S}}$ corresponds to a unique $f(\cdot; \theta) \in \mathcal{F}$. Assume that the VC-dimension of the network is $m$ so we can let $\{x^1, \ldots, x^m\} \subset \mathbb{R}^d$ be a shattered set, that is, 
\[
    K:= \vert \{ (\mathrm{sgn}(f(x^1; \theta)), \ldots, \mathrm{sgn}(f(x^m; \theta))) : \theta \in \mathbb{R}^{\mathcal{S}} \} \vert = 2^m.
\]
If $m \leq \mathcal{S}$, the conclusion is immediate so we assume that $m > \mathcal{S}$. We will give a bound for $K$, which will imply a bound for $m$.

To this end, we find a partition $\mathcal{P}$ of the parameter space $\mathbb{R}^{\mathcal{S}}$ such that for all $P \in \mathcal{P}$, the functions $f(x^1; \cdot), \ldots, f(x^m; \cdot)$ are polynomials on $P$. Then, 
\begin{equation}\label{equation: K}
    K \leq \sum_{P \in \mathcal{P}} \vert \{ (\mathrm{sgn}(f(x^1; \theta)), \ldots, \mathrm{sgn}(f(x^m; \theta))) : \theta \in P \} \vert,
\end{equation}
and we can bound the summand by \autoref{lemma: number of all sign vectors}. To construct such $\mathcal{P}$, we follow an inductive step to construct a sequence of partitions $\mathcal{P}_1, \mathcal{P}_2, \ldots, \mathcal{P}_{L+L_0} = \mathcal{P},$ where $\mathcal{P}_i$ is a refinement of $\mathcal{P}_{i-1}$. For $l \in [L], i \in [d_l]$ and $j \in [m]$, let 
\begin{equation}\label{input into the neuron}
    h_{l, i}(x^j; \theta) := \sum_{s=1}^{d_{l-1}} w^{(l)}_{i-s} \big(F^{(l-1)}(x^j; \theta)\big)_s  - b^{(l)}_i,
\end{equation}
i.e., $h_{l, i}(x^j; \theta)$ is the input into the $i$-th neuron in the $l$-th layer which is inactive. Hence 
\begin{equation}\label{output from the neuron}
    (F^{(l)}(x^j; \theta))_i = \sigma^{(l)}(h_{l, i}(x^j; \theta)) = \mathrm{ReLU}(h_{l, i}(x^j; \theta)),
\end{equation}
where we note that we abbreviate the dependency of $\theta$ when defining $F^{(l)}(x^j; \theta)$ in \autoref{section: The CNN Architectures} but we emphasize it in this proof.

We start with $l = 1$. Consider the concatenated vector $(\mathrm{sgn}(h_{1, i}(x^j; \cdot)))_{i \in [d_1], j \in [m]}$. Since each $h_{1, i}(x^j; \cdot)$ is a polynomial of degree 1 in $\mathcal{S}_{1} \leq \mathcal{S} < md_1$ variables, by \autoref{lemma: number of all sign vectors}, we can choose a partition $\mathcal{P}_1$ of $\mathbb{R}^{\mathcal{S}}$ such that $\vert \mathcal{P}_1 \vert \leq 2(2emd_1/\mathcal{S}_1)^{\mathcal{S}_1}$ and the concatenated vector $(\mathrm{sgn}(h_{1, i}(x^j; \cdot)))_{i \in [d_1], j \in [m]}$ is constant in each $P \in \mathcal{P}_1$. By \eqref{output from the neuron}, clearly, in each $P \in \mathcal{P}_1$, each $(F^{(1)}(x^j; \cdot))_i$ is either a polynomial of degree 1 or a zero function in $\mathcal{S}_{1}$ variables.

In the induction step, suppose that we have constructed $\mathcal{P}_1, \ldots, \mathcal{P}_{l-1}$ for an $l \in [L]$, and for $i \in [d_{l-1}], j \in [m]$, $\big(F^{(l-1)}(x^j; \cdot)\big)_i$ is a polynomial of degree at most $l - 1$ in $\mathcal{S}_{l - 1}$ variables in any $P \in \mathcal{P}_{l-1}$. By \eqref{input into the neuron} we see that for $i \in [d_{l}], j \in [m]$, $h_{l, i}(x^j; \cdot)$ is a polynomial of degree at most $l$ in $\mathcal{S}_{l} < md_l$ variables in any $P \in \mathcal{P}_{l-1}$. By \autoref{lemma: number of all sign vectors}, we can take a partition $\mathcal{P}_{P,l}$ of $P$ such that $\vert \mathcal{P}_{P,l} \vert \leq 2(2elmd_l/\mathcal{S}_l)^{\mathcal{S}_l}$ and in each $P' \in \mathcal{P}_{P,l}$, the concatenated vector $(\mathrm{sgn}(h_{l, i}(x^j; \cdot)))_{i \in [d_l], j \in [m]}$ is a fixed value vector. Define 
\[
    \mathcal{P}_l := \bigcup_{P \in \mathcal{P}_{l-1}} \mathcal{P}_{P,l}.
\]
Then by \eqref{output from the neuron}, in each $P \in \mathcal{P}_l$, each $(F^{(l)}(x^j; \cdot))_i$ is either a polynomial of degree at most $l$ or a zero function. Moreover, we have 
\[
    \vert \mathcal{P}_l \vert \leq 2(2elmd_l/\mathcal{S}_l)^{\mathcal{S}_l} \vert \mathcal{P}_{l-1} \vert.
\]

The case for $l = L+1$ needs an additional discussion. For $i \in [d_{L+1}]$ and $j \in [m]$, let
\[
    h_{L+1, i}(x^j; \theta) := \sum_{s=1}^{\lfloor d_{L}/d \rfloor } (W^{(L+1)})_{i,s} \big(\mathcal{D}\big(F^{(L)}(x^j; \theta)\big)\big)_s  - b^{(L+1)}_i.
\]
Clearly, by induction, $h_{L+1, i}(x^j; \cdot)$ is a polynomial of degree at most $L+1$ in $\mathcal{S}_{L+1} < md_{L+1}$ variables in each $P \in \mathcal{P}_{L}$. Similarly, we can take a $\mathcal{P}_{L+1}$ such that 
\[
    \vert \mathcal{P}_{L+1} \vert \leq 2(2e(L+1)md_{L+1}/\mathcal{S}_{L+1})^{\mathcal{S}_{L+1}} \vert \mathcal{P}_{L} \vert.
\] 
In each $P \in \mathcal{P}_{L+1}$, each $(F^{(L+1)}(x^j; \cdot))_i = \mathrm{ReLU}^k(h_{L+1, i}(x^j; \cdot))$ is either a polynomial of degree at most $k(L + 1)$ or a zero function. Hence we can carry on our induction to $l \geq L + 1$. We conclude that, a partition $\mathcal{P} = \mathcal{P}_{L + L_0}$ of $\mathbb{R}^\mathcal{S}$ can be constructed. In each $P \in \mathcal{P}$, $\big(F^{(L + L_0)}(x^j; \cdot)\big)_i$ is a polynomial of degree at most $k(L + 1) + L_0 - 1$ or zero function in $\mathcal{S}_{L + L_0}$ variables. Now, the $(L+L_0+1)$-th layer has a single output neuron and
\[
    f(x^j; \cdot) = \sum_{s=1}^{d_{L+L_0}} (W^{(L+L_0+1)})_{s} (F^{(L+L_0)}(x^j; \cdot))_s  - b^{(L+L_0+1)}
\]
is a polynomial of degree at most $k(L + 1) + L_0$ in $\mathcal{S}_{L + L_0 + 1} = \mathcal{S} < m$ variables in any $P \in \mathcal{P}$. By \autoref{lemma: number of all sign vectors} and \eqref{equation: K}, we conclude that
\[
    \begin{aligned}
        2^m &= K \leq 2\bigg(\frac{2e(k(L + 1) + L_0)m}{\mathcal{S}}\bigg)^{\mathcal{S}} \vert \mathcal{P} \vert \\
        &\leq 2^{L + L_0 + 1} \left(\frac{2e(k(L + 1) + L_0)m}{\mathcal{S}}\right)^{\mathcal{S}} \\
        {}&\quad \cdot \prod_{l=1}^{L + 1} \bigg(\frac{2elmd_l}{\mathcal{S}_l}\bigg)^{\mathcal{S}_l} 
        \prod_{l=L + 2}^{L+L_0} \bigg(\frac{2e(k(L + 1) + l - L - 1)md_l}{\mathcal{S}_l}\bigg)^{\mathcal{S}_l}\\
        &\leq 2^{2(L + L_0)} \bigg(4emk(L + L_0) \cdot \max\{d_{l}: l \in [L+L_0]\} \bigg)^{2\mathcal{S}(L+ L_0)}
    \end{aligned}
\]
where we use crude bounds $L + L_0 + 1 \leq 2(L + L_0), \mathcal{S}_l \leq \mathcal{S}, k(L + 1) + L_0 \leq 2k(L + L_0)$ and $d_l \leq \max\{d_{l}: l \in [L+L_0]\}$. Taking logarithms we yield that 
\[
    \begin{aligned}
        m &\leq 2(L + L_0) + 2\mathcal{S}(L+ L_0) \log_2\bigg(4emk(L + L_0) \cdot \max\{d_{l}: l \in [L+L_0]\}\bigg) \\
        &\leq 4\mathcal{S}(L + L_0)\log_2\bigg(4emk(L + L_0) \cdot \max\{d_{l}: l \in [L+L_0]\}\bigg),
    \end{aligned}
\]
which implies 
\[
    m \lesssim \mathcal{S}(L + L_0)\log \bigg(k(L + L_0) \cdot \max\{d_{l}: l \in [L+L_0]\} \bigg).
\] 

Now we turn to bound the VC-dimension of the derivatives of functions in our CNN function space. To this end, we first bound the VC-dimension of the derivative $\mathrm{sgn}(\partial \mathcal{F}/\partial x_1)$ following a similar step. For simplicity, let $f'(x ; \theta) := \partial f(x ; \theta)/\partial x_1$. In a slight abuse of notation, assume that the VC-dimension of $\mathrm{sgn}(\mathcal{F}')$ is $m > \mathcal{S}$ and let $\{x^1, \ldots, x^m\} \subset \mathbb{R}^{d}$ be a shattered set. Define 
\[
    K:= \vert \{ (\mathrm{sgn}(f'(x^1; \theta)), \ldots, \mathrm{sgn}(f'(x^m; \theta))) : \theta \in \mathbb{R}^{\mathcal{S}} \} \vert = 2^m.
\]
For $l = 1$ and $i \in [d_1]$, we have 
\begin{equation}\label{h'_1}
    h'_{1,i}(x^j; \theta) = w^{(1)}_{i - 1}.
\end{equation}
For $l = 2, \ldots, L$ and $i \in [d_l]$, we have
\begin{equation}\label{h'_l}
    h'_{l, i}(x; \theta) = \sum_{s=1}^{d_{l-1}} w^{(l)}_{i-s} \big(F^{(l-1)'}(x; \theta)\big)_s = \sum_{s=1}^{d_{l-1}} w^{(l)}_{i-s} h'_{l-1, s}(x; \theta) \mathrm{sgn}(h_{l-1, s}(x; \theta)).            
\end{equation}
For $l = L+1$ and $i \in [d_{L+1}]$, we have
\begin{equation}\label{h'_L+1}
    \begin{aligned}
        h'_{L+1, i}(x; \theta) &= \sum_{s=1}^{\lfloor d_{L}/d \rfloor } (W^{(L+1)})_{i,s} \big(\mathcal{D}\big(F^{(L)'}(x^j; \theta)\big)\big)_s \\
        &= \sum_{s=1}^{\lfloor d_{L}/d \rfloor } (W^{(L+1)})_{i,s} h'_{L, sd}(x; \theta) \mathrm{sgn}(h_{L, sd}(x; \theta)).                
    \end{aligned}
\end{equation}
For $l = L+2$ and $i \in [d_{L+2}]$, we have
\[
    \begin{aligned}
        h'_{L+2, i}(x; \theta) &= \sum_{s=1}^{d_{L+1}} (W^{(L+2)})_{i,s} \big(F^{(L+1)'}(x^j; \theta)\big)_s \\
        &= k\sum_{s=1}^{d_{L+1}} (W^{(L+1)})_{i,s} h'_{L+1, s}(x; \theta) h^{k-1}_{L+1, s}(x; \theta)\mathrm{sgn}(h_{L+1, s}(x; \theta)).
    \end{aligned}
\]
For $l \geq L+3$ and $i \in [d_l]$, we have
\[
    h'_{l, i}(x; \theta) = \sum_{s=1}^{d_{l - 1}} (W^{(l)})_{i, s} h'_{l-1, s}(x; \theta)\mathrm{sgn}(h_{l-1, s}(x; \theta)).
\]
Define $[\lfloor d_{L}/d \rfloor]d := \{d, 2d, \ldots, \lfloor d_{L}/d \rfloor d\} \subset \mathbb{Z}$. By induction, for the derivative of the network output, we can write 
\begin{equation}\label{derivative of network}
    f'(x; \theta) = k \sum_{s} \left[\left(\prod_{l=1}^{L + L_0 + 1} \theta_{l,s} \right) \cdot \left(\prod_{l=1}^{L+L_0} \mathrm{sgn}(h_{l, s_l}(x; \theta))\right) \cdot h^{k-1}_{L+1, s_{L+1}}(x; \theta)\right],
\end{equation}
where $s$ running all over $\prod_{l=1}^{L-1} [d_{l}] \times [\lfloor d_{L}/d \rfloor]d \times \prod_{l=L+1}^{L + L_0} [d_{l}] \subset \mathbb{Z}^{L+L_0}$ and $\theta_{l,s}$ is a specific parameter appears in the $l$-th layer.

Similarly, we can construct an identical partition $\mathcal{P}$ of the parameter space $\mathbb{R}^{\mathcal{S}}$ such that, in each $P \in \mathcal{P}$, $\mathrm{sgn}(h_{l, s_l}(x^{j}; \cdot))$ are constants for all $s \in \prod_{l=1}^{L-1} [d_{l}] \times [\lfloor d_{L}/d \rfloor]d \times \prod_{l=L+1}^{L + L_0} [d_{l}], l \in [L]$ and $j \in [m]$. Moreover, $h_{L+1, s_{L+1}}(x^j; \cdot)$ is a polynomial of degree at most $L + 1$. Hence, in each $P \in \mathcal{P}$, $f'(x^{j}; \cdot)$ is a polynomial of degree at most $(L + L_0 + 1) + (k - 1)(L + 1)  = k(L + 1) + L_0$ in $\mathcal{S} < m$ variables. By \autoref{lemma: number of all sign vectors} and \eqref{equation: K}, we can derive 
\[
    m \lesssim \mathcal{S}(L + L_0)\log \bigg(k(L + L_0) \cdot \max\{d_{l}: l \in [L+L_0]\} \bigg)
\] 
once more.

Notice that this bound also holds for other derivatives $\partial f(x ; \theta)/\partial x_i$ of order $1$, and they share a similar expression \eqref{derivative of network} with different $\theta_{l,s}$. With this observation, we can now consider higher order derivatives $D^\alpha, \vert \alpha \vert \geq 2$. We claim that for all $\vert \alpha \vert \geq 1$, we have 
\begin{equation}\label{higher derivative of network}
    \begin{aligned}
        &D^\alpha f(x; \theta) \\
        ={}& \bigg(\prod_{l=0}^{\vert \alpha \vert - 1}(k-l)\bigg)\sum_{s^1}\sum_{s^2}\cdots\sum_{s^{\vert \alpha \vert}}
        \left[ \left(\prod_{l=1}^{L + L_0 + 1} \theta_{l,s^1} \right) \left(\prod_{l=1}^{L+1} \theta_{l,s^2,s^1_{L+1}} \right) \right.\\
        {}&\left.\cdots \left(\prod_{l=1}^{L+1} \theta_{l,s^{\vert \alpha \vert},s^1_{L+1}} \right) \left(\prod_{l=1}^{L+L_0} \mathrm{sgn}(h_{l, s^1_l}(x; \theta))\right)\left(\prod_{l=1}^{L} \mathrm{sgn}(h_{l, s^2_l}(x; \theta))\right)\right.\\
        {}&\left. \cdots \left(\prod_{l=1}^{L} \mathrm{sgn}(h_{l, s^{\vert \alpha \vert}_l}(x; \theta))\right)
        h^{k-\vert \alpha \vert}_{L+1, s^1_{L+1}}(x; \theta) \right],
    \end{aligned}
\end{equation}
where
\[
    \begin{aligned}
        &s^1 \text{ is ranging  all over } \prod_{l=1}^{L-1} [d_{l}] \times [\lfloor d_{L}/d \rfloor]d \times \prod_{l=L+1}^{L + L_0} [d_{l}] \subset \mathbb{Z}^{L+L_0}, \\
        &s^2, \ldots ,s^{\vert \alpha \vert} \text{ is ranging all over } \prod_{l=1}^{L-1} [d_{l}] \times [\lfloor d_{L}/d \rfloor]d \subset \mathbb{Z}^{L}, 
    \end{aligned}
\]
and $\theta_{l,s^1}, \theta_{l,s^2,s^1_{L+1}}, \ldots \theta_{l,s^{\vert \alpha \vert},s^1_{L+1}}$ are some parameters appearing in the $l$-th layer and the specific locations are determined by their subscripts.

We prove this claim by induction on $\vert \alpha \vert$. For $\vert \alpha \vert = 1$, \eqref{higher derivative of network} is exactly \eqref{derivative of network}. Now assume that the claim holds for all $\vert \beta \vert \leq \vert \alpha \vert$. Without loss of generality, assume that $\alpha_1 \geq 1$ and $\alpha = \beta + (1, 0, 0, \cdots ,0)$. By induction, we calculate that 
\[
    \begin{aligned}
        &D^\alpha f(x; \theta) = \left(D^\beta f(x; \theta)\right)' \\
        ={}& \bigg(\prod_{l=0}^{\vert \beta \vert - 1}(k-l)\bigg)\sum_{s^1}\sum_{s^2}\cdots\sum_{s^{\vert \beta \vert}}
        \left[ \left(\prod_{l=1}^{L + L_0 + 1} \theta_{l,s^1} \right) \left(\prod_{l=1}^{L+1} \theta_{l,s^2,s^1_{L+1}} \right)\right.\\ 
            {}& \left. \cdots \left(\prod_{l=1}^{L+1} \theta_{l,s^{\vert \beta \vert},s^1_{L+1}} \right) \left(\prod_{l=1}^{L+L_0} \mathrm{sgn}(h_{l, s^1_l}(x; \theta))\right)\left(\prod_{l=1}^{L} \mathrm{sgn}(h_{l, s^2_l}(x; \theta))\right)\right.\\
        {}& \left. \cdots \left(\prod_{l=1}^{L} \mathrm{sgn}(h_{l, s^{\vert \beta \vert}_l}(x; \theta))\right)
        \left(h^{k-\vert \beta \vert}_{L+1, s^1_{L+1}}(x; \theta)\right)' \right],
    \end{aligned}
\]
where 
\[
    \left(h^{k-\vert \beta \vert}_{L+1, s^1_{L+1}}(x; \theta)\right)' = (k-\vert \beta \vert)h^{k-\vert \beta \vert-1}_{L+1, s^1_{L+1}}(x; \theta)h'_{L+1, s^1_{L+1}}(x; \theta).
\]
Plugging in \eqref{h'_L+1}\eqref{h'_l}\eqref{h'_1} yield the conclusion immediately.

Now with the expression \eqref{higher derivative of network}, we are ready to bound the VC-dimension of $\mathrm{sgn}(D^\alpha \mathcal{F})$. Again, in a slight abuse of notation, assume that the VC-dimension of $\mathrm{sgn}(D^\alpha \mathcal{F})$ is $m > \mathcal{S}$ and let $\{x^1, \ldots, x^m\} \subset \mathbb{R}^{d}$ be a shattered set. Define 
\[
    K:= \vert \{ (\mathrm{sgn}(D^\alpha f(x^1; \theta)), \ldots, \mathrm{sgn}(D^\alpha f(x^m; \theta))) : \theta \in \mathbb{R}^{\mathcal{S}} \} \vert = 2^m.
\]
By \eqref{higher derivative of network}, clearly, in each $P \in \mathcal{P}$, for all $j \in [m]$, $D^\alpha f(x^j; \cdot)$ is a polynomial of degree at most 
\[
    (L + L_0 + 1) + (\vert \alpha \vert - 1)(L + 1) + (k - \vert \alpha \vert)(L + 1) = L(k + 1) + L_0
\]
in $\mathcal{S} < m$ variables. Hence, we yield \[m \lesssim \mathcal{S}(L + L_0)\log \bigg(k(L + L_0) \cdot \max\{d_{l}: l \in [L+L_0]\} \bigg)\] again. The proof is then finished.

\subsection{Proof of the oracle inequality}\label{appsec: Proof of the oracle inequality}
In this subsection, we will give an estimate of the statistical error by proving the oracle inequality \eqref{oracle inequality}. To this end, we first prove a series of lemmas. Recall that $\mathrm{VC}_{\mathcal{F}}$ is given by \eqref{VC-dimension of F}, the function classes $\mathcal{F}_{\gamma}$ and $\widetilde{\mathcal{F}_\gamma}$ are defined at the beginning of \autoref{appsec: Localized complexity analysis}.

\begin{lemma}\label{one step upper bound}
    Given a dataset $\mathbf{x} = \{X_i\}_{i=1}^n$, if there exists $C_{12} > 0$ and $\gamma > 0$ satisfying
    \begin{align}
        \max_{\vert \alpha \vert \leq s}\sup_{u \in \mathcal{F}_\gamma}  \|D^\alpha u - D^\alpha u^*\|_{L^2(\mathbf{x})} &\leq C_{12} \sqrt{\gamma}, \label{condition of one step upper bound} \\
        (1/n)^2 &\leq \gamma, \notag \\
        \max_{\vert \alpha \vert \leq s} \mathrm{PDim}(D^\alpha \mathcal{F}) \vee (4eM/C_{12})^2 &\leq n, \notag
    \end{align} 
    then we have 
    \[
        \max_{\vert \alpha \vert \leq s} \mathrm{Rad}(D^\alpha\widetilde{\mathcal{F}_\gamma}, \mathbf{x}) \lesssim \sqrt{\frac{\mathrm{VC}_{\mathcal{F}}}{n} \gamma \log n},
    \]
    Moreover, for any $t> 0$, with a probability exceeding $1 - 2\exp(-t)$, there holds
    \[
        \sup_{u \in \mathcal{F}_\gamma} \mathbb{E}_P g_{u} - \mathbb{E}_{\mathbf{x}}g_{u} \lesssim 6 C_{10} \kappa \sqrt{\frac{\mathrm{VC}_{\mathcal{F}}}{n} \gamma \log n} + \sqrt{\frac{2C_{10}M\gamma t}{n}} + \frac{23C_{10}Mt}{3n}.
    \]
    Here, $M>0$ is the constant required in \autoref{Assumption for CNN} and $C_{10} > 0$ is from \autoref{contraction lemma}.
\end{lemma}

\begin{proof}
    Consider the class 
    \[
        \mathcal{G}_\gamma= \{g_u: \mathbb{S}^{d-1} \to \mathbb{R} \ \vert \ u \in \mathcal{F}_\gamma\}.
    \] 
    For all $g_u \in \mathcal{G}_\gamma, \|g_u\|_\infty \leq C_{10}M, \mathrm{Var}(g_u) \leq \mathbb{E}_P(g_u^2) \leq C_{10}M\gamma$. We can apply \autoref{symmetrization} to yield that, with probability at least $1 - 2\exp(-t)$, 
    \[
        \sup_{u \in \mathcal{F}_\gamma} \mathbb{E}_P g_{u} - \mathbb{E}_{\mathbf{x}}g_{u} \leq 6\mathrm{Rad}(\mathcal{G}_\gamma, \mathbf{x}) + \sqrt{\frac{2C_{10}M\gamma t}{n}} + \frac{23C_{10}Mt}{3n}.
    \]
    To bound the empirical Rademacher complexity term, we apply \autoref{contraction lemma}, the condition \eqref{condition of one step upper bound} and \autoref{Dudley's chaining} to yield
    \[
        \begin{aligned}
            &\mathrm{Rad}(\mathcal{G}_\gamma, \mathbf{x}) \\
            \leq{}& C_{10} \sum_{\vert \alpha \vert \leq s}\mathrm{Rad}(D^\alpha \widetilde{\mathcal{F}_\gamma}, \mathbf{x}) \\
            \leq{}& C_{10} \sum_{\vert \alpha \vert \leq s}\mathrm{Rad}(\{D^\alpha u - D^\alpha u^*: u \in \mathcal{F}, \|D^\alpha u - D^\alpha u^*\|_{L^2(\mathbf{x})} \leq C_{12}\sqrt{\gamma}\}, \mathbf{x}) \\
            \leq{}& C_{10} \sum_{\vert \alpha \vert \leq s}\inf_{0 < \beta < C_{12}\sqrt{\gamma}} \bigg\{4\beta + \frac{12}{\sqrt{n}} \int_{\beta}^{C_{12}\sqrt{\gamma}} \sqrt{\log \mathcal{N}(\delta, D^\alpha \mathcal{F}, \| \cdot \|_{L^2(\mathbf{x})})}d\delta \bigg\} \\
            \leq{}& C_{10} \sum_{\vert \alpha \vert \leq s}\inf_{0 < \beta < C_{12}\sqrt{\gamma}} \bigg\{4\beta + \frac{12}{\sqrt{n}} \int_{\beta}^{C_{12}\sqrt{\gamma}} \sqrt{\log \mathcal{N}(\delta, D^\alpha \mathcal{F}\vert_{x_1, \ldots ,x_n}, \| \cdot \|_\infty)}d\delta \bigg\}.
        \end{aligned}
    \]
    When $n \geq \max_{\vert \alpha \vert \leq s} \mathrm{PDim}(D^\alpha \mathcal{F})$, by \autoref{empirical metric entropy bound}, with 
    \[
        \beta = C_{12} \sqrt{\mathrm{PDim}(D^\alpha \mathcal{F}) \gamma/ n} \leq C_{12}\sqrt{\gamma}
    \] 
    we have 
    \[
        \begin{aligned}
            &\inf_{0 < \beta < C_{12}\sqrt{\gamma}} \bigg\{4\beta + \frac{12}{\sqrt{n}} \int_{\beta}^{C_{12}\sqrt{\gamma}} \sqrt{\log \mathcal{N}(\delta, D^\alpha \mathcal{F}\vert_{x_1, \ldots ,x_n}, \| \cdot \|_\infty)}d\delta \bigg\} \\ 
            \leq{}& \inf_{0 < \beta < C_{12}\sqrt{\gamma}} \bigg\{4\beta + \frac{12}{\sqrt{n}} \int_{\beta}^{C_{12}\sqrt{\gamma}} \sqrt{\mathrm{PDim}(D^\alpha \mathcal{F}) \log \left(\frac{2enM}{\delta \mathrm{PDim}(D^\alpha \mathcal{F})}\right)} d\delta \bigg\}\\  
            \leq{}& 16 \sqrt{\frac{\mathrm{PDim}(D^\alpha \mathcal{F})\gamma}{n} \left(\log \frac{4eM}{C_{12}\sqrt{\gamma}} + \frac{3}{2}\log n\right)}.
        \end{aligned}
    \]
    If we choose $\gamma \geq (1/n)^2$ and $n \geq (4eM/C_{12})^2$, then 
    \[
        16 \sqrt{\frac{\mathrm{PDim}(D^\alpha \mathcal{F})\gamma}{n} \left(\log \frac{4eM}{C_{12}\sqrt{\gamma}} + \frac{3}{2}\log n\right)} \leq 32 \sqrt{\frac{\mathrm{PDim}(D^\alpha \mathcal{F})\gamma}{n}\log n}.
    \]
    By \autoref{pseudo-dimension bound} and \autoref{VCDim bound}, note that 
    \[
        \mathrm{VCDim}(\mathrm{sgn}((D^\alpha \mathcal{F})_0)) \asymp \mathrm{VCDim}(\mathrm{sgn}(D^\alpha \mathcal{F}))
    \] 
    we have 
    \[
        32 \sqrt{\frac{\mathrm{PDim}(D^\alpha \mathcal{F})\gamma}{n}\log n} \leq 32 \sqrt{\frac{\mathrm{VCDim}(\mathrm{sgn}((D^\alpha \mathcal{F})_0))}{n}\gamma \log n} \lesssim \sqrt{\frac{\mathrm{VC}_{\mathcal{F}}}{n} \gamma \log n}.
    \]
    We conclude that, if $\gamma \geq (1/n)^2$ and $n \geq \max_{\vert \alpha \vert \leq s} \mathrm{PDim}(D^\alpha \mathcal{F}) \vee (4eM/C_{12})^2$, then 
    \[
        \max_{\vert \alpha \vert \leq s} \mathrm{Rad}(D^\alpha \widetilde{\mathcal{F}_\gamma}, \mathbf{x}) \lesssim \sqrt{\frac{\mathrm{VC}_{\mathcal{F}}}{n}\gamma \log n}, \quad \mathrm{Rad}(\mathcal{G}_\gamma, \mathbf{x}) \lesssim C_{10} \kappa \sqrt{\frac{\mathrm{VC}_{\mathcal{F}}}{n}\gamma \log n}
    \]
    and with probability at least $1 - 2\exp(-t)$, 
    \[
        \sup_{u \in \mathcal{F}_\gamma} \mathbb{E}_P g_{u} - \mathbb{E}_{\mathbf{x}}g_{u} \lesssim 6 C_{10} \kappa \sqrt{\frac{\mathrm{VC}_{\mathcal{F}}}{n}\gamma \log n} + \sqrt{\frac{2C_{10}M\gamma t}{n}} + \frac{23C_{10}Mt}{3n}.
    \] 
    The proof is then finished.    
\end{proof}
    
The critical radius, denoted as $\gamma_*$, is defined as the smallest non-zero fixed point that satisfies the inequality 
\[
    \gamma \geq \max_{\vert \alpha \vert \leq s}\bigg\{ 12M\mathrm{Rad}(D^\alpha \widetilde{\mathcal{F}_\gamma}), \frac{8C_{11}M^2t}{n} ,\frac{16M^2t}{3n} \bigg\},
\]
where $C_{11} > 0$ is the constant from \eqref{H^s estimates with C_11}. Therefore, $\gamma_*$ can be expressed as 
\begin{equation}\label{critical radius}
    \gamma_* := \inf\bigg\{ \gamma > 0: \gamma' \geq \max_{\vert \alpha \vert \leq s} \bigg\{ 12M\mathrm{Rad}(D^\alpha \widetilde{\mathcal{F}_{\gamma'}}), \frac{8C_{11}M^2t}{n} ,\frac{16M^2t}{3n}\bigg\}, \forall \gamma' \geq \gamma \bigg\}.
\end{equation}
    
This $\gamma_*$ is a positive value that relies on $t > 0$.
An upper bound for the critical radius can be determined, shedding light on the inherent complexity of the problem to some extent.

\begin{lemma}\label{upper bound of critical radius}
    Assume that $t \geq 1$ and 
    \[
        n \geq \max_{\vert \alpha \vert \leq s} \mathrm{PDim}(D^\alpha \mathcal{F}) \vee \frac{8e^2M^2}{C_{11} + 3}.
    \]
    Here, $C_{11} > 0$ is the constant from \eqref{H^s estimates with C_11}. Then 
    \[
        \gamma_* \lesssim 144 M^2 \frac{\mathrm{VC}_{\mathcal{F}}}{n}\log n + 48\kappa M^2\exp(-t) + \frac{16C_{11}M^2t}{n} + \frac{32M^2t}{3n}.
    \]
\end{lemma}

\begin{proof}
    By definition of $\gamma_*$, for all $\varepsilon > 0$, there exists a $\gamma \in [\gamma_* - \varepsilon, \gamma_*]$, such that 
    \[
        \begin{aligned}
            \gamma &< \max_{\vert \alpha \vert \leq s}\bigg\{ 12M\mathrm{Rad}(D^\alpha \widetilde{\mathcal{F}_\gamma}), \frac{8C_{11}M^2t}{n} ,\frac{16M^2t}{3n} \bigg\} \\
            &\leq \max_{\vert \alpha \vert \leq s}\bigg\{ 12M\mathrm{Rad}(D^\alpha \widetilde{\mathcal{F}_{\gamma_*}}), \frac{8C_{11}M^2t}{n} ,\frac{16M^2t}{3n}\bigg\}.
        \end{aligned}
    \]
    Letting $\varepsilon \to 0$ yields that 
    \[
        \gamma_* \leq \max_{\vert \alpha \vert \leq s}\bigg\{ 12M\mathrm{Rad}(D^\alpha \widetilde{\mathcal{F}_{\gamma_*}}), \frac{8C_{11}M^2t}{n} ,\frac{16M^2t}{3n}\bigg\}.
    \]
    Then we define an event 
    \[
        E = \left\{ \max_{\vert \alpha \vert \leq s} \sup_{u \in \mathcal{F}_{2\gamma_*}} \|D^\alpha u - D^\alpha u^*\|_{L^2(\mathbf{x})} \leq \sqrt{2(C_{11} + 3)\gamma_*} \right\}.
    \]
    By \autoref{upper bound of the empirical norm}, $P(E) \geq 1 - \kappa\exp(-t)$ for all $t \geq 1$. 
    Notice that $\gamma_* \geq 16M^2t/(3n) \geq (1/n)^2$, by \autoref{one step upper bound}, 
    \[
        \begin{aligned}
            \gamma_* &\leq \max_{\vert \alpha \vert \leq s}\bigg\{ 12M\mathrm{Rad}(D^\alpha \widetilde{\mathcal{F}_{\gamma_*}}), \frac{8C_{11}M^2t}{n} ,\frac{16M^2t}{3n}\bigg\}  \\
            &\lesssim 12M\max_{\vert \alpha \vert \leq s}\mathbb{E}_P(\mathbf{1}_E \cdot \mathrm{Rad}(D^\alpha \widetilde{\mathcal{F}_{\gamma_*}}, \mathbf{x})) + 24M^2(1 -P(E)) + \frac{8C_{11}M^2t}{n} + \frac{16M^2t}{3n} \\
            &\lesssim 12M \sqrt{\frac{\mathrm{VC}_{\mathcal{F}}}{n}\gamma_*\log n} + 24\kappa M^2\exp(-t) + \frac{8C_{11}M^2t}{n} + \frac{16M^2t}{3n}.
        \end{aligned}
    \]
    We solve this quadratic inequality for $\gamma^*$ and yield that 
    \[
        \gamma_* \lesssim 144 M^2\frac{\mathrm{VC}_{\mathcal{F}}}{n}\log n + 48\kappa M^2\exp(-t) + \frac{16C_{11}M^2t}{n} + \frac{32M^2t}{3n}.
    \]
    We then complete the proof of this lemma.
\end{proof}
In the final part of this subsection, we utilize a peeling technique to demonstrate the oracle inequality \autoref{oracle inequality}.
\begin{proof}[Proof of \autoref{oracle inequality}]
    Define the critical radius $\gamma_*$ as \eqref{critical radius} with some $t' \geq 1$ to be specified. Considering $\gamma > \gamma_* \vee (\log n / n)$, its precise value will also be specified subsequently. 
    
    Recall that we have defined the localized class
    \[
        \mathcal{F}_\gamma:= \{u \in \mathcal{F}: \mathbb{E}_P(g_u(X)) = \mathcal{R}(u) - \mathcal{R}(u^*) \leq \gamma\}.
    \]
    We can partition $\mathcal{F}$ into shells without intersection 
    \[
        \mathcal{F} = \mathcal{F}_\gamma \cup (\mathcal{F}_{2\gamma} \backslash \mathcal{F}_\gamma) \cup \cdots \cup (\mathcal{F}_{2^l\gamma} \backslash \mathcal{F}_{2^{l-1}\gamma}),
    \]
    where $l \leq \log_2(C_{10}M / \gamma) \leq \log_2(C_{10}M n/ \log n)$ and $C_{10}$ is from \autoref{contraction lemma}. 
        
    Assume that for some $j \leq l, u_n \in \mathcal{F}_{2^j \gamma}$. Notice that $2^j \gamma > \gamma_*$ satisfies the condition of \autoref{upper bound of the empirical norm}, with probability at least $1 - \kappa\exp(-t')$, there holds
    \[
        \max_{\vert \alpha \vert \leq s} \|D^\alpha u_n - D^\alpha u^*\|_{L^2(\mathbf{x})} \leq \max_{\vert \alpha \vert \leq s}\sup_{u \in \mathcal{F}_{2^j \gamma}}\|D^\alpha u - D^\alpha u^*\|_{L^2(\mathbf{x})} \leq \sqrt{2^j(C_{11} + 3) \gamma},
    \] 
    where $\kappa := \# \{\alpha: \vert \alpha \vert \leq s, a_\alpha (\cdot) \neq 0\}$ to represent the number of non-zero terms in $\mathcal{L}$ satisfying \autoref{Assumption for PDE}.
                
    From \autoref{one step upper bound}, with probability at least $1 - (\kappa + 2)\exp(-t')$ we obtain 
    \[
        \mathbb{E}_P g_{u_n} - \mathbb{E}_{\mathbf{x}}g_{u_n} \lesssim 6C_{10}\kappa\sqrt{\frac{\mathrm{VC}_{\mathcal{F}}}{n} 2^j \gamma \log n} + \sqrt{\frac{2C_{10}M2^j \gamma t'}{n}} + \frac{23C_{10}Mt'}{3n}.
    \]
    Using \autoref{Bernstein's bound}, with probability at least $1 - (\kappa + 3)\exp(-t')$, there holds 
    \[
        \begin{aligned}
            &\mathbb{E}_P g_{u_n} \\
            \lesssim{} &\mathbb{E}_{\mathbf{x}}g_{u_n} + 6C_{10}\kappa\sqrt{\frac{\mathrm{VC}_{\mathcal{F}}}{n} 2^j \gamma \log n} + \sqrt{\frac{2C_{10}M2^j \gamma t'}{n}} + \frac{23C_{10}Mt'}{3n} \\
            \leq{}& \mathbb{E}_{\mathbf{x}}g_{u_\mathcal{F}} - \mathbb{E}_{P}g_{u_\mathcal{F}} + \mathbb{E}_{P}g_{u_\mathcal{F}} + 6C_{10}\kappa\sqrt{\frac{\mathrm{VC}_{\mathcal{F}}}{n} 2^j \gamma \log n} + \sqrt{\frac{2C_{10}M2^j \gamma t'}{n}} + \frac{23C_{10}Mt'}{3n} \\
            \leq{}& 2\mathbb{E}_{P}g_{u_\mathcal{F}} + 6C_{10}\kappa\sqrt{\frac{\mathrm{VC}_{\mathcal{F}}}{n} 2^j \gamma \log n} + \sqrt{\frac{2C_{10}M2^j \gamma t'}{n}} + \frac{23C_{10}Mt'}{3n} + \frac{7C_0t'}{6n},
        \end{aligned}
    \]
    where $C_0$ is the constant from \autoref{Bernstein's bound}.
                
    Assume that we have chosen a suitable $\gamma$ such that  
    \begin{equation}\label{chosen a suitable gamma}
        2\mathbb{E}_{P}g_{u_\mathcal{F}} + 6C_{10}\kappa\sqrt{\frac{\mathrm{VC}_{\mathcal{F}}}{n} 2^j \gamma \log n} + \sqrt{\frac{2C_{10}M2^j \gamma t'}{n}} + \frac{23C_{10}Mt'}{3n} + \frac{7C_0t'}{6n} \leq 2^{j-1} \gamma.
    \end{equation}
    Then with probability at least $1 - (\kappa + 3)\exp(-t')$, we have $u_n \in \mathcal{F}_{2^{j-1} \gamma}$.
    We continue this process with a shell-by-shell argument and conclude that, with probability at least $1 - l(\kappa + 3)\exp(-t')$, there holds $u_n \in \mathcal{F}_\gamma.$
    Now we choose a suitable $\gamma$.  The selected $\gamma$ must satisfy \eqref{chosen a suitable gamma} for all $1 \leq j \leq l$ and 
    \[
        \gamma_* \vee \frac{\log n}{n} \leq \gamma.
    \]
    \eqref{chosen a suitable gamma} can be ensured for all $1 \leq j \leq l$ if 
    \[
        2\mathbb{E}_{P}g_{u_\mathcal{F}} + \frac{23C_{10}Mt'}{3n} + \frac{7C_0t'}{6n} \leq 2^{j-2} \gamma, \quad \forall 1 \leq j \leq l,
    \]
    and
    \[
        6C_{10}\kappa\sqrt{\frac{\mathrm{VC}_{\mathcal{F}}}{n} 2^j \gamma \log n} + \sqrt{\frac{2C_{10}M2^j \gamma t'}{n}} \leq 2^{j-2} \gamma, \quad \forall 1 \leq j \leq l.
    \]
    Hence, it is sufficient to choose 
    \[
        \gamma =  1152C_{10}^2\kappa^2\frac{\mathrm{VC}_{\mathcal{F}}}{n} \log n + \frac{64C_{10}Mt'}{n} + 4\mathbb{E}_{P}g_{u_\mathcal{F}} + \frac{46C_{10}Mt'}{3n} + \frac{7C_0t'}{3n} + \frac{\log n}{n} + \gamma_*.
    \]
    Now we choose $t' = t + \log ((\kappa + 3)\log_2(C_{10}M n/ \log n)) \geq 1$, with probability at least $1 - l(\kappa + 3)\exp(-t') \geq 1 - \exp(-t)$, we have $u_n \in \mathcal{F}_\gamma$ and by \autoref{upper bound of critical radius},
    \[
        \begin{aligned}
            &\mathcal{R}(u_n) - \mathcal{R}(u^*) \\
            \leq{}& 1152C_{10}^2\kappa^2\frac{\mathrm{VC}_{\mathcal{F}}}{n} \log n + \frac{64C_{10}Mt'}{n} + 4\mathbb{E}_{P}g_{u_\mathcal{F}} + \frac{46C_{10}Mt'}{3n} + \frac{7C_0t'}{3n} + \frac{\log n}{n} + \gamma_* \\
            \lesssim{}& 1152C_{10}^2\kappa^2\frac{\mathrm{VC}_{\mathcal{F}}}{n} \log n +  144M^2C^2 \frac{\mathrm{VC}_{\mathcal{F}}}{n}\log n + 48\kappa M^2\exp(-t') + 4\mathbb{E}_{P}g_{u_\mathcal{F}} \\
            & + \frac{7C_0t' + 238C_{10}Mt' + 48C_{11}M^2t' + 32M^2t'}{3n} + \frac{\log n}{n} \\
            \lesssim{}& \mathbb{E}_{P}g_{u_\mathcal{F}} + \frac{\mathrm{VC}_{\mathcal{F}}}{n} \log n + \frac{t}{n} + \frac{\exp(-t)}{\log(n / \log n)}.
        \end{aligned}
    \]
    Thus we complete the proof.
\end{proof}



\bibliographystyle{plain}
\bibliography{references.bib}

\begin{thebibliography}{10}

\bibitem{Agmon1959Estimates}
S.~Agmon, A.~Douglis, and L.~Nirenberg.
\newblock Estimates near the boundary for solutions of elliptic partial
  differential equations satisfying general boundary conditions. {I}.
\newblock {\em Comm. Pure Appl. Math.}, 12(4):623--727, 1959.

\bibitem{Aistleitner2012Point}
C.~Aistleitner, J.~S. Brauchart, and J.~Dick.
\newblock Point sets on the sphere {$\mathbb{S}^2$} with small spherical cap
  discrepancy.
\newblock {\em Discrete Comput. Geom.}, 48(4):990--1024, 2012.

\bibitem{Anthony1999Neural}
Martin Anthony and Peter Bartlett.
\newblock {\em Neural Network Learning: Theoretical Foundations}.
\newblock Cambridge University Press, Cambridge, 1999.

\bibitem{Bartlett2005Local}
Peter~L. Bartlett, Olivier Bousquet, and Shahar Mendelson.
\newblock Local {Rademacher} complexities.
\newblock {\em Ann. Statist.}, 33(4):1497--1537, 2005.

\bibitem{Bastek2023PhysicsInformed}
Jan-Hendrik Bastek and Dennis~M. Kochmann.
\newblock Physics-informed neural networks for shell structures.
\newblock {\em Eur. J. Mech. A Solids}, 97:1--16, 2023.

\bibitem{Chen2020comparison}
Jingrun Chen, Rui Du, and Keke Wu.
\newblock A comparison study of deep {Galerkin} method and deep {Ritz} method
  for elliptic problems with different boundary conditions.
\newblock {\em Commun. Math. Res.}, 36(3):354--376, 2020.

\bibitem{Chen2019Efficient}
Minshuo Chen, Haoming Jiang, Wenjing Liao, and Tuo Zhao.
\newblock Efficient approximation of deep {ReLU} networks for functions on low
  dimensional manifolds.
\newblock In {\em Advances in Neural Information Processing Systems},
  volume~32, pages 8174--8184., 2019.

\bibitem{Chen2022Nonparametric}
Minshuo Chen, Haoming Jiang, Wenjing Liao, and Tuo Zhao.
\newblock Nonparametric regression on low-dimensional manifolds using deep
  {ReLU} networks: Function approximation and statistical recovery.
\newblock {\em Inf. Inference}, 11(4):1203--1253, 2022.

\bibitem{Dai2013Approximation}
Feng Dai and Yuan Xu.
\newblock {\em Approximation Theory and Harmonic Analysis on Spheres and
  Balls}.
\newblock Springer New York, New York, 2013.

\bibitem{Dung2021Deep}
Dinh D{\~u}ng and Van~Kien Nguyen.
\newblock Deep {ReLU} neural networks in high-dimensional approximation.
\newblock {\em Neural Netw.}, 142:619--635, 2021.

\bibitem{E2018Deep}
Weinan E and Bing Yu.
\newblock The deep {Ritz} method: A deep learning-based numerical algorithm for
  solving variational problems.
\newblock {\em Commun. Math. Stat.}, 6(1):1--12, 2018.

\bibitem{Fang2020PhysicsInformed}
Zhiwei Fang and Justin Zhan.
\newblock A physics-informed neural network framework for {PDEs} on {3D}
  surfaces: Time independent problems.
\newblock {\em IEEE Access}, 8:26328--26335, 2020.

\bibitem{Fang2020Theory}
Zhiying Fang, Han Feng, Shuo Huang, and Ding-Xuan Zhou.
\newblock Theory of deep convolutional neural networks {II}: Spherical
  analysis.
\newblock {\em Neural Netw.}, 131:154--162, 2020.

\bibitem{Feng2023Generalization}
Han Feng, Shuo Huang, and Ding-Xuan Zhou.
\newblock Generalization analysis of {CNNs} for classification on spheres.
\newblock {\em IEEE Trans. Neural Netw. Learn. Syst.}, 34(9):6200--6213, 2023.

\bibitem{Hamm2021Adaptive}
Thomas Hamm and Ingo Steinwart.
\newblock Adaptive learning rates for support vector machines working on data
  with low intrinsic dimension.
\newblock {\em Ann. Statist.}, 49(6):3153--3180, 2021.

\bibitem{Harvey2017Nearlytight}
Nick Harvey, Christopher Liaw, and Abbas Mehrabian.
\newblock Nearly-tight {VC}-dimension bounds for piecewise linear neural
  networks.
\newblock In {\em Proceedings of Machine Learning Research}, volume~65, pages
  1064--1068., 2017.

\bibitem{Hendrycks2016Gaussian}
Dan Hendrycks and Kevin Gimpel.
\newblock Gaussian error linear units ({GELUs}).
\newblock {\em arXiv preprint}, 2016.

\bibitem{Jiao2022rate}
Yuling Jiao, Yanming Lai, Dingwei Li, Xiliang Lu, Fengru Wang, Yang Wang, and
  Jerry~Zhijian Yang.
\newblock A rate of convergence of physics informed neural networks for the
  linear second order elliptic {PDEs}.
\newblock {\em Commun. Comput. Phys.}, 31(4):1272--1295, 2022.

\bibitem{Jiao2023Deep}
Yuling Jiao, Guohao Shen, Yuanyuan Lin, and Jian Huang.
\newblock Deep nonparametric regression on approximate manifolds: Nonasymptotic
  error bounds with polynomial prefactors.
\newblock {\em Ann. Statist.}, 51(2):691--716, 2023.

\bibitem{Johnstone1998Oracle}
Iain~M. Johnstone.
\newblock Oracle inequalities and nonparametric function estimation.
\newblock {\em Doc. Math.}, III:267--278, 1998.

\bibitem{Karniadakis2021Physicsinformed}
George~Em Karniadakis, Ioannis~G. Kevrekidis, Lu~Lu, Paris Perdikaris, Sifan
  Wang, and Liu Yang.
\newblock Physics-informed machine learning.
\newblock {\em Nat. Rev. Phys.}, 3(6):422--440, 2021.

\bibitem{Kingma2015Adam}
Diederik~P. Kingma and Jimmy Ba.
\newblock Adam: A method for stochastic optimization.
\newblock In {\em International Conference on Learning Representations.}, 2015.

\bibitem{Kiranyaz20211D}
Serkan Kiranyaz, Onur Avci, Osama Abdeljaber, Turker Ince, Moncef Gabbouj, and
  Daniel~J. Inman.
\newblock {1D} convolutional neural networks and applications: A survey.
\newblock {\em Mech. Syst. Signal Process.}, 151:107398, 2021.

\bibitem{Koltchinskii2002Empirical}
V.~Koltchinskii and D.~Panchenko.
\newblock Empirical margin distributions and bounding the generalization error
  of combined classifiers.
\newblock {\em Ann. Statist.}, 30(1):1--50, 2002.

\bibitem{Koltchinskii2006Local}
Vladimir Koltchinskii.
\newblock Local {Rademacher} complexities and oracle inequalities in risk
  minimization.
\newblock {\em Ann. Statist.}, 34(6):2593--2656, 2006.

\bibitem{Koltchinskii2011Oracle}
Vladimir Koltchinskii.
\newblock {\em Oracle Inequalities in Empirical Risk Minimization and Sparse
  Recovery Problems}, volume 2033.
\newblock Springer, Berlin, Heidelberg, 2011.

\bibitem{Koltchinskii2000Rademacher}
Vladimir Koltchinskii and Dmitriy Panchenko.
\newblock {Rademacher} processes and bounding the risk of function learning.
\newblock In {\em High Dimensional Probability II}, pages 443--457., 2000.

\bibitem{Kovachki2023Neural}
Nikola Kovachki, Zongyi Li, Burigede Liu, Kamyar Azizzadenesheli, Kaushik
  Bhattacharya, Andrew Stuart, and Anima Anandkumar.
\newblock {Neural Operator}: Learning maps between function spaces with
  applications to {PDEs}.
\newblock {\em J. Mach. Learn. Res.}, 24(89):1--97., 2023.

\bibitem{Lawson1990Spin}
H.~Blaine Lawson and Marie-Louise Michelsohn.
\newblock {\em Spin Geometry (PMS-38)}, volume~38.
\newblock Princeton University Press, Princeton, 1990.

\bibitem{Ledoux1991Probability}
Michel Ledoux and Michel Talagrand.
\newblock {\em Probability in {Banach} Spaces}.
\newblock Springer, Berlin, Heidelberg, 1991.

\bibitem{Lei2024Pairwise}
Guanhang Lei and Lei Shi.
\newblock Pairwise ranking with {Gaussian} kernel.
\newblock {\em Adv. Comput. Math.}, 50:1–56, 2024.

\bibitem{Li2020Neural}
Zongyi Li, Nikola Kovachki, Kamyar Azizzadenesheli, Burigede Liu, Kaushik
  Bhattacharya, Andrew Stuart, and Anima Anandkumar.
\newblock {Neural Operator}: Graph kernel network for partial differential
  equations.
\newblock In {\em ICLR 2020 Workshop on Integration of Deep Neural Models and
  Differential Equations}, 2020.

\bibitem{Li2021Fourier}
Zongyi Li, Nikola~Borislavov Kovachki, Kamyar Azizzadenesheli, Burigede Liu,
  Kaushik Bhattacharya, Andrew Stuart, and Anima Anandkumar.
\newblock Fourier neural operator for parametric partial differential
  equations.
\newblock In {\em International Conference on Learning Representations.}, 2021.

\bibitem{Liu2021Besov}
Hao Liu, Minshuo Chen, Tuo Zhao, and Wenjing Liao.
\newblock Besov function approximation and binary classification on
  low-dimensional manifolds using convolutional residual networks.
\newblock In {\em Proceedings of Machine Learning Research}, volume 139, pages
  6770--6780., 2021.

\bibitem{Lu2021Learning}
Lu~Lu, Pengzhan Jin, Guofei Pang, Zhongqiang Zhang, and George~Em Karniadakis.
\newblock Learning nonlinear operators via {DeepONet} based on the universal
  approximation theorem of operators.
\newblock {\em Nat. Mach. Intell.}, 3(3):218--229, 2021.

\bibitem{Lu2022Machine}
Yiping Lu, Haoxuan Chen, Jianfeng Lu, Lexing Ying, and Jose Blanchet.
\newblock Machine learning for elliptic {PDEs}: Fast rate generalization bound,
  neural scaling law and minimax optimality.
\newblock In {\em International Conference on Learning Representations.}, 2022.

\bibitem{Fu2011Manifold}
Yunqian Ma and Yun Fu.
\newblock {\em Manifold Learning Theory and Applications (1st ed.)}.
\newblock CRC Press, Boca Raton., 2011.

\bibitem{Mao2021Theory}
Tong Mao, Zhongjie Shi, and Ding-Xuan Zhou.
\newblock Theory of deep convolutional neural networks {III}: Approximating
  radial functions.
\newblock {\em Neural Netw.}, 144:778--790, 2021.

\bibitem{Mao2022Approximation}
Tong Mao and Ding-Xuan Zhou.
\newblock Approximation of functions from {Korobov} spaces by deep
  convolutional neural networks.
\newblock {\em Adv. Comput. Math.}, 48(6):1--26, 2022.

\bibitem{Mendelson2003Few}
Shahar Mendelson.
\newblock A few notes on statistical learning theory.
\newblock In {\em Advanced Lectures on Machine Learning: Machine Learning
  Summer School 2002 Canberra, Australia, February 11--22, 2002 Revised
  Lectures}, pages 1--40. Springer, 2003.

\bibitem{Montanelli2021Deep}
Hadrien Montanelli.
\newblock Deep {ReLU} networks overcome the curse of dimensionality for
  generalized bandlimited functions.
\newblock {\em J. Comput. Math.}, 39(6):801--815, 2021.

\bibitem{Montanelli2019New}
Hadrien Montanelli and Qiang Du.
\newblock New error bounds for deep {ReLU} networks using sparse grids.
\newblock {\em SIAM J. Math. Data Sci.}, 1(1):78--92, 2019.

\bibitem{Raissi2019Physicsinformed}
M.~Raissi, P.~Perdikaris, and G.~E. Karniadakis.
\newblock Physics-informed neural networks: A deep learning framework for
  solving forward and inverse problems involving nonlinear partial differential
  equations.
\newblock {\em J. Comput. Phys.}, 378:686--707, 2019.

\bibitem{SahliCostabal2024Delta}
Francisco Sahli~Costabal, Simone Pezzuto, and Paris Perdikaris.
\newblock {$\Delta$-PINNs}: Physics-informed neural networks on complex
  geometries.
\newblock {\em Eng. Appl. Artif. Intell.}, 127:107324, 2024.

\bibitem{Schmidt-Hieber2019Deep}
Johannes {Schmidt-Hieber}.
\newblock Deep {ReLU} network approximation of functions on a manifold.
\newblock {\em arXiv preprint}, 2019.

\bibitem{Schumaker2007Spline}
Larry Schumaker.
\newblock {\em Spline Functions: Basic Theory}.
\newblock Cambridge University Press, Cambridge, 3rd edition, 2007.

\bibitem{Sirignano2018DGM}
Justin Sirignano and Konstantinos Spiliopoulos.
\newblock {DGM}: A deep learning algorithm for solving partial differential
  equations.
\newblock {\em J. Comput. Phys.}, 375:1339--1364, 2018.

\bibitem{Suzuki2018Adaptivity}
Taiji Suzuki.
\newblock Adaptivity of deep {ReLU} network for learning in {Besov} and mixed
  smooth {Besov} spaces: Optimal rate and curse of dimensionality.
\newblock In {\em International Conference on Learning Representations.}, 2018.

\bibitem{Suzuki2021Deep}
Taiji Suzuki and Atsushi Nitanda.
\newblock Deep learning is adaptive to intrinsic dimensionality of model
  smoothness in anisotropic {Besov} space.
\newblock In {\em Advances in Neural Information Processing Systems},
  volume~34, pages 3609--3621., 2021.

\bibitem{Tang2021Physicsinformed}
Zhuochao Tang and Zhuojia Fu.
\newblock Physics-informed neural networks for elliptic partial differential
  equations on {3D} manifolds.
\newblock {\em arXiv preprint}, 2021.

\bibitem{Taylor2011Partial}
Michael~E. Taylor.
\newblock {\em Partial Differential Equations}.
\newblock Springer, New York, 2nd edition, 2011.

\bibitem{Warner1983Foundations}
Frank~W. Warner.
\newblock {\em Foundations of Differentiable Manifolds and {Lie} Groups},
  volume~94.
\newblock Springer New York, New York, 1983.

\bibitem{Yang2016Bayesian}
Yun Yang and David~B. Dunson.
\newblock Bayesian manifold regression.
\newblock {\em Ann. Statist.}, 44(2):876--905, 2016.

\bibitem{Yang2024nonparametric}
Yunfei Yang and Ding-Xuan Zhou.
\newblock Nonparametric regression using over-parameterized shallow {ReLU}
  neural networks.
\newblock {\em J. Mach. Learn. Res.}, 25:1--35, 2024.

\bibitem{Yang2024optimal}
Yunfei Yang and Ding-Xuan Zhou.
\newblock Optimal rates of approximation by shallow {ReLU$^k$} neural networks
  and applications to nonparametric regression.
\newblock {\em Constr. Approx.}, 2024.

\bibitem{Ye2008Learning}
Gui-Bo Ye and Ding-Xuan Zhou.
\newblock Learning and approximation by {Gaussians} on {Riemannian} manifolds.
\newblock {\em Adv. Comput. Math.}, 29(3):291--310, 2008.

\bibitem{Ye2009SVM}
Gui-Bo Ye and Ding-Xuan Zhou.
\newblock {SVM} learning and {$L_p$} approximation by {Gaussians} on
  {Riemannian} manifolds.
\newblock {\em Anal. Appl.}, 07(03):309--339, 2009.

\bibitem{Zelig2023Numerical}
Yuval Zelig and Shai Dekel.
\newblock Numerical methods for {PDEs} over manifolds using spectral physics
  informed neural networks.
\newblock {\em arXiv preprint}, 2023.

\bibitem{Zhou2020Theory}
Ding-Xuan Zhou.
\newblock Theory of deep convolutional neural networks: Downsampling.
\newblock {\em Neural Netw.}, 124:319--327, 2020.

\bibitem{Zhou2020Universality}
Ding-Xuan Zhou.
\newblock Universality of deep convolutional neural networks.
\newblock {\em Appl. Comput. Harmon. Anal.}, 48(2):787--794, 2020.

\bibitem{Zhou2024learning}
Tian-Yi Zhou and Xiaoming Huo.
\newblock Learning ability of interpolating deep convolutional neural networks.
\newblock {\em Appl. Comput. Harmon. Anal.}, 68:101582, 2024.

\end{thebibliography}
\end{document}